%% file: Main.tex
\documentclass[11pt]{amsart}

\usepackage{latexsym}
\usepackage{float}
\usepackage{amssymb}
\usepackage{amsmath}
\usepackage{amsthm}
\usepackage{mathrsfs}
\usepackage{euscript}
\usepackage[cmtip,all]{xy}
\newcommand{\vsp}{\vskip 0.2in}

\newlength{\whitecirclewidth}
\newlength{\whitecircleheight}
\newlength{\blackcirclewidth}
\newlength{\blackcircleheight}
\newlength{\corssandcircleheight}
\newlength{\corssandcirclewidth}
\newlength{\nodesize}
\newcommand{\whitecirclesymbol}{\makebox[\whitecirclewidth]{\Large$\circ$}}
\newcommand{\blackcirclesymbol}{\makebox[\blackcirclewidth]{\Large$\bullet$}}
\newcommand{\crossandcirclesymbol}{\makebox[\corssandcirclewidth]{\large$\otimes$}}

\newcommand{\wnode}{*-<\nodesize>{\whitecirclesymbol}}

\newcommand{\single}{\ar@{-}}
\newcommand{\rdouble}{\ar@2{->}}
\newcommand{\ldouble}{\ar@2{<-}}
\newcommand{\rtriple}{\ar@3{->}}
\newcommand{\ltriple}{\ar@3{<-}}

\newcommand{\abovewnode}[1]%
{
*-<\nodesize>{\raisebox{0pt}[\whitecircleheight][0pt]{$\overset{\rlap{$\displaystyle
#1$}}{\whitecirclesymbol}$}}}

\newcommand{\belowwnode}[1]%
{
*-<\nodesize>{\raisebox{1pt}[\whitecircleheight][0pt]{$\underset{\rlap{$\displaystyle
#1$}}{\whitecirclesymbol}$}}}

\newcommand{\abovebnode}[1]%
{
*-<\nodesize>{\raisebox{0pt}[\whitecircleheight][0pt]{$\overset{\rlap{$\displaystyle
#1$}}{\blackcirclesymbol}$}}}

\newcommand{\belowbnode}[1]%
{
*-<\nodesize>{\raisebox{0pt}[\whitecircleheight][0pt]{$\underset{\rlap{$\displaystyle
#1$}}{\blackcirclesymbol}$}}}

\newcommand{\abovecnode}[1]%
{
*-<\nodesize>{\raisebox{0pt}[\whitecircleheight][0pt]{$\overset{\rlap{$\displaystyle
#1$}}{\crossandcirclesymbol}$}}}

\newcommand{\belowcnode}[1]%
{
*-<\nodesize>{\raisebox{1pt}[\whitecircleheight][0pt]{$\underset{\rlap{$\displaystyle
#1$}}{\crossandcirclesymbol}$}}}

\newcommand{\gb}{\beta}
\newcommand{\ga}{\alpha}
\newcommand{\gk}{\kappa}

\newcommand{\gl}{\lambda}

\newcommand{\gd}{\delta}
\newcommand{\gD}{\Delta}
\renewcommand{\ge}{\epsilon}

\newcommand{\gt}{\theta}

\newcommand{\gs}{\sigma}

\newcommand{\fb}{{\mathfrak b}}
\newcommand{\fg}{{\mathfrak g}}
\newcommand{\fh}{{\mathfrak h}}

\newcommand{\fl}{{\mathfrak l}}

\newcommand{\fn}{{\mathfrak n}}

\newcommand{\fq}{{\mathfrak q}}

\newcommand{\fu}{{\mathfrak u}}

\newcommand{\fz}{{\mathfrak z}}
\newcommand{\f}{\mathfrak}

\newcommand{\eV}{\EuScript{V}}

\newcommand{\eX}{\EuScript{X}}

\newcommand{\cga}{\alpha^\vee}

\newcommand{\cge}{\epsilon^\vee}
\newcommand{\cmu}{\mu^\vee}

\newcommand{\nbar}{\bar{n}}                 

\newcommand{\R}{\mathbb{R}}          
\newcommand{\C}{\mathbb{C}}          
           
\newcommand{\Z}{\mathbb{Z}}

\newtheorem{Thm}[equation]{Theorem}
\newtheorem{Lem}[equation]{Lemma}
\newtheorem{Cor}[equation]{Corollary}
\newtheorem{Prop}[equation]{Proposition}

\newtheorem{Example}[equation]{Example}

\newtheorem{Def}[equation]{Definition}

\newtheorem{Rem}[equation]{Remark}

\numberwithin{equation}{section}
\newcommand{\be}{\begin{equation}}
\newcommand{\beu}{\begin{equation*}}

\newcommand{\acts}{ {\raisebox{1pt} {$\scriptstyle \bullet$} } }
\newcommand{\ad}{\text{ad}}
\newcommand{\Ad}{\text{Ad}}
\newcommand{\Cal}{\mathcal}

\newcommand{\Hom}{\text{Hom}}

\newcommand{\IP}[2]{\langle#1 , #2\rangle}     

\newcommand{\flg}{\frak{l}_{\gamma}}
\newcommand{\flng}{\frak{l}_{n\gamma}}
\newcommand{\Lg}{L_\gamma}
\newcommand{\Lng}{L_{n\gamma}}
\newcommand{\cfn}{\frak{z}(\frak{n})}
\newcommand{\xig}{\xi_\gamma}
\newcommand{\xing}{\xi_{n\gamma}}
\newcommand{\geg}{\epsilon_\gamma}
\newcommand{\geng}{\epsilon_{n\gamma}}
\newcommand{\Ht}{\text{ht}}
\newcommand{\cxig}{\xig^\vee}

\newcommand{\ttau}{\tilde{\tau}}
\newcommand{\Sym}{\text{Sym}}

\newcommand{\btau}{\bar{\tau}}

\newcommand{\spi}{\pi_s}

\newcommand{\vep}{\varepsilon}

\begin{document}

\bibliographystyle{amsplain}

\baselineskip=16pt

\title[Conformally Invariant Systems of Differential Operators]
{Special Values for Conformally Invariant Systems 
Associated to Maximal Parabolics of Quasi-Heisenberg Type}

\author{Toshihisa Kubo}
\address{Graduate School of Mathematical Sciences, 
The University of Tokyo,
 3-8-1 Komaba, Meguro-ku, Tokyo 153-8914, Japan}
\email{toskubo@ms.u-tokyo.ac.jp}


\maketitle

\begin{abstract}
In this paper we construct conformally invariant systems of
first order and second order differential operators
associated to a homogeneous line bundle $\Cal{L}_{s} \to G_0/Q_0$ 
with $Q_0$ a maximal parabolic subgroup of quasi-Heisenberg type.
This generalizes the results by Barchini, Kable, and Zierau.
To do so we use techniques different from ones used by them.
\end{abstract}


\input{SS_Intro}

\input{SS_CIS}

\input{SS_Omega}

\input{SS_TwoStep}

\input{SS_O1}
\input{SS_Special}
\input{SS_O2}

\appendix
\input{SS_Apx}


\bibliographystyle{amsplain}
\bibliography{main}

\end{document}

%% file: SS_Intro.tex
 \section{Introduction}\label{chap:intro}


The main work of this paper concerns systems of differential operators
that are equivariant under an action of a Lie algebra.
We call such systems \emph{conformally invariant}.
To explain the meaning of the equivariance condition, 
suppose that $\eV \to M$ is a vector bundle over a smooth manifold $M$
and $\fg$ is a Lie algebra of first-order differential operators that act on 
sections of $\eV$. A linearly independent list $D_1, \ldots, D_n$ of linear
differential operators on sections of $\eV$ is called a 
\emph{conformally invariant system} 
if, for each $X \in \fg$, there are smooth functions $C_{ij}^X(m)$ on $M$
so that, for all $1\leq i \leq n$, and sections $f$ of $\eV$, we have
\begin{equation}\label{Eqn:CIS1}
\big( [X, D_i]\acts f \big)(m)= \sum^n_{j=1}C_{ji}^X(m)(D_j \acts f)(m),
\end{equation}
where $[X, D_j] = XD_j - D_j X$. 
(See Definition \ref{Def1.1.3} for the precise definition.)
Here, the dot $\acts$ denotes the action of
differential operators on smooth functions.


An important consequence of the definition (\ref{Eqn:CIS1}) is that
the common kernel of the operators in the conformally invariant system $D_1, \ldots, D_n$
is invariant under a Lie algebra action. The representation theoretic
question of understanding the common kernel as a $\fg$-module 
is an open question (except for a small number of very special examples).

The notion of conformally invariant systems generalizes 
that of quasi-invariant differential operators introduced by Kostant in \cite{Kostant75}
and is related to work of Huang (\cite{Huang93}).
It is also compatible with the definition given by Ehrenpreis in \cite{Ehrenpreis88}.
Conformally invariant systems are explicitly or implicitly presented
in the work of 
Davidson-Enright-Stanke (\cite{DES91}),
Kobayashi-\O rsted (\cite{KO03}),
Wallach (\cite{Wallach79}), among others.
They are also related to the project of Dobrev on 
constructing intertwining differential operators. (See for example
\cite{Dobrev-a12} and \cite{Dobrev88}.)
Much of the published work is for the case that $M=G/Q$ with 
$Q=LN$, $N$ abelian.
The systematic study of conformally invariant systems started with the work of 
Barchini-Kable-Zierau in \cite{BKZ08} and \cite{BKZ09} and is continued in
\cite{Kable}, \cite{Kable11}, \cite{Kable12A}, \cite{Kable12C}, \cite{Kable12B}, 
\cite{KuboThesis2}, and \cite{Kubo11}.

Although the theory of conformally invariant systems can be viewed as 
a geometric-analytic theory, it is closely related to algebraic objects
such as generalized Verma modules. It has been shown in \cite{BKZ09}
that a conformally invariant system yields a homomorphism between
certain generalized Verma modules. 
The classification of non-standard homomorphisms 
between generalized Verma modules is an open problem.
In \cite{KuboThesis2}, it is determined whether or not
the homomorphisms between the generalized Verma modules
that arise from certain conformally invariant systems are standard.

The main goal of this paper is to build systems of differential operators 
that satisfy the condition (\ref{Eqn:CIS1}),
when $M$ is a homogeneous manifold $G/Q$ with $Q$ 
in a certain class of maximal parabolic subgroups.
This is to construct systems 
$D_1, \ldots, D_n$ acting on sections of bundles $\eV_s \to G/Q$ over $G/Q$
in a systematic manner and to determine the bundles $\eV_s$ 
on which the systems are conformally invariant.
The method that we use is different from one used 
by Barchini-Kable-Zierau in \cite{BKZ08}.


To describe our work more precisely,
let $G$ be a complex, simple, connected, simply-connected Lie group
with Lie algebra $\fg$. It is known that $\fg$ has a $\Z$-grading
$\fg = \bigoplus_{j=-r}^r \fg(j)$ so that 
$\fq = \fg(0) \oplus \bigoplus_{j>0} \fg(j) = \fl \oplus \fn$ is a parabolic
subalgebra of $\fg$.
Let $Q = N_G(\fq) = LN$. For a real form $\fg_0$ of $\fg$,
define $G_0$ to be an analytic subgroup of $G$ with Lie algebra $\fg$.
Set $Q_0 = N_{G_0}(\fq)$. Our manifold is $M=G_0/Q_0$ and we consider
a line bundle $\Cal{L}_{s} \to G_0/Q_0$ for each $s\in \C$.
By the Bruhat theory, that  $G_0/Q_0$ admits an open dense
submanifold $\bar{N}_0Q_0/Q_0$. We restrict our bundle to this submanifold.
The systems that we study are on smooth sections of the restricted bundle.

To build systems of differential operators we observe that 
$L$ acts by the adjoint representation on $\fg(1)$ with a unique open orbit. 
This makes $\fg(1)$ a prehomogeneous vector space. Our construction
is based on the invariant theory of a prehomogeneous vector space. 
It is natural to associate $L$-equivariant polynomial maps called \emph{covariant maps} 
to the prehomogeneous vector space $(L, \Ad, \fg(1))$.
To define our systems of differential operators,
we use covariant maps 
$\tau_k: \fg(1) \to \fg(-r+k) \otimes \fg(r)$ 
that are associated to $\fg(1)$.
(See Definition \ref{Def6.3}.)
Each $\tau_k$ can be thought of as giving the symbols of the differential operators
that we study.

Let
$\fg(-r+k) \otimes \fg(r) = V_1 \oplus \cdots \oplus V_m$
be the irreducible decomposition of $\fg(-r+k) \otimes \fg(r)$
as an $L$-module.
Covariant map $\tau_k$ induces an $L$-equivariant linear map
$\ttau_k|_{V^*_j} :V^*_j \to \Cal{P}^k(\fg(1)))$ with
$V^*_j$ the dual of an irreducible constituent $V_j$ of $\fg(-r+k) \otimes \fg(r)$
and $\Cal{P}^k(\fg(1))$ the space of homogeneous polynomials on $\fg(1)$ of degree $k$.
We define differential operators from $\ttau_k|_{V^*_j}(Y^*)$.
For $Y^* \in V^*_j$, let $\Omega_k(Y^*)$ denote
the $k$-th order differential operators that are constructed from
$\ttau_k|_{V^*_j}(Y^*)$.
We say that a list of differential operators $D_1, \ldots, D_n$ 
is the \emph{$\Omega_k|_{V^*_j}$ system} 
if it is equivalent (see Definition \ref{Def2.1.4}) 
to a list of differential operators
$\Omega_k(Y^*_{1}), \ldots, \Omega_k(Y^*_{n})$,
where $\{Y^*_1, \ldots, Y^*_n\}$ is a basis for $V^*_j$.
By construction the $\Omega_k|_{V^*_j}$ system consists of 
$\dim_\C(V_j)$ operators.

The conformal invariance of the $\Omega_k|_{V^*_j}$ system
depends on the complex parameter $s$ for the line bundle $\Cal{L}_{s}$. 
Then we say that the $\Omega_k|_{V^*_j}$ system has \emph{special value $s_0$} 
if the system is conformally invariant on the line bundle $\Cal{L}_{s_0}$.
The special values for the case that $\dim([\fn,\fn])=1$ 
for $\fq = \fl \oplus \fn$
are studied by Barchini-Kable-Zierau in \cite{BKZ08} and \cite{BKZ09}, and myself in
\cite{Kubo2} and 
\cite{Kubo11}.

In this paper we consider a more general case; namely,
$\fq=\fl \oplus \fn$ is a maximal parabolic subalgebra and $\fn$ satisfies the condition that 
$[\fn, [\fn, \fn]] = 0$ and $\dim_\C([\fn, \fn]) >1$.
We call such maximal parabolic subalgebras \emph{quasi-Heisenberg type}.
In this case we have $r=2$ in (\ref{EqnTau}).
Therefore the $\Omega_k$ systems for $k\geq 5$ are zero.
We determine the special values of 
the $\Omega_1$ system and $\Omega_2$ systems 
associated to the parabolic subalgebras under consideration.

We may want to remark that, 
although the special value of $s$ for the $\Omega_1$ system is easily found 
by computing the bracket $[X, \Omega_1(Y^*_i)]$,
it is not easy to find the special values for the $\Omega_2$ systems
by a direct computation. (See Section 5 of \cite{BKZ08}.)
In this paper, to find the special values for the $\Omega_2$ systems,
we use two reduction techniques. These techniques significantly 
reduce the amount of computations. 
(See Proposition \ref{Prop7.2.3} and Proposition \ref{Prop7.1.2}.)

This paper consists of seven sections including this introduction and one appendix.
We now outline the contents of the rest of this paper.
In Section \ref{chap:CIS}, we first recall the definition of conformally invariant systems 
of differential operators and then collect some useful formulas.
In Section \ref{SS721},
the construction of the $\Omega_k$ systems is given precisely.
To construct the $\Omega_1$ system and 
$\Omega_2$ systems for maximal parabolic subalgebra $\fq$ 
of quasi-Heisenberg type, 
we study such a parabolic subalgebra $\fq$ and
the associated 2-grading 
$\fg=\bigoplus_{j=-2}^2 \fg(j)$ in Section \ref{chap:TwoStep}.

We construct the $\Omega_1$ system
and find its special value in Section \ref{chap:O1}. 
In this section we also fix normalizations for root vectors. 
The normalizations play an important role
to construct the system. 
We show that the special value $s_1$ for the $\Omega_1$ system 
is $s_1=0$. This is done in Theorem \ref{Thm8.1.2}.

To build the $\Omega_2$ systems, we need to find the irreducible constituents
$V^*$ of $\fl^* \otimes \fz(\fn)^*$ so that $\ttau_2|_{V^*} \neq 0$. 
In Section \ref{chap:Sp}, we show preliminary
results to find such irreducible constituents.
First we decompose $\fl \otimes \fz(\fn)$
into the direct sum of the irreducible constituents.
By using the decomposition results,
we then determine the candidates of the irreducible constituents $V^*$ 
so that $\ttau_2|_{V^*} \neq 0$. 
We build the $\Omega_2$ systems and 
find their special values in Section \ref{chap:O2}. 
The special values are determined in Theorem \ref{Thm8.3.1}.

Finally, in Appendix \ref{chap:Data}, we summarize 
the miscellaneous useful data for the parabolic subalgebras 
under consideration.



%% file: SS_CIS.tex
\section{Conformally Invariant Systems}
\label{chap:CIS}

In this section we recall from \cite{BKZ09} the definition of
a conformally invariant system of differential operators.
We also collect some properties of such a system of differential operators.

\subsection{Conformally invariant systems}\label{SS11}

Let $\fg_0$ be a real Lie algebra and 
$\eX(M)$ be the space of smooth vector fields on 
a smooth manifold $M$.

\begin{Def}\emph{\cite[page 790]{BKZ09}}\label{Def1.1.1}
A smooth manifold $M$ is called 
a \textbf{$\fg_0$-manifold} if there is an $\R$-linear map 
$\pi_M: \fg_0 \to C^\infty(M) \oplus \eX(M)$ so that 
for all $X, Y \in \fg_0$,
\begin{equation*}
\pi_M([X,Y]) = [\pi_M(X), \pi_M(Y)].
\end{equation*}
\end{Def}

For each $X \in \fg_0$, we write $\pi_M(X) = \pi_0(X) + \pi_1(X)$
with $\pi_0(X) \in C^\infty(M)$ and $\pi_1(X) \in \eX(M)$.
We denote by $\mathbb{D}(\eV)$ the space of differential operators
on smooth sections of $\eV$.

\begin{Def}\emph{\cite[page 791]{BKZ09}}\label{Def1.1.2}
Let $M$ be a $\fg_0$-manifold. 
A vector bundle $\eV \to M$ is called 
a \textbf{$\fg_0$-bundle} if there is an $\R$-linear map 
$\pi_{\eV} : \fg_0 \to \mathbb{D}(\eV)$ 
that satisfies the following properties:
\begin{enumerate}
\item[\emph{(B1)}] We have
$\pi_\eV([X,Y]) = [\pi_\eV(X), \pi_\eV(Y)]$ for all $X, Y \in \fg_0$.  
\item[\emph{(B2)}] In $\mathbb{D}(\eV)$, 
$[ \pi_{\eV}(X), f] = \pi_1(X) \acts f$ for all $X \in \fg_0$ and $f \in C^{\infty}(M)$.
\end{enumerate}
\end{Def}


\begin{Def}\emph{\cite[page 791]{BKZ09}}\label{Def1.1.3}
Let $\eV \to M$ be a $\fg_0$-bundle.
A \textbf{conformally invariant system} on $\eV$ with respect to $\pi_{\eV}$
is a list of differential operators 
$D_1, \ldots, D_m \in \mathbb{D}(\eV)$
so that the following two conditions hold:
\begin{enumerate}
\item[\emph{(S1)}] At each point $p \in M$, 
the list $ D_1, \ldots, D_m$ is linearly independent.
\item[\emph{(S2)}] For each $X\in \fg_0$, there is a matrix $C(X)$ in 
$M_{m \times m}(C^{\infty}(M))$ so that in $\mathbb{D}(\eV)$,
\begin{equation*}
[\pi_\eV(X), D_i] = \sum_{j=1}^m C_{ji}(X) D_j.
\end{equation*}
\end{enumerate}
The map $C: \fg_0 \to M_{m \times m}(C^{\infty}(M))$ is called
the \textbf{structure operator} of the conformally invariant system.
\end{Def}

If $\fg$ is the complexification of $\fg_0$ then $\fg$-manifolds and 
$\fg$-bundles are defined by extending the $\fg_0$-action $\C$-linearly.
In p. 792 in \cite{BKZ09}, the equivalence of two conformally invariant systems
are defined. For later convenience we apply the same definition to any systems
of differential operators. (See Definition \ref{Def2.3.5}.)

\begin{Def}
\label{Def2.1.4}
We say that two systems of differential operators 
(not necessarily conformally invariant)
$D_1, \ldots, D_n$ and $D_1', \ldots, D_n'$ 
in $\mathbb{D}(\eV)$
are \textbf{equivalent} if there is a matrix 
$A \in GL(n, C^{\infty}(M))$ so that, for $1\leq i \leq n$,
\begin{equation*}
D'_i = \sum_{j=1}^n A_{ji}D_j.
\end{equation*}
\end{Def}

\begin{Def}\emph{\cite[page 793]{BKZ09}}\label{Def2.1.5}
A conformally invariant system $D_1, \ldots, D_n$ is called 
\textbf{reducible} if there are an equivalent conformally invariant system
$D_1', \ldots, D_n'$ and an $m<n$ such that the system $D_1', \ldots, D_m'$
is conformally invariant. Otherwise we say that $D_1, \ldots, D_n$ is 
\textbf{irreducible}.
\end{Def}



We now specialize the $\fg$-manifold and $\fg$-bundle that we will work with.
Let $G$ be a complex, simple, connected, simply-connected Lie group
with Lie algebra $\fg$. Such $G$ contains a maximal connected solvable
subgroup $B$. Write $\fb =\fh \oplus \fu$ for its Lie algebra with $\fh$ 
the Cartan subalgebra and $\fu$ the nilpotent subalgebra. 
Let $\fq \supset \fb$ be a parabolic subalgebra of $\fg$. 
We define $Q = N_G(\fq)$, a parabolic subgroup of $G$.
Write $Q = LN$ for the Levi decomposition of $Q$.
with $L$ the Levi subgroup and $N$ the nilpotent subgroup.

Let $\fg_0$ be a real form of $\fg$ in which 
the complex parabolic subalgebra $\fq$ has a real form $\fq_0$,
and let $G_0$ be the analytic subgroup of $G$ with Lie algebra $\fg_0$. 
Define $Q_0 = N_{G_0}(\fq) \subset Q$, and write $Q_0=L_0N_0$.
We will work with $M=G_0/Q_0$ for a class of maximal parabolic subgroup 
$Q_0$ that will be specified in Section \ref{chap:TwoStep}.

Next, we need to specify a vector bundle $\Cal{V}$ on $M$. 
To this end let $\gD = \gD(\fg,\fh)$ be the set of roots of $\fg$ with respect to $\fh$.
Let $\gD^+$ be the positive system attached to $\fb$ and 
denote by $\Pi$ the set of simple roots.
For each subset $S \subset \Pi$, let $\fq_S$ be the 
corresponding standard parabolic subalgebra. 
Write $\fq_S = \fl_S \oplus \fn_S$ with Levi factor 
$\fl_S = \fh \oplus  \bigoplus_{\ga \in \gD_S}\fg_\ga$ and 
nilpotent radical $\fn_S = \bigoplus_{\ga \in \gD^+ \backslash \gD_S} \fg_\ga$,
where $\gD_S = \{ \ga \in \gD \; | \; \ga \in \textrm{span}(\Pi \backslash S) \}$.
If $Q_0$ is a maximal parabolic then 
there exists a unique simple root $\ga_\fq \in \Pi$ so that 
$\fq = \fq_{\{\ga_\fq\}}$. Let $\gl_\fq$ be the fundamental weight of $\ga_\fq$.
The weight $\gl_\fq$ is orthogonal to any roots $\ga$ 
with $\fg_\ga \subset [\fl, \fl]$. Hence it exponentiates to 
a character $\chi_\fq$ of $L$. As $\chi_\fq$ takes real values
on $L_0$, for $s \in \C$, character $\chi^{s} =|\chi_\fq|^{s}$
is well-defined on $L_0$. 
Let $\C_{\chi^{s}}$ be the one-dimensional representation of $L_0$
with character $\chi^{s}$. 
The representation $\chi^{s}$ is extended to a representation of $Q_0$
by making it trivial on $N_0$. Then it deduces a line bundle $\Cal{L}_{s}$
on $M = G_0/Q_0$ with fiber $\C_{\chi^{s}}$.

The group $G_0$ acts on the space
\begin{align*}
&C^\infty_{\chi}(G_0/Q_0, \C_{\chi^{s}})\\
&= \{ F \in C^\infty(G_0, \C_{\chi^{s}}) \; |\;
\text{$F(gq) = \chi^{s}(q^{-1})F(g)$ for all $q \in Q_0$ and $g \in G_0$} \}
\end{align*}
\noindent by left translation.
The action $\pi_s$ of $\fg_0$ on $C^\infty_\chi(G_0/Q_0, \C_{\chi^{s}})$
arising from this action is given by
\beu
(\pi_s(Y) \acts F)(g) = \frac{d}{dt}F(\exp(-tY)g)\big|_{t=0}
\end{equation*}

\noindent for $Y \in \fg_0$.
This action is extended $\C$-linearly
to $\fg$ and then naturally to the universal enveloping algebra $\Cal{U}(\fg)$.
We use the same symbols for the extended actions.

Let $\bar{N}_0$ be the unipotent subgroup opposite to $N_0$.
By the Bruhat theory, the subset $\bar{N}_0Q_0$ 
is open and dense in $G_0$. Then
the restriction map 
$C^\infty_{\chi}(G_0/Q_0, \C_{\chi^{s}}) \to C^\infty(\bar N_0, \C_{\chi^{s}})$
is an injection,
where $C^\infty(\bar{N}_0,\C_{\chi^{s}})$
is the space of the smooth
functions from $\bar{N}_0$ to $\C_{\chi^{s}}$. 
Then,
for $u \in \Cal{U}(\fg)$ and $F \in C^\infty_{\chi}(G_0/Q_0, \C_{\chi^{s}})$,
we let $f = F|_{\bar N_0}$ and define the action 
of $\Cal{U}(\fg)$ on the image of the restriction map by 
\begin{equation}\label{Eqn223}
\pi_s(u) \acts f = \big(\pi_s(u) \acts F\big)|_{\bar N_0}.
\end{equation} 
The line bundle $\Cal{L}_{s} \to G_0/Q_0$ restricted to
$\bar{N}_0$ is the trivial bundle 
$\bar{N}_0 \times \C_{\chi^{s}} \to \bar{N}_0$.
By slight abuse of notation, we refer to the trivial bundle over $\bar{N}_0$
as $\Cal{L}_{s}$.
Then in practice our manifold $M$ will be $M=\bar{N}_0$ 
and our vector bundle will be the trivial bundle.


\vsp

Now we show that, with the action $\pi_s$,
the group $\bar{N}_0$ and the trivial bundle $\Cal{L}_{s}$ 
are a $\fg$-manifold and $\fg$-bundle, respectively.
Let $\bar\fn$ and $\fq$ be the complexifications of 
the Lie algebras of $\bar N_0$ and $Q_0$, respectively;
we have the direct sum $\fg = \bar \fn \oplus \fq$.
For $Y \in \fg$, write $Y = Y_{\bar \fn} + Y_{\fq}$ for the decomposition 
of $Y$ in this direct sum.
Similarly, write the Bruhat decomposition of $g \in \bar N_0 Q_0$ as 
$g= \mathbf{\bar n}(g)\mathbf{q}(g)$ with $\mathbf{\bar n}(g) \in \bar N_0$ and 
$\mathbf{q}(g) \in Q_0$. For $Y \in \fg_0$, we have
\begin{equation}\label{Eqn7.2.1}
Y_{\bar \fn} = \frac{d}{dt} \mathbf{\bar n}(\exp(tY)) \big|_{t=0},
\end{equation}
and a similar equality holds for $Y_{\fq}$.
Define a right action
$R$ of $\Cal{U}(\bar \fn)$ on $C^\infty(\bar N_0, \C_{\chi^{s}})$ by
\begin{equation}\label{Eqn7.2.0}
\big( R(X) \acts f \big)(\nbar) = \frac{d}{dt} f\big( \nbar \exp(tX) \big) \big|_{t=0}
\end{equation}
\noindent for $X \in \bar\fn_0$ and $f \in C^\infty(\bar N_0 ,\C_{\chi^{s}})$. 
Observe that, by definition, the differential $d\chi$ of $\chi$ is $d\chi = \gl_\fq$. 
A direct computation then shows that, 
for $Y \in \fg$ and $f$ in the image of the restriction map 
$C^\infty_{\chi}(G_0/Q_0, \C_{\chi^{s}}) \to C^\infty(\bar N_0, \C_{\chi^{s}})$,
we have
\begin{equation}\label{Eqn7.2.2}
\big(\pi_s(Y) \acts f \big)(\nbar) = s\gl_\fq \big( (\Ad(\nbar^{-1})Y)_\fq \big) f(\nbar) 
-\big( R \big( (\Ad(\nbar^{-1})Y)_{\bar \fn} \big) \acts f \big)(\nbar).
\end{equation}
Observe that (\ref{Eqn7.2.2}) implies that the representation $\pi_s$ extends 
to a representation of $\Cal{U}(\fg)$ on 
the whole space $C^\infty(\bar{N}_0 , \C_{\chi^{s}})$. 
Moreover, it also shows that for all $Y \in \fg$,
the linear map $\pi_s(Y)$ is in $C^\infty(\bar{N}_0) \oplus \eX(\bar{N}_0)$.
Therefore, with this linear map $\pi_s$, $\bar{N}_0$ is a $\fg$-manifold.

Next, we show that 
the linear map $\pi_s$ gives $\Cal{L}_{s}$ the structure of a $\fg$-bundle.
As $\pi_s$ is a representation of $\fg$, the condition (B1) of 
Definition \ref{Def1.1.2} is trivial. Thus it suffices to show that 
the condition (B2) holds.
Since $\Cal{L}_{s}$ is the trivial bundle of $\bar{N}_0$ with fiber $\C_{\chi^{s}}$,
the space of smooth sections of $\Cal{L}_{s}$
is identified with $C^\infty(\bar{N}_0 , \C_{\chi^{s}})$. 
The following proposition establishes the condition (B2) in our situation.

\begin{Prop}\label{Prop7.2.2}
In $\mathbb{D}(\Cal{L}_{s})$, 
for $Y \in \fg$ and $f \in C^\infty(\bar{N}_0)$,
we have
\begin{equation*}
\big( [\pi_s(Y),f] \big)(\nbar) 
=-\big( R \big( (\emph{\Ad}(\nbar^{-1})Y)_{\bar \fn} \big) \acts f \big)(\nbar).
\end{equation*}
\end{Prop}

\begin{proof}
This follows from the definition of  $[\pi_s(Y),f]$ and formula (\ref{Eqn7.2.2}).
\end{proof}

\subsection{Properties of conformally invariant systems}

In Section \ref{SS721} we are going to construct systems of differential operators
on $\Cal{L}_{s}$. The systems of operators will satisfy several properties
of conformally invariant systems.
For convenience 
we collect those properties from \cite{BKZ09} here.

We first define an action of $L_0$ on 
$\mathbb{D}(\Cal{L}_{s})$.
As on p. 805 of \cite{BKZ09}, 
we define an action of $L_0$ on 
$C^\infty(\bar{N}_0, \C_{\chi^{s}})$ by
\begin{equation*}
(l \cdot f)(\nbar) = \chi^{s}(l)f(l^{-1} \nbar l).
\end{equation*}
This action agrees with the action of $L_0$ by the left translation on the image
of the restriction map 
$C^\infty_\chi(G_0/Q_0, \C_{\chi^{s}}) \to C^\infty(\bar{N}_0, \C_{\chi^{s}})$.
In terms of this action we define an action of $L_0$ on 
$\mathbb{D}(\Cal{L}_{s})$ by
\begin{equation}\label{Eqn234}
(l \cdot D) \acts f = l \cdot (D \acts (l^{-1} \cdot f)).
\end{equation}

\begin{Def}\emph{\cite[page 806]{BKZ09}} \label{Def2.3.6}
A conformally invariant system $D_1, \ldots, D_m$
on $\Cal{L}_{s} \to \bar{N}_0$ is called \textbf{$L_0$-stable}
if there is a map $c: L_0 \to GL(n, C^{\infty}(\bar{N}_0))$ such that
\begin{equation*}
l \cdot D_i = \sum_{j=1}^m c(l)_{ji}D_j.
\end{equation*}
\end{Def}

It is known that 
there exists a semisimple element $H_0 \in \fl$,
so that $\ad(H_0)$ has only integer eigenvalues on $\fg$ with
$\fg(1) \neq \{0\}$, $\fl = \fg(0)$, $\fn = \bigoplus_{j>0}\fg(j)$,
and $\bar \fn = \bigoplus_{j>0}\fg(-j)$, where $\fg(j)$ is the 
$j$-eigenspace of $\ad(H_0)$. 
(See for example \cite[Section X.3]{Knapp02})

 \begin{Def}\emph{\cite[page 804]{BKZ09}}
A conformally invariant system $D_1, \ldots, D_m$
is called \textbf{homogeneous} if $C(H_0)$ is a scalar matrix,
where $C$ is the structure operator of the conformally invariant system.
(See Definition \ref{Def1.1.3}.)
\end{Def}

\begin{Prop}\emph{\cite[Proposition 17]{BKZ09}}\label{Prop218}
Any irreducible conformally invariant system is homogeneous.
\end{Prop}

Define 
\beu
\mathbb{D}(\Cal{L}_{s})^{\bar \fn}
=\{ D \in \mathbb{D}(\Cal{L}_{s}) \; | \; [\pi_s(X), D] = 0 
\text{ for all $X \in \bar \fn$} \}.
\end{equation*}

\noindent Observe that 
 in the sense of 
\cite[page 796]{BKZ09},
the $\fg$-manifold $\bar{N}_0$ 
is \emph{straight}
with respect to the subalgebra $\bar{\fn}$ of $\fg$
(\cite[page 799]{BKZ09}). 
Then we state the definition of \emph{straight} conformally invariant systems
specialized to the present situation. (For the general definition
see p. 797 of \cite{BKZ09}.)

\begin{Def}
We say that a conformally invariant system $D_1, \ldots, D_m$ 
is \textbf{straight} if  
$D_j \in \mathbb{D}(\Cal{L}_{s})^{\bar \fn}$ for $j=1,\ldots, m$.
\end{Def}

In general, to show that a given list $D_1, \ldots, D_m$ of differential operators
on $\bar{N}_0$ is a conformally invariant system, 
we need check (S2) of Definition \ref{Def1.1.3}
at each point of $\bar{N}_0$.
Proposition \ref{Prop7.2.3} below shows that in the case $D_1,\ldots, D_m$
in $\mathbb{D}(\Cal{L}_{s})^{\bar \fn}$, 
it suffices to check the condition only at the identity $e$.

\begin{Prop}\emph{\cite[Proposition 13]{BKZ09}}\label{Prop7.2.3}
Let $D_1, \ldots, D_m$ be a list of operators in $\mathbb{D}(\Cal{L}_{s})^{\bar \fn}$.
Suppose that the list is linearly independent at $e$ and that there is a map
$b: \fg \to \f{gl}(m,\C)$ such that
\begin{equation*}
\big([\pi_s(Y), D_i] \acts f\big)(e) = \sum_{j=1}^m b(Y)_{ji}(D_j \acts f)(e)
\end{equation*} 
\noindent for all $Y \in \fg, \; f \in C^\infty(\bar{N}_0, \C_{\chi^{s}})$, and $1\leq i \leq m$. 
Then $D_1, \ldots, D_m$ is a conformally invariant system on $\Cal{L}_{s}$.
The structure operator of the system is given by 
$C(Y)(\nbar) = b(\emph{\Ad}(\nbar^{-1})Y)$ for all $\nbar \in \bar{N}_0$ and $Y \in \fg$.
\end{Prop}

\vsp


To end this section we are going to give 
two formulas that will make certain arguments simple
in Section \ref{chap:O1} and Section \ref{chap:O2}.

\begin{Prop}\label{Prop7.2.5}
Let $Y \in \fg$ and $f \in C^\infty(\bar{N_0} , \C_{\chi^{s}})$.
For $X, X_1, X_2 \in \bar{\fn}$, we have
\begin{equation*}
\big( [\pi_s(Y), R(X)] \acts f \big)(\nbar)\\
=\big( R([(\emph{\Ad}(\nbar^{-1})Y)_{\fq}, X]_{\bar{\fn}}) \acts f \big)(\nbar)
-s\gl_\fq\big([\emph{\Ad}(\nbar^{-1})Y, X]_{\fq}\big)f(\nbar)
\end{equation*}
\noindent and
\begin{align*}
&\big( [\pi_s(Y), R(X_1)R(X_2)]\acts f\big)(\nbar)\\
&= \big(R([(\emph{\Ad}(\nbar^{-1})Y)_\fq,X_1]_{\bar\fn})R(X_2)\acts f\big)(\nbar)
+\big(R(X_1)R([(\emph{\Ad}(\nbar^{-1})Y)_\fq, X_2]_{\bar\fn})\acts f\big)(\nbar)\\
&\quad +\big( R([[\emph{\Ad}(\nbar^{-1})Y, X_1]_\fq,X_2]_{\bar\fn})\acts f\big)(\nbar)
-s\gl_\fq([\emph{\Ad}(\nbar^{-1})Y, X_1]_\fq)(R(X_2)\acts f)(\nbar)\\
&\quad -s\gl_\fq([\emph{\Ad}(\bar{n}^{-1})Y, X_2]_\fq)(R(X_1)\acts f)(\bar{n})
-s\gl_\fq([[\emph{\Ad}(\bar{n}^{-1})Y, X_1],X_2]_\fq)f(\bar{n}).
\end{align*}
\end{Prop}

\begin{proof}
These formulas follow from 
a direct computation on the left hand side of each equation
using (\ref{Eqn7.2.2}).
\end{proof}



%% file: SS_Omega.tex
\section{The $\Omega_k$ Systems}\label{SS721}

The purpose of this section is to construct systems of 
$k$-th order differential operators
in $\mathbb{D}(\Cal{L}_{s})^{\bar \fn}$
in a systematic manner.
We shall call the systems of operators $\Omega_k$ systems.


\subsection{Construction of the $\Omega_k$ systems}
\label{SS7211}

Let $\fg = \bigoplus_{j=-r}^r \fg(j)$ on $\fg$ be a 
$\Z$-grading on $\fg$ with $\fg(1) \neq 0$.
By construction, $\fq = \fg(0) \oplus \bigoplus_{j>0} \fg(j)$
is a parabolic subalgebra. Take $L$ to be the analytic subgroup 
of $G$ with Lie algebra $\fg(0)$.
Observe that, by Vinberg's Theorem (\cite[Theorem 10.19]{Knapp02}),
the triple $(L, \Ad, \fg(1))$ is a prehomogeneous vector space, that is, 
$L$ has an open orbit in $\fg(1)$. 
To define our systems of differential operators,
we use covariant maps ($L$-equivariant polynomial maps),
which we denote by $\tau_k$,
associated to prehomogeneous vector space $(L, \Ad, \fg(1))$.
These maps can be thought to give symbols of a class of differential operators
that we will study.
We would like to acknowledge that the construction of $\tau_k$ as in this paper
was suggested by Anthony Kable.

\begin{Def}\label{Def6.3}
Let $\fg = \bigoplus_{j=-r}^r \fg(j)$ be a graded complex simple Lie algebra
with $\fg(1) \neq 0$.
Then, for $1\leq k \leq 2r$, the map $\tau_k$ on $\fg(1)$ is defined by 
\begin{align*}\label{Eqn25}
\tau_k: \fg(1) &\to \fg(-r+k) \otimes \fg(r) \\
X &\mapsto \frac{1}{k!} \big( \emph{\ad}(X)^k\otimes \emph{Id} \big) \omega \nonumber
\end{align*}
with
\begin{equation*}
\omega = \sum_{\gamma_j \in \gD(\fg(r))}
X_{\gamma_j}^* \otimes X_{\gamma_j} \in \fg(-r) \otimes \fg(r),
\end{equation*}
where $\gD(\fg(r))$ is the set of roots $\ga$ so that $\fg_\ga \subset \fg(r)$,
and where $X_{\gamma_j}^*$ are the elements in $\fg(-r)$ 
dual to $X_{\gamma_j}$ with respect to the Killing form;
namely,
$X^*_{\gamma_j}(X_{\gamma_t}) 
:= \gk(X^*_{\gamma_j}, X_{\gamma_t})
=\gd_{j,t}$
with $\gd_{i,t}$ the Kronecker delta.
\end{Def}


We shall check in Lemma \ref{Lem6.5} 
that these maps are indeed $L$-equivariant.
Observe that, by the standard argument, the element $\omega$
is independent of a choice of a basis for $\fg(r)$ and the dual basis
for $\fg(-r)$.

\begin{Lem}\label{Cor6.5}
Let $\fg = \bigoplus_{j=-r}^r \fg(j)$ be a graded complex simple Lie algebra
with $\fg(1) \neq 0$ and $G$ be a complex analytic group with Lie algebra $\fg$.
If $L$ is the analytic subgroup of $G$ with Lie algebra $\fg(0)$ 
and $\omega$ is as in Definition \ref{Def6.3} then, for all $l \in L$,
\beu
(\emph{\Ad}(l) \otimes \emph{\Ad}(l)) \omega = \omega.
\end{equation*}
\end{Lem}

\begin{proof}
If $g \in L$ then $\{\Ad(l)X_{\gamma_j} \; | \; \gamma_j \in \gD(\fg(r))\}$
forms a basis for $\fg(r)$. It also holds that
$\{\Ad(l)X^*_{\gamma_j}\; |\; \gamma_j \in \gD(\fg(r))\}$
is the dual basis for $\fg(-r)$ with respect to the Killing form.
Now the assertion follows from the property that $\omega$ is 
independent of a choice of a basis for $\fg(r)$ and the dual basis for 
$\fg(-r)$.
\end{proof}


\begin{Lem}\label{Lem6.5}
Let $\fg = \bigoplus_{j=-r}^r \fg(j)$, $G$, and $L$ be as in Lemma \ref{Cor6.5}.
For all $l \in L$, $X \in \fg(1)$, and for $0 \leq k \leq 2r$, we have 
\begin{equation*}\label{Eqn6.6}
\tau_k(\emph{\Ad}(l)X) = (\emph{\Ad}(l)\otimes \emph{\Ad}(l))\tau_k(X).
\end{equation*}
\end{Lem}

\begin{proof}
For $l \in L$, we have
\begin{align*}
\tau_k(\Ad(l)X) 
&= \frac{1}{k!} \sum_{\gamma_j \in \gD(\fz(\fn))} 
\ad(\Ad(l)(X))^k(X^*_{\gamma_j}) \otimes X_{\gamma_j}\\
&=\frac{1}{k!} \sum_{\gamma_j \in \gD(\fz(\fn))} 
\Ad(l) \big( \ad(X)^k(\Ad(l^{-1})X^*_{\gamma_j}) \big )\otimes X_{\gamma_j}\\
&= (\Ad(l)\otimes \Ad(l)) \bigg( \frac{1}{k!} \sum_{\gamma_j \in \gD(\fz(\fn))} 
\ad(X)^k(\Ad(l^{-1})X^*_{\gamma_j}) \otimes \Ad(l^{-1})(X_{\gamma_j}) \bigg)\\
&= (\Ad(l)\otimes \Ad(l)) \bigg( \frac{1}{k!}\big(\ad(X)^k\otimes \text{Id}\big) \omega \bigg)\\
&= (\Ad(l)\otimes \Ad(l)) \tau_k(X).
\end{align*}
\noindent Note that Lemma \ref{Cor6.5} is applied from line four to line five.
\end{proof}

\vsp

Now we are going to build the systems of differential operators in 
$\mathbb{D}(\Cal{L}_{s})^{\bar{\fn}}$ that we study. 
Observe that, as
$\tau_k: \fg(1) \to \fg(-r+k)\otimes \fg(r) =: W$
are $L$-equivariant polynomial maps of degree $k$,
the maps $\tau_k$ can be thought of as elements in 
$(\Cal{P}^k(\fg(1)) \otimes W)^{L}$, where $\Cal{P}^k(\fg(1))$
denotes the space of homogeneous polynomials on $\fg(1)$ of degree $k$.
Then the isomorphism 
$(\Cal{P}^k(\fg(1)) \otimes W)^{L} \cong  \Hom_{L}(W^*, \Cal{P}^k(\fg(1)))$
yields the $L$-intertwining operators $\ttau_k$, that are given by
\begin{equation}\label{EqnTau}
\ttau_k(Y^*)(X) = Y^*(\tau_k(X)),
\end{equation}
where $W^*$ is the dual module of $W$ with respect to the Killing form.
For each $Y^*\in W^*$, we have 
$\ttau_k(Y^*) \in \Cal{P}^k(\fg(1)) \cong \Sym^k(\fg(-1))$.
We define differential operators in $\mathbb{D}(\Cal{L}_{s})^{\bar{\fn}}$
from $\ttau_k(Y^*)$. This is done as follows.
Let $\gs: \Sym^k(\fg(-1)) \to \Cal{U}(\bar \fn)$ be the symmetrization operator.
Identify $\Cal{U}(\bar{\fn})$ with $\mathbb{D}(\Cal{L}_{s})^{\bar\fn}$ 
by making $\bar \fn$ act on $C^\infty(\bar{N}_0, \C_{\chi^{s}})$ 
via right differentiation $R$. 
Then we have a composition of linear maps
\begin{equation*}
W^* \stackrel{\ttau_k}{\to} \Cal{P}^k(\fg(1))
\cong \Sym^k(\fg(-1)) \stackrel{\gs}{\hookrightarrow} \Cal{U}(\bar \fn)
\stackrel{R}{\to} \mathbb{D}(\Cal{L}_{s})^{\bar \fn}.
\end{equation*}
For $Y^* \in W^*$, we define
a differential operator $\Omega_k(Y^*) \in \mathbb{D}(\Cal{L}_{s})^{\bar\fn}$
by 
\begin{equation*}\label{Eqn:Omega}
\Omega_k(Y^*) = R \circ \gs \circ \ttau_k(Y^*).
\end{equation*}

As we will work with irreducible systems we need to be a little more 
careful with our construction; in particular, 
we need to take an irreducible constituent of $\fg(-r+k)^* \otimes \fg(r)^*$.
Let
$\fg(-r+k) \otimes \fg(r) = V_1 \oplus \cdots \oplus V_m$
be the irreducible decomposition of $\fg(-r+k) \otimes \fg(r)$
as an $L$-module, and let
$\fg(-r+k)^* \otimes \fg(r)^* = V^*_1 \oplus \cdots \oplus V^*_m$
be the corresponding irreducible decomposition of $\fg(-r+k)^* \otimes \fg(r)^*$,
where $\fg(j)^*$ are the dual spaces of
$\fg(j)$ with respect to the Killing form.
For each irreducible constituent $V^*_j$ of $\fg(-r+k)^* \otimes \fg(r)^*$,
there exists an $L$-intertwining operator 
$\ttau_k|_{V^*_j} \in \Hom_L(V^*_j, \Cal{P}^k(\fg(1)))$ given
as in (\ref{EqnTau}). Then we define a
linear operator 
$\Omega_k|_{V^*_j}: V_j^* \to \mathbb{D}(\Cal{L}_{s})^{\bar \fn}$ by
\begin{equation*}
\Omega_k|_{V^*_j} = R \circ \gs \circ \ttau_k|_{V^*_j}.
\end{equation*}
Since, for $Y^* \in V_j^*$, we have 
$\Omega_k|_{V^*_j}(Y^*) = \Omega_k(Y^*)$ as a differential operator, 
we simply write $\Omega_k(Y^*)$ for the differential operator
arising from $Y^* \in V^*_j$.

\begin{Def}\label{Def2.3.5}
If $V^*$ is an irreducible constituent of $\fg(-r+k)^* \otimes \fg(r)^*$
so that $\ttau_k|_{V^*} \not \equiv 0$
then a list of differential operators $D_1, \ldots, D_n \in 
\mathbb{D}(\Cal{L}_{s})^{\bar \fn}$
is called the \textbf{$\Omega_k|_{V^*}$ system} 
if it is equivalent (see Definition \ref{Def2.1.4}) 
to a list of differential operators
\begin{equation}
\Omega_k(Y^*_{1}), \ldots, \Omega_k(Y^*_{n}),
\end{equation}
where $\{Y^*_1, \ldots, Y^*_n\}$ is a basis for $V^*$.
\end{Def}

We also simply refer each $\Omega_k|_{W^*}$ system
to an $\Omega_k$ system.
We want to remark that the construction of 
the $\Omega_k$ systems might require additional modification
to secure the conformal invariance.
See Section 6 in \cite{BKZ08} and Section 3 in \cite{Kubo11} 
for the modification for the $\Omega_3$ systems
of the parabolic subalgebra of Heisenberg type.

\vsp

It is important to notice that 
it is not necessary for the $\Omega_k$ systems to be conformally 
invariant; their conformal invariance strongly depends on the complex
parameter $s$ for the line bundle $\Cal{L}_{s}$. 
So, we give the following definition.

\begin{Def}\label{Def2.3.5}
Let $V^*$ be an irreducible constituent of $\fg(-r+k)^* \otimes \fg(r)^*$.
Then we say that the $\Omega_k|_{V^*}$ system has \textbf{special value $s_0$} if
the system is conformally invariant on the line bundle $\Cal{L}_{s_0}$.
\end{Def}

Note that, as the opposite parabolic $\bar{Q}_0 = L_0 \bar{N}_0$ 
is chosen in \cite{BKZ08}, our special values $s_0$ are of the form 
$s_0 = -s'_0$, where $s'_0$ are the special values shown in \cite{BKZ08}.

\vsp

Observe that the linear operator
$\Omega_k|_{V^*}: V^* \to \mathbb{D}(\Cal{L}_{s})^{\bar \fn}$
is an $L_0$-intertwining operator with respect to the action
given in (\ref{Eqn234}); in particular,
the $\Omega_k|_{V^*}$ system is $L_0$-stable (see Definition \ref{Def2.3.6}).
Indeed, one can check that we have $l \cdot R(u) = R(\Ad(l)u)$ for $l \in L_0$
and $u \in \Cal{U}(\bar \fn)$. This action stabilizes the subspace
$\mathbb{D}(\Cal{L}_{s})^{\bar \fn}$. With the adjoint action of $L_0$ on
$\Cal{U}(\bar \fn)$, the linear isomorphism
$\Cal{U}(\bar \fn) \stackrel{R}{\to} \mathbb{D}(\Cal{L}_{s})^{\bar \fn}$
is then $L_0$-equvariant. 
It is clear that each map in 
$V^* \stackrel{\ttau_k|_{V^*}}{\to} \Cal{P}^k(\fg(1))
\cong \Sym^k(\fg(-1)) \stackrel{\gs}{\hookrightarrow} \Cal{U}(\bar \fn)$
is $L_0$-equivariant with respect to
the natural actions of $L_0$ on each space, which are induced by the 
adjoint action of $L_0$ on $\fg$.
Therefore, with the $L_0$-action (\ref{Eqn234}), 
the operator $\Omega_k|_{V^*}: V^* \to \mathbb{D}(\Cal{L}_{s})^{\bar \fn}$
is an $L_0$-intertwining operator.
Now we summarize some properties of the $\Omega_k|_{V^*}$ system.

\begin{Rem}\label{Rem2.3.6}
It follows from the definition and observation above
that  the $\Omega_k|_{V^*}$ system satisfies the following properties:
\begin{enumerate}
\item[(1)] The $\Omega_k|_{V^*}$ system
satisfies the condition (S1) of Definition \ref{Def1.1.3}.
\item[(2)] When the $\Omega_k|_{V^*}$ system is conformally invariant
then it is an irreducible, straight, and $L_0$-stable system.
By Proposition \ref{Prop218}, it is also a homogeneous system.
\end{enumerate}
\end{Rem}


\subsection{Computations involving the $\Omega_k$ systems}\label{SSLem}

We are going to show two technical lemmas
that will be used in Section \ref{chap:O2}.
For $D \in \mathbb{D}(\Cal{L}_{s})$, we denote by $D_{\nbar}$
the linear functional $f \mapsto (D\acts f)(\nbar)$ for 
$f \in C^\infty(\bar{N}_0, \C_{\chi^{s}})$.
A simple observation shows that 
$(D_1D_2)_{\nbar} = (D_1)_{\nbar}D_2$ for $D_1, D_2 \in \mathbb{D}(\Cal{L}_{s})$;
in particular, if $(D_1)_{\nbar} = 0$ then $[D_1, D_2]_{\nbar} = -(D_2)_{\nbar}D_1$.

\begin{Lem}\label{Lem832}
Suppose that $V^*$ is an irreducible constituent of 
$\fg(-r+k)^* \otimes \fg(r)^*$.
Let $X_1, X_2 \in \fg$ and $Y^*_1, \ldots, Y^*_n \in V^*$.
If $\spi(X_1)_e=0$ and if  we have
$[\pi_s(X_i), \Omega_k(Y^*_t)]_e \in  
\emph{\text{span}}_\C\{\Omega_k(Y^*_j)_e \; | \; j=1,\ldots n\}$ 
for $i=1,2$ then
\begin{equation}
\big [\spi(X_1), [\spi(X_2), \Omega_k(Y^*_t)] \big]_e \in 
\emph{\text{span}}_\C\{\Omega_k(Y^*_1)_e, 
\ldots, \Omega_k(Y^*_n)_e\}.
\end{equation}
\end{Lem}

\begin{proof}
Observe that $[\spi(X_1), [\spi(X_2), \Omega_k(Y^*_t)]]$ is
\begin{equation}\label{Eqn7.1.5}
\pi_s(X_1)[\spi(X_2), \Omega_k(Y^*_t)] - [\spi(X_2), \Omega_k(Y^*_t)]\pi_s(X_1).
\end{equation}
Since, by assumption, we have $\pi_s(X_1)_e =0$, 
the first term is zero at $e$. 
By assumption, the bracket $[\spi(X_2), \Omega_k(Y^*_t)]_e$
is a linear combination of $\Omega_k(Y^*_1)_e, \ldots, \Omega_k(Y^*_n)_e$ over $\C$. So it may be written as
$[\spi(X_2), \Omega_k(Y^*_t)]_e 
= \sum_{j=1}^na_{jt}\Omega_k(Y^*_j)_e$
with $a_{jt} \in \C$.
Then, the second term in (\ref{Eqn7.1.5}) evaluates to 
$-\sum_{j=1}^na_{jt}\Omega_k(Y^*_j)_e\pi_s(X_1)$
at the identity $e$.
Since $(\pi_s(X_1)\Omega_k(Y_j^*))_e 
= \pi_s(X_1)_e\Omega_k(Y_j^*) = 0$, we obtain
\begin{align*}
[\spi(X_1), [\spi(X_2), \Omega_k(Y^*_t)]]_e
&=-\sum_{j=1}^na_{jt}\Omega_k(Y^*_j)_e\pi_s(X_1)\\
&=-\sum_{j=1}^na_{jt}[\pi_s(X_1), \Omega_k(Y^*_j)]_e.
\end{align*}
Now the proposed result follows from the assumption that 
$[\spi(X_1), \Omega_k(Y^*_t)]_e$ is a linear combination of 
$\Omega_k(Y^*_j)_e$ over $\C$.
\end{proof}

If $\gD^+(\fl)$ is the set of positive roots in $\fl$ then 
we set
\begin{equation*}
\fu_\fl = \bigoplus_{\gD^+(\fl)}\fg_\ga \text{ and }
\bar{\fu}_\fl = \bigoplus_{\gD^+(\fl)}\fg_{-\ga}.
\end{equation*}

\begin{Lem}\label{Lem831}
Suppose that $\fg(1)$ is irreducible and that 
$V^*$ is an irreducible constituent of $\fg(-r+k)^* \otimes \fg(r)^*$.
Let $X_h$ be a highest weight vector for $\fg(1)$ and 
$Y^*_l$ be a lowest weight vector for $V^*$.
If 
\begin{equation*}\label{Eqn:Omega_k1}
[\pi_s(X_h), \Omega_k(Y^*_l)]_e \in 
\emph{\text{span}}_\C\{\Omega_k(Y^*_1)_e, 
\ldots, \Omega_k(Y^*_n)_e\}
\end{equation*}
with $\{Y^*_1, \ldots, Y^*_n \}$ a basis for $V^*$
then, for any $X \in \fg(1)$ and $Y^* \in V^*$, 
\begin{equation*}
\big [\spi(X), \Omega_k(Y^*)]_e \in 
\emph{\text{span}}_\C\{\Omega_k(Y^*_1)_e, 
\ldots, \Omega_k(Y^*_n)_e\}.
\end{equation*}
\end{Lem}

\begin{proof}
Set 
$E =\text{span}_{\C}\{\Omega_k(Y^*_1)_e, \ldots, \Omega_k(Y^*_n)_e\}$.
We first show that for each $X\in \fg(1)$,
\begin{equation}\label{Eqn2.6.3}
[\spi(X), \Omega_k(Y^*_l)]_e \in E.
\end{equation}
Observe that since $(L, \fg(1))$ is assumed to be irreducible,
the $L$-module $\fg(1)$ is given by $\fg(1) = \Cal{U}(\bar{\fu}_\fl)X_h$.
Then, as $\pi_s$ is linear on $\fg(1)$, it suffices to show that 
(\ref{Eqn2.6.3}) holds when $X = \bar{u}_k \cdot X_h$ with $\bar{u}_k$ 
a monomial in $\Cal{U}(\bar{\fu}_\fl)$.
This is done by induction on the order of $\bar{u}_k$.
Indeed, the proof is clear once we show that (\ref{Eqn2.6.3})
holds for $X=\bar{Z}\cdot X_h =[\bar{Z}, X_h]$ with $\bar{Z} \in \bar\fu_\fl$.

By the Jocobi identity, 
the commutator $[\spi([\bar{Z}, X_h]), \Omega_k(Y^*_l)]$ is 
\begin{align}\label{Eqn239}
&[\spi([\bar{Z}, X_h]), \Omega_k(Y^*_l)] \nonumber\\
&=[\spi(\bar{Z}), [\spi(X_h), \Omega_k(Y^*_l)]]
-[\spi(X_h),[\spi(\bar{Z}), \Omega_k(Y^*_l)]].
\end{align}
By the $\fl$-equivariance of the operator
$\Omega_k: V^* \to \mathbb{D}(\Cal{L}_{s})^{\bar \fn}$,
it follows that 
\begin{equation*}
[\pi_s(\bar{Z}), \Omega_k(Y^*_l)] = \Omega_k([\bar{Z}, Y^*_l]).
\end{equation*}
Since $\bar{Z} \in \bar{\fu}_\fl$ and $Y^*_l$ is a lowest weight vector,
we have $\Omega_k([\bar{Z}, Y^*_l]) =0$, and so is the second term
of the right hand side of (\ref{Eqn239}). Thus we have 
\begin{equation}\label{Eqn2.6.5}
[\spi([\bar{Z}, X_h]), \Omega_k(Y^*_l)]_e
=[\spi(\bar{Z}), [\spi(X_h), \Omega_k(Y^*_l)]]_e.
\end{equation}
Now, by hypotheses and the $\fl$-equivariance of $\Omega_k$, it follows that 
\begin{equation*}
[\spi(X_h), \Omega_k(Y^*_l)]_e, [\spi(\bar{Z}), \Omega_k(Y^*_l)]_e \in E.
\end{equation*}
As $\bar{Z} \in \bar{\fu}_\fl$,
by (\ref{Eqn7.2.2}), we have $\spi(\bar{Z})_e = 0$.
Thus it follows from Lemma \ref{Lem832} that we have
$[\spi(\bar{Z}), [\spi(X_h), \Omega_k(Y^*_l)]]_e \in E$.
Therefore, $[\spi([\bar{Z}, X_h]), \Omega_k(Y^*_l)]_e \in E$
by (\ref{Eqn2.6.5}).

Next we show that for any $X \in \fg(1)$ and $Y^* \in V^*$,
\begin{equation}\label{Eqn2.6.4}
[\spi(X), \Omega_k(Y^*)]_e \in E.
\end{equation}
Once again since $V^*$ is irreducible, it is given by 
$V^* = \Cal{U}(\fu_\fl)Y^*_l$. As before, it is enough to
show that (\ref{Eqn2.6.4}) holds for 
$Y^* = Z \cdot Y^*_l$ with $Z \in \fu_\fl$.
Since 
$\Omega_k(Z\cdot Y^*_l) = [\spi(Z), \Omega_k(Y^*_l)]$, 
by the Jacobi identity, the commutator 
$[\spi(X), \Omega_k(Z\cdot Y_l^*)]$ is 
\begin{align}\label{Eqn2.6.7}
&[\spi(X), \Omega_k(Z\cdot Y_l^*)] \nonumber \\
&=[\pi_s(Z), [\pi_s(X), \Omega_k(Y^*_l)]]
-[[\pi_s(Z), \pi_s(X)], \Omega_k(Y^*_l)]. 
\end{align}
We showed above that $[\pi_s(X), \Omega_k(Y^*_l)]_e \in E$.
Since we have $\pi_s(Z)_e = 0$ and $[\pi_s(Z), \Omega_k(Y^*_l)]_e\in E$,
by Lemma \ref{Lem832}, the first term of the right hand side of (\ref{Eqn2.6.7})
satisfies
\begin{equation*}
[\pi_s(Z), [\pi_s(X), \Omega_k(Y^*_l)]]_e \in E.
\end{equation*}
Moreover, as $[\pi_s(Z), \pi_s(X)] = \pi_s([Z, X])$ with $[Z,X] \in \fg(1)$,
by what we have shown above, the second term satisfies
\begin{equation*}
[[\pi_s(Z), \pi_s(X)], \Omega_k(Y^*_l)]_e \in E.
\end{equation*}
Hence, 
$[\spi(X), \Omega_k(Z\cdot Y_l^*)]_e  \in E$. 
\end{proof}

%% file: SS_TwoStep.tex
\section
{Parabolic Subalgebras and $\Z$-gradings}\label{chap:TwoStep}

It has been observed in Subsection \ref{SS7211}
that the $\Z$-grading $\fg = \bigoplus_{j=-r}^r\fg(j)$ on $\fg$ 
and parabolic subalgebra $\fq$ play a role to construct
the $\Omega_k$ systems.
In this section we observe these in detail for $\fq$ a
maximal parabolic subalgebra of quasi-Heisenberg type.
The $\Omega_1$ system and $\Omega_2$ systems of those 
parabolic subalgebras will be 
constructed in Section \ref{chap:O1} and Section \ref{chap:O2}, respectively.


\subsection{Parabolic subalgebras of $k$-step nilpotent type}\label{SS23}

Let $\frak{r}$ be any nonzero Lie algebra. 
Put $\frak{r}_0 = \frak{r}$, $\frak{r}_1 = [\frak{r},\frak{r}]$, 
and $\frak{r}_k = [\frak{r},\frak{r}_{k-1}]$ for $k \in \Z_{> 0}$. 
We call $\frak{r}_k$ the \textbf{$k$-th step} of $\frak{r}$ for $k \in \Z_{\geq 0}$.
The Lie algebra $\frak{r}$ is called \textbf{nilpotent} if $\frak{r}_k =  0$ for some $k$,
and it is called \textbf{$k$-step nilpotent} if $\frak{r}_{k-1} \neq 0$ and $\frak{r}_k = 0$.  
In particular, if $[\frak{r},\frak{r}] = 0$ then $\frak{r}$ is called \textbf{abelian}, 
and if $\dim([\frak{r},\frak{r}]) = 1$ 
then $\frak{r}$ is called \textbf{Heisenberg}.
If $[\frak{r},[\frak{r}, \frak{r}]] = 0$ and $\dim([\frak{r},\frak{r}]) > 1$ then
we call $\frak{r}$ \textbf{quasi-Heisenberg}.
Note that $\frak{r}$ is Heisenberg if and only if its center $\fz(\frak{r})$
is one-dimensional.
If the nilpotent radical $\fn$ of a parabolic subalgebra $\fq = \fl \oplus \fn$ is
$k$-step nilpotent (resp. abelian, Heisenberg, or quasi-Heisenberg) 
then we say that $\fq$ is of \textbf{$k$-step nilpotent} 
(resp. \textbf{abelian}, \textbf{Heisenberg}, or \textbf{quasi-Heisenberg}) \textbf{type}.

To build the $\Omega_1$ system and $\Omega_2$ systems
of a maximal parabolic subalgebra $\fq$ of quasi-Heisenberg type,  
it is convenient to classify the parabolic subalgebras $\fq$ of $k$-step nilpotent type
by the subsets of simple roots. 
If $\gb = \sum_{\ga \in \Pi}m_{\ga} \ga \in 
\sum_{\ga \in \Pi}\Z \ga$
then we say that $|m_{\ga}|$ are 
the \emph{multiplicities} of $\ga$ in $\gb$.

\begin{Prop} \label{Prop2.1.10}
Let $\fg$ be a complex simple Lie algebra with highest root $\gamma$,
and $\fq_S = \fl \oplus \fn$ be the parabolic subalgebra of $\fg$ 
that is parametrized by $S$ 
with $S= \{\ga_{i_1}, \ldots, \ga_{i_r} \} \subset \Pi$.
Then $\fn$ is $k$-step nilpotent 
if and only if 
$k = m_{i_1} + m_{i_2} + \cdots + m_{i_r}$, 
where $m_{i_j}$ are the multiplicities of $\ga_{i_j}$ in $\gamma$.
\end{Prop}

\begin{proof}
As this is a well-known fact, we omit a proof.
(For a proof, see for instance Section 3.1 of \cite{KuboThesis}.)
\end{proof}


The following observation would be useful when we consider 
parabolic subalgebras of $k$-step nilpotent type.
First, observe that, by the one-to-one correspondence between
the standard parabolic subalgebras $\fq_S$ and 
the subsets $S \subset \Pi$,
we can associate subdiagrams of Dynkin diagrams
to parabolic subalgebras $\fq_S$. 
The subdiagrams that associates to $\fq_S$
are obtained by deleting the nodes of the Dynkin diagram of $\fg$ that
correspond to the simple roots in $S$,
and the edges in incident on them.
We call such subdiagrams \textbf{deleted Dynkin diagrams}.
With the multiplicities of simple roots in the highest root of $\fg$ in hand,
by Proposition \ref{Prop2.1.10},
we then see the number of steps of nilradical $\fn$ of $\fq_S$
from the deleted Dynkin diagram. 
Table \ref{TMult} shows the multiplicities the simple roots 
in the highest root $\gamma$.
We use the Bourbaki conventions \cite{Bourbaki08} 
for the labels of the simple roots. 

\begin{table}[h]
\caption{Highest Roots}
\begin{center}
\begin{tabular}{cc} 
\hline
Type        & Highest root \\
\hline
$A_n$ &$\ga_1 + \ga_2 + \cdots + \ga_n$\\
$B_n$ &$\ga_1 + 2\ga_2 + 2\ga_3 + \cdots + 2\ga_n$\\
$C_n$ &$2\ga_1 + 2\ga_2  + \cdots + 2\ga_{n-1} +  1\ga_n$\\
$D_n$ &$\ga_1 + 2\ga_2 + 2\ga_3 + \cdots + 2\ga_{n-2} + \ga_{n-1} + \ga_n$\\
$E_6$ &$\ga_1 + 2\ga_2 + 2\ga_3 +  3\ga_4 + 2\ga_5 + \ga_6$\\
$E_7$ &$2\ga_1 + 2\ga_2 + 3\ga_3 +  4\ga_4 + 3\ga_5 + 2\ga_6+\ga_7$\\
$E_8$ &$2\ga_1 + 3\ga_2 + 4\ga_3 +  6\ga_4 + 5\ga_5 + 4\ga_6+3\ga_7+2\ga_8$\\
$F_4$ &$2\ga_1 + 3\ga_2 + 4\ga_3 +  2\ga_4$\\
$G_2$ &$3\ga_1 + 2\ga_2$\\
\hline
\end{tabular} \label{TMult}
\end{center}
\end{table}

Example \ref{Exam2.1.4} below describes the deleted Dynkin diagram
of a given parabolic $\fq_S$ and how we read the diagram.
For simplicity, we depict deleted Dynkin diagrams by crossing out the deleted nodes.

\begin{Example}\label{Exam2.1.4}
Let $\fg = \f{sl}(6, \C)$. 
The Dynkin diagram is
\begin{equation*}
\xymatrix{
\belowwnode{\ga_1}\single[r]&\belowwnode{\ga_2}\single[r]&\belowwnode{\ga_3}\single[r]
&\belowwnode{\ga_4}\single[r]&\belowwnode{\ga_5} } \hskip 0.05in
\end{equation*}

\noindent 
Choose $S = \{ \ga_2, \ga_4 \} \subset \Pi$. 
Then the deleted Dynkin diagram of parabolic subalgebra $\fq_S$
corresponding to the subset $S$ is 
\begin{equation*}
\xymatrix{
\belowwnode{\ga_1}\single[r]&\belowcnode{\ga_2}\single[r]&\belowwnode{\ga_3}\single[r]
&\belowcnode{\ga_4}\single[r]&\belowwnode{\ga_5} } \hskip 0.05in .
\end{equation*}
Moreover, by Table \ref{TMult},
the multiplicity of each simple root in the highest root of $\fg$ is 1.
Thus, $\fq_S$ is a parabolic subalgebra of two-step nilpotent type.
\end{Example}

By the above observation
we often refer to parabolic subalgebras $\fq_S$ by their corresponding
subset $S$ of simple roots. To this end, we are going to define classification 
types of parabolics $\fq_S$.
In Definition \ref{Def3.1.4} below, we mean by classification type $\Cal{T}$ of $\fg$ 
type $A_n$, $B_n$, $C_n$, $D_n$, $E_6$, $E_7$, $E_8$, $F_4$, or $G_2$.

\begin{Def}\label{Def3.1.4}
If $\fg$ is a complex simple Lie algebra of classification type $\Cal{T}$
and if $S$ is a subset of $\Pi$ of simple roots then
we say that a parabolic subalgebra $\fq_S$ of $\fg$ is of \textbf{type $\Cal{T}(S)$},
or \textbf{type $\Cal{T}(i_1, \ldots, i_k)$} if $S = \{\ga_{i_1},\ldots, \ga_{i_k} \}$.
\end{Def}

For example, the parabolic subalgebra $\fq_S$ in Example \ref{Exam2.1.4}
is of type $A_5(2,4)$.
Any maximal parabolic subalgebra is of type $\Cal{T}(i)$
for some $\ga_i \in \Pi$. 


\subsection{Maximal parabolic subalgebra $\fq$ of quasi-Heisenberg type} 
\label{SS41}

Now we observe the $2$-grading 
$\fg = \bigoplus_{j=-2}^2 \fg(j)$ on $\fg$,
that is induced from a maximal parabolic 
subalgebra $\fq$ of quasi-Heisenberg type.

Assume that $\fg$ has rank greater than one and that
$\ga_\fq$ is a simple root, so that the parabolic 
subalgebra $\fq=\fq_{\{\ga_\fq\}} = \fl \oplus \fn$ 
parameterized by $\ga_\fq$ is a maximal parabolic subalgebra
of quasi-Heisenberg type, namely, $[\fn, [\fn, \fn]] = 0$ and 
 $\dim([\fn,\fn]) > 1$.
Let $\IP{\cdot}{\cdot}$ be the inner product induced on $\fh^*$
corresponding to the Killing form $\kappa$.
Write $||\ga||^2 = \IP{\ga}{\ga}$ for $\ga \in \gD$.
The coroot of $\ga$ is $\ga^{\vee} = 2\ga/\IP{\ga}{\ga}$.

Recall from Section \ref{chap:CIS} that 
$\gl_\fq$ denotes the fundamental weight for $\ga_\fq$.
As $\gD(\fl) = \{\ga \in \gD \; | \; \ga \in \text{span}(\Pi \backslash\{\ga_\fq\})\}$
and $\gD(\fn) = \gD^+ \backslash \gD(\fl)$, we have
\beu
\IP{\gl_\fq}{\gb} 
\begin{cases}
= 0 &\text{ if $\gb \in \gD(\fl)$ }\\
> 0 & \text{ if $\gb \in \gD(\fn)$ }.
\end{cases}
\end{equation*}
Observe that if $H_{\gl_\fq} \in \fh$ is defined by
$\kappa(H, H_{\gl_\fq}) = \gl_\fq(H)$ for all $H \in \fh$
and if
\begin{equation}\label{Eqn2.1.14}
H_\fq = \frac{2}{||\ga_\fq||^2}H_{\gl_\fq}
\end{equation}
\noindent then $\gb(H_\fq)$ is the multiplicity of $\ga_\fq$ in $\gb$.
In particular, it follows from Proposition \ref{Prop2.1.10} that
for $\gb \in \gD^+$, $\gb(H_\fq)$ can only take the values of 
$0$, $1$, or $2$.
Therefore,
if $\fg(j)$ denotes the $j$-eigenspace of $\ad(H_\fq)$
then the action of $\ad(H_\fq)$ on $\fg$ induces a 2-grading
\begin{equation*}
\fg = \fg(-2) \oplus \fg(-1) \oplus \fg(0) \oplus \fg(1) \oplus \fg(2)
\end{equation*}
with parabolic subalgebra
\begin{equation*}\label{Eqn3.3.16}
\fq = \fg(0) \oplus \fg(1) \oplus \fg(2).
\end{equation*}
Here we have $\fl = \fg(0)$ and $\fn = \fg(1) \oplus \fg(2)$.
The subalgebra $\bar \fn$, 
the opposite of $\fn$, is given by
\begin{equation*} \label{Eqn3.3.18}
\bar \fn = \fg(-1) \oplus \fg(-2).
\end{equation*}

\vsp

Let $\fl = \fz(\fl) \oplus [\fl,\fl]$ be the decomposition of $\fl$,
that corresponds to $L = Z(L)^\circ L_{ss}$ with $Z(L)^\circ$ the identity component 
of the center of $L$ and $L_{ss}$ the semisimple part of $L$.
We say that a weight $\nu \in \fh^*$ is a highest weight of 
a finite dimensional $L$-module $V$ if 
$\nu|_{\fh_{ss}}$ is a highest weight of $V$ as an $L_{ss}$-module, 
where $\fh_{ss} = \fh  \cap [\fl,\fl]$.
A lowest weight of a finite dimensional $L$-module is similarly defined.

\begin{Prop}\label{Prop3.3.1}
Let $\fq = \fg(0) \oplus \fg(1) \oplus \fg(2)$ be the maximal 
parabolic of quasi-Heisenberg type determined by $\ga_\fq$.
\begin{enumerate}
\item[(1)] The subspace $\fg(1)$ is the irreducible $L$-module with lowest weight $\ga_\fq$.
\item[(2)] The subspace $\fg(2)$ is the irreducible $L$-module with highest weight $\gamma$.
\item[(3)] We have $\fz(\fn) = \fg(2)$.
\end{enumerate}

\end{Prop}

\begin{proof}
Observe that, as $\Ad(L)$ preserves $\fg(j)$, to prove 
the assertions (1) and (2), it suffices to consider $\fg(1)$
and $\fg(2)$ as $\fl$-modules.
For the assertion (1),
the $\fl$-irreducibility of $\fg(1)$ just follows from 
a  well-known fact that, for $\fq = \fg(0) \oplus \bigoplus \fg(j)_{j>0}$ 
with $\fg(1) \neq 0$, $\fg(1)$ is $\fg(0)$-irreducible
if and only if $\fq$ is a maximal parabolic subalgebra.
The lowest weight of $\fg(1)$ follows from Corollary 10.2A of \cite{Hum72}.
For the assertion (2), it is clear that $\Cal{U}(\fg(0))X_\gamma \subset \fg(2)$.
On the other hand, as 
$\bar \fn = \fg(-1)\oplus \fg(-2)$, 
it follows that
$\Cal{U}(\bar \fn)\fg(2) \subset
\bigoplus_{j=-2}^{1}\fg(j)$. 
As $\fg = \bigoplus_{j=-2}^2 \fg(j)
= \Cal{U}(\bar \fn)(\Cal{U}(\fg(0))X_{\gamma})$, 
this shows that $\Cal{U}(\fg(0))X_\gamma \supset \fg(2)$.
To prove assertion (3), since $\fg(2) \subset \fz(\fn)$, 
it suffices to show the other inclusion.
If $X \in \fz(\fn)$ then, as $\fn = \fg(1) \oplus \fg(2)$, 
there exist $X_j \in \fg(j)$ for $j = 1, 2$ so that $X = X_1 + X_2$.
Since $X, X_2 \in \fz(\fn)$, for any $Y \in \fn$, we have
\begin{equation*}
[Y,X_1] = [Y,X_1] + [Y, X_2] =  [Y,X] = 0.
\end{equation*}
\noindent Thus $X_1 \in \fz(\fn)\cap \fg(1)$.
Now observe that the assertion (1) implies 
that $\fz(\fn) \cap \fg(1) = \{ 0 \}$.
Thus $X_1 = 0$ and so $X = X_2 \in \fg(2)$. 
\end{proof}

\vsp

Now, since $\fl= \fg(0)$, $\fg(2) = \fz(\fn)$ and $\fg(-2) = \fz(\bar\fn)$,
we write the 2-grading $\fg = \bigoplus_{j=-2}^2 \fg(j)$ as
\begin{equation}\label{Eqn4.1.6}
\fg = \fz(\bar \fn) \oplus \fg(-1) \oplus \fl \oplus \fg(1) \oplus \cfn
\end{equation}
\noindent with parabolic subalgebra 
\begin{equation}\label{Eqn4.1.7}
\fq = \fl \oplus \fg(1) \oplus \fz(\fn).
\end{equation}


\subsection
{The simple ideals $\flg$ and $\flng$}
\label{SS42}

We next observe the structure of the 
Levi subalgebra $\fl = \fz(\fl) \oplus [\fl,\fl]$.
The structure of $\fl$ will play a role in Section \ref{chap:Sp},
when we decompose $\fl \otimes \fz(\fn)$ into 
irreducible $L$-submodules.

We start with the center $\fz(\fl)$. The center $\fz(\fl)$ is of the form 
$\fz(\fl) = \bigcap_{\ga \in \Pi(\fl)}\ker(\ga)$.
Since $\fg$ has rank greater than one and since $\Pi(\fl) = \Pi \backslash \{\ga_\fq\}$,
the center $\fz(\fl)$ is non-zero and one-dimensional. 
It is clear from (\ref{Eqn2.1.14}) that $H_\fq$ is an element of $\fz(\fl)$. 
Therefore we have $\fz(\fl) = \C H_\fq$. 

Next we consider the structure of $[\fl,\fl]$.
Observe that if $\fg$ is not of type $A_n$ then there is exactly one simple root
that is not orthogonal to $\gamma$.
Let $\ga_\gamma$ denote the unique simple root. 
It is easy to see that $\fq_{\{\ga_\gamma\}}$ is the parabolic subalgebra 
of Heisenberg type of $\fg$; that is, the parabolic subalgebra with 
$\dim([\fn, \fn]) = 1$. Hence, if $\fq_{\{\ga_\fq\}}$ is a maximal
parabolic subalgebra of quasi-Heisenberg type 
then $\ga_\gamma \in \Pi(\fl) = \Pi \backslash \{\ga_\fq\}$.
If we delete the node corresponding to $\ga_\fq$ then we obtain
one, two, or three subgraphs with one subgraph containing $\ga_\gamma$.
This implies that the subalgebra $[\fl,\fl]$ is either simple or the direct sum of two or three 
simple ideals with only one simple ideal containing the root space 
$\fg_{\ga_\gamma}$ for $\ga_\gamma$.
The three subgraphs occur only when
$\fq$ is of type $D_n(n-2)$.
So, if $\fq$ is not of type $D_n(n-2)$ then 
there are at most two subgraphs. In this case
we denote by $\flg$ (resp. $\flng$) the simple ideal of $\fl$
whose subgraph in the deleted Dynkin diagram
contains (resp. does not contain) the node for  $\ga_\gamma$.
Thus the Levi subalgebra $\fl$ may decompose into
\begin{equation}\label{Eqn4.1.8}
\fl = \C H_\fq \oplus \flg \oplus \flng.
\end{equation}
Then, for the rest of this section,
we assume that $\fq$ is not of type $D_n(n-2)$,
so that the Levi subalgebra $\fl$ can be expressed as (\ref{Eqn4.1.8}).
Recall from Definition \ref{Def3.1.4} that
if $\fg$ is of type $\Cal{T}$ then we say that
the parabolic subalgebra $\fq$ determined by $\ga_i \in \Pi$ is of type $\Cal{T}(i)$.
Then the parabolic subalgebras $\fq$
under consideration are given as follows:

\begin{equation} \label{Eqn4.0.1}
B_n(i)\; (3 \leq i \leq n), \quad C_n(i) \; (2 \leq i \leq n-1), \quad D_n(i)\; (3\leq i \leq n-3),
\end{equation}
\noindent and 
\begin{equation}\label{Eqn4.0.2}
E_6(3),\; E_6(5),\; E_7(2),\; E_7(6),\; E_8(1), \;F_4(4).
\end{equation}
Note that, in type $A_n$,  
any maximal parabolic subalgebra is of abelian type,
and also that, in type $G_2$, the two maximal parabolic subalgebras
are of either  3-step type or Heisenberg type.

Write $\Pi(\flg) = \{ \ga \in \Pi \; | \; \ga \in \gD(\flg)\}$
and $\Pi(\flng) = \{ \ga \in \Pi \; | \; \ga \in \gD(\flng)\}$.
Example \ref{Exam4.1.1} below exhibits the subgraphs
for $\flg$ and $\flng$ of $\fq$ of type $B_5(3)$ with $\Pi(\flg)$ and $\Pi(\flng)$.
One can find those data in Appendix \ref{chap:Data} for each maximal parabolic subalgebra 
in (\ref{Eqn4.0.1}) or (\ref{Eqn4.0.2}).

\begin{Example}\label{Exam4.1.1}
Let $\fq$ be the parabolic subalgebra of type $B_5(3)$ with deleted Dynkin diagram
\begin{equation*}
\xymatrix{
\belowwnode{\ga_1}\single[r]&\belowwnode{\ga_2}\single[r]&\belowcnode{\ga_3}\single[r]
&\belowwnode{\ga_4}\rdouble[r]&\belowwnode{\ga_5} } \hskip 0.05in .
\end{equation*}
\vspace{5pt}
Observe that the unique simple root $\ga_\gamma$ that is not orthogonal to 
the highest root $\gamma$ is $\ga_\gamma = \ga_2$. 
Therefore, the subgraph for $\flg$ is 
\begin{equation*}
\xymatrix{
\belowwnode{\ga_1}\single[r]&\belowwnode{\ga_2} }
\end{equation*}
and that for $\flng$ is
\begin{equation*}
\xymatrix{
\belowwnode{\ga_4}\rdouble[r]&\belowwnode{\ga_5} }
\end{equation*}
with $\Pi(\flg) = \{ \ga_1, \ga_2 \}$ and $\Pi(\flng) = \{ \ga_4, \ga_5\}$.
\end{Example}

\begin{Rem}\label{Rem2.1.2}
As $\ga_\gamma$ is the unique simple root that is not orthogonal to $\gamma$,
we have 
$\IP{\gamma}{\ga_\gamma} > 0$ and $\IP{\gamma}{\ga} =0$
for any other simple roots $\ga$. In particular,
$\IP{\ga}{\gamma} = 0$ for all $\ga \in \Pi(\flng)$.
\end{Rem}


\subsection
{The highest weights for $\flg$, $\flng$, $\fg(1)$, and $\fz(\fn)$}
\label{SS44}

For the rest of this section we summarize technical lemmas on 
the $L$-highest weights for $\flg$, $\flng$, $\fg(1)$, and $\fz(\fn)$. 
These technical facts will be used in later computations.

Proposition \ref{Prop3.3.1} shows that
$\fz(\fn)$ has highest weight $\gamma$, which is the highest root of $\fg$.
We denote by $\xig$, $\xing$, and $\mu$
the highest weights for $\flg$, $\flng$, and $\fg(1)$, respectively. 
These highest weights are summarized in Appendix \ref{chap:Data}
for each of the parabolic subalgebras under consideration.
We remark that all these highest weights are indeed roots in $\gD^+$.
Observe that the highest weights $\xig$ and $\xing$ of $\flg$ and $\flng$,
respectively, are also the highest roots of $\flg$ and $\flng$ as simple algebras;
in particular, the multiplicities of $\ga \in \Pi(\flg)$ (resp. $\ga \in \Pi(\flng)$) 
in $\xig$ (resp. $\xing$) are all strictly positive.

\begin{Lem}\label{Lem4.3.1}
If $\ga_\fq$ is the simple root that determines $\fq = \fl \oplus \fg(1) \oplus \fz(\fn)$
then $\xig + \ga_\fq$ and $\xing + \ga_\fq$ are roots.
\end{Lem}

\begin{proof}
We only prove that $\xig + \ga_\fq \in \gD$; the other assertion
that $\xing + \ga_\fq \in \gD$ can be proven similarly.
It suffices to show that $\IP{\xig}{\ga_\fq} < 0$, since both $\xig$ and $\ga_\fq$ are roots.
For $\ga \in \Pi$ we observe that $\IP{\ga}{\ga_\fq} < 0$ if $\ga$ is adjacent to $\ga_\fq$
in the Dynkin diagram and $\IP{\ga}{\ga_\fq} = 0$ otherwise. 
An observation on the deleted Dynkin diagrams shows
that there exists a unique simple root $\ga_k$ in $\Pi(\flg)$ that is adjacent to $\ga_\fq$.
Since $\xig$ is the highest root for $\flg$ as a simple algebra,
the multiplicity of $\ga_k$ in $\xig$ is strictly positive. 
Thus $\IP{\xig}{\ga_\fq}<0$.
\end{proof}


\begin{Lem}\label{Lem4.3.3}
If $\xig$, $\xing$, $\mu$, and $\gamma$ are the highest weights
of $\flg$, $\flng$, $\fg(1)$, and $\fz(\fn)$, respectively,
then the following hold:
\begin{enumerate}
\item[(1)] $\gamma - \xig \in \gD$, but $\gamma - \xing \notin \gD$.
\item[(2)] $\gamma - \mu \in \gD$.
\item[(3)] $\mu-\xig, \mu-\xing \in \gD$.
\end{enumerate}
\end{Lem}

\begin{proof}
To prove $\gamma-\xing \notin \gD$, 
observe that if $n$ and $m$ are the largest non-negative integers 
so that $\gamma -n \xing \in \gD$ and $\gamma+m \xing \in \gD$, respectively,
then $\IP{\gamma}{\xing^\vee} = n-m$.
Since $\IP{\gamma}{\ga^\vee} = 0$ for all $\ga \in \gD(\flng)$,
we have $\IP{\gamma}{\xing^\vee} = 0$ and so $n=m$. 
As $\xing \in \gD^+$ and $\gamma$ is the highest root, 
$\gamma + \xing \notin \gD$. Therefore, $n=m=0$,
which concludes that $\gamma -\xing$ is not a root.
To prove $\gamma-\xig \in \gD$,
it suffices to show that $\IP{\gamma}{\xig} > 0$,
since both $\gamma$ and $\xig$ are roots. 
Write $\xig$ in terms of simple roots in $\Pi(\flg)$.
Observe that each $\ga \in \Pi(\flg)$ has positive multiplicity
$m_\ga$ in $\xig$. As $\gamma$ is orthogonal to $\ga$ for any 
$\ga \in \Pi(\flg) \backslash \{\ga_\gamma\}$, we have
$\IP{\gamma}{\xig} = m_{\ga_\gamma}\IP{\gamma}{\ga_\gamma} > 0$.
The assertions (2) and (3) can be shown similarly.
\end{proof}

The following technical lemma will simplify arguments
concerning the long roots later.
When $\fg$ is simply laced, we regard any root as a long root.

\begin{Lem}\label{Lem3.1.3}
Suppose that $\ga \in \gD$ is a long root.
For any $\gb \in \gD$, the following hold.
\begin{enumerate}
\item[(1)] If $\gb - \ga \in \gD$ then $\IP{\gb}{\cga} =1$.
\item[(2)] If $\gb + \ga \in \gD$  then $\IP{\gb}{\cga} =-1$.
\item[(3)] If $\gb\pm \ga \in \gD$ then $\gb \mp \ga \notin \gD$.
\item[(4)] $\gb \pm 2\ga \notin \gD$.
\end{enumerate}
\end{Lem}

\begin{proof}
These simply follow from the standard arguments using
the structure theory of Lie algebras.
\end{proof}

\begin{Lem}\label{Lem4.3.4}
If $\xig$, $\xing$, $\mu$, and $\gamma$ are the highest weights
of $\flg$, $\flng$, $\fg(1)$, and $\fz(\fn)$, respectively,
then the following hold:
\begin{enumerate}
\item[(1)] $\gamma - \mu + \xing \in \gD$.
\item[(2)] $\gamma - \mu - \xing \notin \gD$.
\item[(3)] If $\xig$ is a long root then $\gamma - \mu \pm \xig \notin \gD$. 
\end{enumerate}
\end{Lem}

\begin{proof}
Lemma \ref{Lem4.3.3} shows that $\gamma - \mu \in \gD$.
Then in order to prove (1), 
it is enough to show that $\IP{\xing}{\gamma-\mu} < 0$.
It follows from Remark \ref{Rem2.1.2} that 
$\IP{\xing}{\gamma} = 0$. On the other hand, 
we have $\IP{\xing}{\mu} > 0$ by the proof for (3) of Lemma \ref{Lem4.3.3}.
Therefore,
\beu
\IP{\xing}{\gamma - \mu} = \IP{\xing}{\gamma} - \IP{\xing}{\mu} < 0.
\end{equation*}

When $\xing$ is a long root of $\fg$, the assertion (2) follows from (1) and Lemma \ref{Lem3.1.3}.
The data in Appendix \ref{chap:Data} shows that
$\xing$ is a long root unless $\fq$ is of type $B_n(n-1)$.
If $\fq$ is of type $B_n(n-1)$ then we have 
$\gamma = \varepsilon_1 + \varepsilon_2$, $\mu = \varepsilon_1 + \varepsilon_n$, and $\xing = \varepsilon_n$.
Thus $\gamma -\mu - \xing  \notin \gD$.

To show (3), observe that, by Lemma \ref{Lem4.3.3},
we have $\gamma - \xig, \; \mu-\xig \in \gD$.
Since $\xig$ is assumed to be a long root,
it follows from Lemma \ref{Lem3.1.3} that 
$\IP{\gamma}{\cxig} = \IP{\mu}{\cxig} = 1$.
Therefore $\IP{\gamma - \mu}{\cxig} = 0$, which forces that 
\begin{equation}\label{Eqn3.3.25}
||\gamma-\mu\pm \xig||^2 = ||\gamma-\mu||^2 + ||\xig||^2.
\end{equation}
\noindent Since $\gamma - \mu$ is a root,
we have $||\gamma-\mu|| \neq 0$.
As $\xig$ is assumed to be a long root,
(\ref{Eqn3.3.25}) implies that $(\gamma - \mu) \pm \xig \notin \gD$.
\end{proof}

\begin{Rem}\label{Rem4.21}
Direct observation shows that $\xig$ is a long root, unless $\fq$ is of type $C_n(i)$.
If $\fq$ is of type $C_n(i)$ then the data in Appendix \ref{chap:Data} shows
$\gamma = 2\varepsilon_1$, $\mu = \varepsilon_1 + \varepsilon_{i+1}$, and $\xig = \varepsilon_1 - \varepsilon_i$.
Thus
$\gamma-\mu + \xig \notin \gD$, but
$\gamma-\mu - \xig \in \gD$.
\end{Rem}

%% file: SS_O1.tex
\section{The $\Omega_1$ System}\label{chap:O1}

The aim of this section is to determine the complex parameter $s_1 \in \C$ for
the line bundle $\Cal{L}_{s}$ so that 
the $\Omega_1$ system of a maximal 
parabolic subalgebra $\fq$ of quasi-Heisenberg type 
is conformally invariant on $\Cal{L}_{s_1}$.
To do so, it is essential to set up convenient normalizations.

If $\ga , \gb \in \gD$ then define
\begin{align} \label{Eqn3.1.1}
p_{\ga, \gb} &= \max\{j \in \Z_{\geq 0} \; | \; \gb - j\ga \in \gD \} \text{ and } \nonumber \\
q_{\ga, \gb} &=\max\{ j \in \Z_{\geq 0} \; | \; \gb + j\ga \in \gD\}. 
\end{align}
\noindent In particular, we have 
\begin{equation}\label{Eqn3.1.2} 
\IP{\gb}{\cga} = p_{\ga,\gb} - q_{\ga,\gb}.
\end{equation}
It is known that we can choose 
$X_\ga \in \fg_\ga$ and $H_\ga \in \fh$ for each $\ga \in \gD$
in such a way that the following conditions hold
(see for instance \cite[Sections III.4 and III.5]{Hel01}).  
The reader may want to notice that our normalizations are different from
those used in \cite{BKZ08}. 

\begin{enumerate}
\item[(H1)] For each $\ga \in \gD^+$,  
$\{X_\ga, X_{-\ga}, H_{\ga} \}$ is an $\f{sl}(2)$-triple;
in particular, we have $[X_\ga, X_{-\ga}] = H_\ga$.
\item[(H2)] For each $\ga, \gb \in \gD^+$, $[H_\ga, X_\gb] = \gb(H_\ga)X_\gb$.
\item[(H3)] For $\ga \in \gD$ we have $\kappa(X_\ga, X_{-\ga})$ = 1.
\item[(H4)] For $\ga,\gb \in \gD$ we have $\gb(H_\ga) =  \IP{\ga}{\gb}$.
\item[(H5)] For $\ga, \gb \in \gD$ with $\ga + \gb \neq 0$, 
there is a constant $N_{\ga,\gb}$ so that 
\begin{align*}
[X_\ga, X_\gb] &= N_{\ga,\gb} X_{\ga + \gb}  \quad \text{ if $\ga + \gb \in \gD$,}\\
N_{\ga, \gb} &= 0 \qquad \qquad \hskip 0.2in \text{if $\ga + \gb \notin \gD$}.
\end{align*}
\item[(H6)] If $\ga_1, \ga_2, \ga_3 \in \gD^+$ with $\ga_1 + \ga_2 + \ga_3 = 0$ then
\beu
N_{\ga_1,\ga_2} = N_{\ga_2,\ga_3}=N_{\ga_3,\ga_1}.
\end{equation*} 
\item[(H7)] If $\ga ,\gb \in \gD$ and $\ga + \gb \in \gD$ then
\beu
N_{\ga,\gb}N_{-\ga,-\gb} = -\frac{q_{\ga,\gb}(1+p_{\ga,\gb})}{2}||\ga||^2
\end{equation*}
In particular, $N_{\ga, \gb}$ is non-zero if $\ga + \gb \in \gD$.
\end{enumerate}

We call the constants $N_{\ga,\gb}$ structure constants. Observe that,
by the normalization (H3), for all $\ga \in \gD$, the vector $X_{-\ga}$ is 
the dual vector for $X_{\ga}$ with respect to the Killing form. Therefore,
in this normalization, the element $\omega = \sum_{\gamma_j \in \gD(\fz(\fn))}
X_{\gamma_j}^* \otimes X_{\gamma_j} \in \fz(\bar \fn) \otimes \fz(\fn)$ is
\begin{equation*} 
\omega = \sum_{\gamma_j \in \gD(\fz(\fn)}
X_{-\gamma_j}\otimes X_{\gamma_j}
\end{equation*}
with $\{X_{-\ga}, H_\ga, X_\ga\}$ an $\f{sl}(2)$-triple.
(See Definition \ref{Def6.3} for $\omega$.)
\vsp

As we have observed in Subsection \ref{SS7211}, we use the covariant map $\tau_1$
and the associated $L$-intertwining operators $\ttau_1|_{V^*}$,
where $V^*$ are irreducible constituents of 
$\fg(-1)^* \otimes \fg(2)^* = \fg(-1)^* \otimes \fz(\fn)^*$.
By Definition \ref{Def6.3},
the covariant map $\tau_1$ is given by
\begin{align*}
\tau_1 : \fg(1) &\to \fg(-1) \otimes \fz(\fn)\\
X &\mapsto \big( \ad(X)\otimes \text{Id} \big) \omega
\end{align*}
with $\omega = \sum_{\gamma_j \in \gD(\fz(\fn))}X_{-\gamma_j} \otimes X_{\gamma_j}$.
It is clear that $\tau_1$ is not identically zero. Indeed, if $X = X_\mu$ with $\mu$ 
the highest weight for $\fg(1)$ then
\begin{equation*}
\tau_1(X_\mu) 
=\big( \ad(X_\mu) \otimes \text{Id} \big) \omega
=\sum_{\gD_\mu(\fz(\fn))}N_{\mu,-\gamma_j}X_{\mu-\gamma_j}\otimes X_{\gamma_j}
\end{equation*}
with $\gD_\mu(\fz(\fn)) = \{\gamma_j \in \gD(\fz(\fn))\; | \; \mu-\gamma_j \in \gD\}$.
By Lemma \ref{Lem4.3.3}, we have $\mu-\gamma \in \gD$ with $\gamma$ the highest weight 
for $\fz(\fn)$, so $\gD_\mu(\fz(\fn)) \neq \emptyset$.
Since the vectors $X_{\mu-\gamma_j} \otimes X_{\gamma_j}$ for 
$\gamma_j \in \gD_\mu(\fz(\fn))$ 
are linearly independent,
we have $\tau_1(X_\mu)\neq 0$.

For each irreducible constituent $V^*$ of $\fg(-1)^* \otimes \fz(\fn)^*$,
there exists an associated $L$-intertwining operator 
$\ttau_1|_{V^*} \in \Hom_L(V^* , \Cal{P}^1(\fg(1)))$ so that,
for all $Y^* \in V^*$,
\begin{equation*}\label{Eqn8.1.11}
\ttau_1|_{V^*}(Y^*)(X) = Y^*(\tau_1(X)).
\end{equation*} 
Observe that the duality for $V^*$ is defined with respect to the Killing form $\kappa$.
Moreover, via the Killing form $\kappa$,
we have $\fg(-1)^* \otimes \fz(\fn)^* \cong \fg(1) \otimes \fz(\bar \fn)$.
Thus, if $Y^* = X_{\ga}\otimes X_{-\gamma_t}$ with $\ga \in \gD(\fg(1))$ and 
$\gamma_t \in \gD(\fz(\fn))$ then $Y^*(\tau_1(X))$ is given by
\begin{equation}\label{Eqn8.1.1}
Y^*(\tau_1(X)) = \sum_{\gamma_j \in \gD(\fz(\fn))}
\kappa(X_\ga, \ad(X)X_{-\gamma_j})\kappa(X_{-\gamma_t}, X_{\gamma_j}),
\end{equation}
as $\tau_1(X) = \sum_{\gamma_j\in \gD(\fz(\fn))} 
\ad(X)X_{-\gamma_j} \otimes X_{\gamma_j}$.

\vsp

Now we wish to determine all the irreducible constituents $V^*$
of $\fg(1) \otimes \fz(\bar\fn)$, so that $\ttau_1|_{V^*}$ are not 
identically zero. Observe that $\Cal{P}^1(\fg(1)) \cong \Sym^1(\fg(-1)) = \fg(-1)$
and that $\fg(-1)$ is an irreducible $L$-module, as $\fq$ is a maximal parabolic
subalgebra.
Thus, if $\ttau_1|_{V^*}$ is not identically zero then $V^* \cong \fg(-1)$.
Proposition \ref{Prop8.1.1} below shows that the converse also holds.
\begin{Prop}\label{Prop8.1.1}
Let $V^*$ be an irreducible constituent of $\fg(1) \otimes \fz(\bar \fn)$.
Then $\ttau_1|_{V^*}$ is not identically zero 
if and only if $V^* \cong \fg(-1)$.
\end{Prop}

\begin{proof}
First observe that $\fg(-1)$ is an irreducible constituent of $\fg(1) \otimes \fz(\bar \fn)$.
Indeed, since $\tau_1$ is linear, we have $\tau_1(\fg(1)) \cong \fg(1)$
as an $L$-module; in particular, $\fg(1)$ is an irreducible constituent of 
$\fg(-1) \otimes \fz(\fn)$. Therefore
$\fg(-1) \cong \fg(1)^*$ is an irreducible constituent of 
$\fg(1) \otimes \fz(\bar \fn)\cong (\fg(-1) \otimes \fz(\fn))^*$.

To prove $\ttau_1|_{\fg(-1)}$ is a non-zero map, it suffices to show that 
$\ttau_1|_{\fg(-1)}(Y^*) \neq 0$ for some $Y^* \in \fg(-1) \subset \fg(1) \otimes \fz(\bar \fn)$.
To do so, consider a map
\begin{align*}
\btau_1 : \fg(-1) &\to \fg(1) \otimes \fz(\bar \fn)\\
\bar{X} &\mapsto \big( \ad(\bar X) \otimes \text{Id} \big) \bar{\omega}
\end{align*}
with $\bar{\omega} =
\sum_{\gamma_t \in \gD(\fz(\fn))}X_{\gamma_t} \otimes X_{-\gamma_t}$.
This is a non-zero $L$-intertwining operator.
Thus $\btau_1(\fg(-1)) \cong \fg(-1)$ as an $L$-module,
and $\btau_1(X_{-\ga})$ is a weight vector with weight $-\ga$
for all $\ga \in \gD(\fg(1))$. 
As $\fg(1)$ has highest weight $\mu$, the lowest weight for $\fg(-1)$ is $-\mu$. 

Now we set
\begin{equation*}
c_\mu = \sum_{\gamma_t \in \gD_\mu(\fz(\fn))} 
N_{-\mu, \gamma_t}N_{\mu, -\gamma_t}
\end{equation*}
with $\gD_\mu(\fz(\fn)) = \{ \gamma_t \in \gD(\fz(\fn))\; | \; \gamma_t - \mu \in \gD\}$.
By Lemma \ref{Lem4.3.3}, it follows that
$\gamma - \mu \in \gD$; in particular, $\gD_\mu(\fz(\fn)) \neq \emptyset$. 
The normalization (H7) shows that 
$N_{-\mu, \gamma_t}N_{\mu, -\gamma_t} < 0 $ for all $\gamma_t \in \gD_\mu(\fz(\fn))$.
Therefore  $c_\mu \neq 0$. Then define $Y^*_l \in \fg(-1)$ by means of
\begin{equation*}
Y^*_l = \frac{1}{c_{\mu}}\btau_1(X_{-\mu})
=\frac{1}{c_{\mu}} 
\sum_{\gamma_t \in \gD_\mu(\fz(\fn))}N_{-\mu, \gamma_t}
X_{\gamma_t -\mu} \otimes X_{-\gamma_t}.
\end{equation*}
We claim that $\ttau_1|_{\fg(-1)}(Y^*_l)(X) \neq 0$.
By (\ref{Eqn8.1.1}), the polynomial $\ttau_1|_{\fg(-1)}(Y^*_l)(X)$ is
\begin{align*}
\ttau_1|_{\fg(-1)}(Y^*_l)(X) 
&= Y^*_l(\tau_1(X)) \nonumber \\
&=\frac{1}{c_\mu} \sum_{\substack{\gamma_t \in \gD_\mu(\fz(\fn)) \nonumber\\
\gamma_j \in \gD(\fz(\fn))}} N_{-\mu, \gamma_t}
\kappa(X_{\gamma_t-\mu}, \ad(X)X_{-\gamma_j})\kappa(X_{-\gamma_t}, X_{\gamma_j}) \nonumber\\
&=\frac{1}{c_\mu} \sum_{\gamma_t \in \gD_\mu(\fz(\fn))}
N_{-\mu, \gamma_t} \kappa(X_{\gamma_t-\mu}, \ad(X)X_{-\gamma_t}).
\end{align*}
Write $X = \sum_{\ga \in \gD(\fg(1))} \eta_\ga X_{\ga}$, where $\eta_\ga \in \fn^*$ is 
the coordinate dual to $X_\ga$ with respect to the Killing form $\kappa$. 
Then,
\begin{align}\label{Eqn7.1.31}
\ttau_1|_{\fg(-1)}(Y^*_l)(X)
&= \frac{1}{c_\mu} \sum_{\gamma_t \in \gD_\mu(\fz(\fn))}
N_{-\mu, \gamma_t} \kappa(X_{\gamma_t-\mu}, \ad(X)X_{-\gamma_t}) \nonumber\\
&= \frac{1}{c_\mu} \sum_{\substack{
\ga \in \gD(\fg(1)) \nonumber\\
\gamma_t \in \gD_\mu(\fz(\fn))}}
N_{-\mu, \gamma_t} \eta_\ga \kappa(X_{\gamma_t-\mu}, \ad(X_\ga)X_{-\gamma_t})\\
&= \frac{1}{c_\mu} \sum_{\gamma_t \in \gD_\mu(\fz(\fn))} 
N_{-\mu, \gamma_t}N_{\mu, -\gamma_t} \eta_\mu \nonumber\\
&= \eta_\mu \nonumber\\
&=\kappa(X, X_{-\mu}).
\end{align}
Hence $\ttau_1|_{\fg(-1)}(Y^*_l)(X) \neq 0$.
\end{proof}

Since only $\fg(-1)$ contributes to the construction of the $\Omega_1$ systems,
we simply refer to the $\Omega_1$ system as the $\Omega_1|_{\fg(-1)}$ system.
As we observed in Subsection \ref{SS7211}, 
the operator $\Omega_1|_{\fg(-1)}: \fg(-1) \to \mathbb{D}(\Cal{L}_{s})^{\bar \fn}$ 
is obtained via the 
composition of maps
\begin{equation*}
\fg(-1)\stackrel{\ttau_1|_{\fg(-1)}}{\to} \Cal{P}^1(\fg(1))
\to \fg(-1) \stackrel{\gs}{\hookrightarrow} 
\Cal{U}(\bar \fn) \stackrel{R}{\to} \mathbb{D}(\Cal{L}_{s})^{\bar \fn}.
\end{equation*}
By (\ref{Eqn7.1.31}), we have
$\ttau_1|_{\fg(-1)}(Y^*_l)(X) = \kappa(X, X_{-\mu})$.
Therefore,
\begin{equation*}\label{Eqn8.1.4}
\Omega_1(Y^*_l) = R(X_{-\mu}).
\end{equation*}
Now, for all $\ga \in \gD(\fg(1))$, set 
\begin{equation*}
Y_{-\ga} = \btau_1(X_{-\ga}).
\end{equation*}
Then, as $Y^*_l = (1/c_{\mu}) \btau_1(X_{-\mu})$, we have
$\Omega_1(Y_{-\mu}) = c_\mu R(X_{-\mu})$.
Since both $\Omega_1|_{\fg(-1)}$ and $\bar{\tau}_1$ are $L_0$-intertwining operators
and since $\fg(-1) = \Cal{U}(\fl) X_{-\mu}$, for any $\ga \in \gD(\fg(1))$, we obtain
\begin{equation}\label{Eqn8.1.5}
\Omega_1(Y_{-\ga}) = c_\ga R(X_{-\ga})
\end{equation}
with some constant $c_\ga$. 
Thus, if $\gD(\fg(1)) = \{\ga_1, \ldots, \ga_m\}$ then the $\Omega_1$ system is given by
\begin{equation*}
R(X_{-\ga_1}), \ldots, R(X_{-\ga_m}).
\end{equation*}

\begin{Thm}\label{Thm8.1.2}
Let $\fg$ be a complex simple Lie algebra, 
and let $\fq$ be a maximal parabolic subalgebra 
of quasi-Heisenberg type. 
Then the $\Omega_1$ system
is conformally invariant on $\Cal{L}_{s}$ if and only if $s=0$.
\end{Thm}

\begin{proof}
By Remark \ref{Rem2.3.6},
we only need to show that 
the condition (S2) in Definition \ref{Def1.1.3} holds
if and only if $s=0$.
By Theorem \ref{Prop7.2.5}, 
for any $Y \in \fg$ and any $f \in C^\infty(\bar{N_0} , \C_{\chi^{s}})$,
we have
\begin{align*}
&\big( [\pi_s(Y), R(X_{-\ga_j})] \acts f \big)(\nbar)\\
&=\big( R([(\Ad(\nbar^{-1})Y)_{\fq}, X_{-\ga_j}]_{\bar{\fn}}) \acts f \big)(\nbar)
-s\gl_\fq\big([\Ad(\nbar^{-1})Y, X_{-\ga_j}]_{\fq}\big)f(\nbar).
\end{align*}
Hence, the condition (S2) holds if and only if $s=0$.
\end{proof}

%% file: SS_Special.tex
\section
{Special Constituents of $\fl \otimes \cfn$} 
\label{chap:Sp}

Our next goal is to construct the $\Omega_2$ systems and to find their special values.
To do so, we need to detect the irreducible constituents $V^*$ of 
$\fl^* \otimes \fz(\fn)^*$ so that $\ttau_2|_{V^*}$ is not identically zero.
(See Subsection \ref{SS7211} for the general construction of the $\Omega_k$ systems.)
In this section we shall show preliminary results to find such irreducible constituents.


\subsection
{Irreducible decomposition of $\flg \otimes \cfn$}
\label{SS51}

We continue with $\fq = \fl \oplus \fg(1) \oplus \fz(\fn)$ a maximal 
parabolic subalgebra of quasi-Heisenberg type
listed in (\ref{Eqn4.0.1}) or (\ref{Eqn4.0.2}), 
and $Q = LN = N_G(\fq)$. The Levi subgroup $L$ acts on 
$\fl \otimes \fz(\fn) \subset \fg \otimes \fg$ via the standard action
on the tensor product induced by 
the adjoint representation on $\fl$ and $\fz(\fn)$. 
As $L$ is complex reductive, this action is completely reducible.
Since $\fl = \fz(\fl) \oplus \flg \oplus \flng$ with $\fz(\fl) = \C H_\fq$,
we have
\begin{equation} \label{Eqn5.1.1}
\fl \otimes \fz(\fn) 
= \big( \C H_\fq \otimes \fz(\fn) \big) 
\oplus \big( \flg \otimes \fz(\fn) \big)
\oplus \big( \flng \otimes \fz(\fn) \big).
\end{equation}
It is clear that $\C H_\fq \otimes \fz(\fn) \cong \fz(\fn)=\fg(2)$ as an $L$-module.
Thus, by Corollary \ref{Prop3.3.1}, 
$\C H_\fq \otimes \fz(\fn)$ is $L$-irreducible. 
It is also easy to show that $\flng \otimes \fz(\fn)$ is $L$-irreducible. 
Let $\Lg$ (resp. $\Lng$) be the analytic subgroup of $L$
with Lie algebra $\flg$ (resp. $\flng$).
As in Subsection \ref{SS41},
we call a weight $\nu$ for a finite dimensional $L$-module $V$  
a highest weight for $V$ if the restriction $\nu|_{\fh_{ss}}$ onto 
$\fh_{ss}$ is a highest weight for $V$ as an $L_{ss}$-module.

\begin{Prop}\label{Lem5.3.1}
Suppose that $\flng \neq 0$.
If $\xing$ and $\gamma$ are the highest weights of $\flng$ and $\fz(\fn)$,
respectively, then $\flng \otimes \fz(\fn)$ is the irreducible $L$-module 
with highest weight $\xing + \gamma$.
\end{Prop}

\begin{proof}
First we observe that $\Lng$ acts trivially on $\fz(\fn)$. 
By Corollary \ref{Prop3.3.1}, we have 
$\fz(\fn) = \fg(2) = \Cal{U}([\fl,\fl])X_\gamma$.
By the observation made in Remark \ref{Rem2.1.2},
it follows that $\ga \perp \gamma$ for all $\ga \in \gD(\flng)$.
Thus $\fz(\fn) = \Cal{U}(\flg)X_\gamma$. Hence $\Lng$ acts trivially;
in particular, the irreducible $L$-module $\fz(\fn)$ is $\Lg$-irreducible.
On the other hand, it is clear that $\Lg$ acts on $\flng$ trivially.
Therefore the representation $(L, \Ad \otimes \Ad, \flng \otimes \fz(\fn))$ 
is equivalent to $(\Lg \times \Lng, \Ad \hat{\otimes} \Ad,  \flng \otimes \fz(\fn))$,
where $\hat{\otimes}$ denotes the outer tensor product.
Since $\flng$ and $\fz(\fn)$ have highest weight $\xing$ and $\gamma$, 
respectively, the lemma follows.
\end{proof}

Now we focus on the irreducible decomposition of $\flg \otimes \fz(\fn)$.
As noted in the proof for Proposition \ref{Lem5.3.1}, the subgroup $\Lng$
acts trivially on $\flg \otimes \fz(\fn)$.
Hence we study $\flg \otimes \fz(\fn)$ as an $\Lg$-module.
For $\gl \in \fh^*$ with $\IP{\gl}{\cga} \in \Z_{\geq 0}$ 
for all $\ga \in \Pi(\flg)$,
we will denote by 
$V(\gl)$ the irreducible constituent with highest weight 
$\gl|_{\fh_{\gamma}}$, where $\fh_{\gamma} = \fh \cap \flg$.
For classical algebras, we use the standard realization of the roots 
$\varepsilon_i$, the dual basis of the standard orthonormal basis for $\R^n$.
For exceptional algebras the Bourbaki conventions are used to label 
the simple roots.

\begin{Thm} \label{Thm5.1}
The $L$-module $\flg \otimes \fz(\fn)$ is reducible.
If $V(\gl)$ denotes the irreducible representation of $L$ 
with highest weight $\gl|_{\fh_\gamma}$
then the irreducible decomposition of $\flg \otimes \fz(\fn)$ is given as follows.
\end{Thm}

\noindent (1) $B_n(i), \; 3\leq i \leq n:$\\
$\begin{cases}
V( \xig + \gamma) \oplus 
V( \gamma) \oplus
V(\xig + (\varepsilon_1 + \varepsilon_3))  & \text{if $i = 3$}\\
V( \xig + \gamma) \oplus
V(\gamma) \oplus
V( \xig + (\varepsilon_1 + \varepsilon_i))\oplus
V( \xig + (\varepsilon_2 +\varepsilon_3) ) &\text{if $4 \leq i \leq n$}\\
\end{cases}$
\vskip 0.15in

\noindent (2) $C_n(i), \; 2\leq i \leq n-1:$\\
$\begin{cases}
V( \xig + \gamma) \oplus 
V( \gamma) \oplus
V(\xig + 2\varepsilon_2)  & \text{if $i = 2$}\\
V( \xig + \gamma) \oplus
V(\gamma) \oplus
V( \xig + (\varepsilon_2 + \varepsilon_i))\oplus
V( \xig + (\varepsilon_1 +\varepsilon_2) ) &\text{if $3 \leq i \leq n-1$}\\
\end{cases}$
\vskip 0.15in

\noindent (3) $D_n(i), \; 3\leq i \leq n-3:$\\
$\begin{cases}
V( \xig + \gamma) \oplus 
V( \gamma) \oplus
V(\xig + (\varepsilon_1 + \varepsilon_3))  & \text{if $i = 3$}\\
V( \xig + \gamma) \oplus
V(\gamma) \oplus
V( \xig + (\varepsilon_1 + \varepsilon_i))\oplus
V( \xig + (\varepsilon_2 +\varepsilon_3) ) &\text{if $4 \leq i \leq n-3$}\\
\end{cases}$

\vskip 0.15in

\noindent (4) All exceptional cases 
($E_6(3)$, $E_6(5)$, $E_7(2)$, $E_7(6)$, $E_8(1)$, $F_4(4)$): 
\beu
V( \xig + \gamma) \oplus
V(\gamma) \oplus
V( \xig + \gamma_0),
\end{equation*}
\noindent where $\gamma_0$ is the following root contributing to $\fz(\fn)$:
\begin{enumerate}
\item[] $E_6(3):$ $\gamma_0 = \ga_1 + \ga_2 + 2\ga_3 + 3\ga_4 + 2\ga_5 + \ga_6$
\item[] $E_6(5):$ $\gamma_0 = \ga_1 + \ga_2 + 2\ga_3 + 3\ga_4 + 2\ga_5 + \ga_6$
\item[] $E_7(2):$ $\gamma_0 = \ga_1 + 2\ga_2 + 3\ga_3 + 4\ga_4 + 3\ga_5 + 2\ga_6 + \ga_7$
\item[] $E_7(6):$ $\gamma_0 = \ga_1 + 2\ga_2 + 2\ga_3 + 4\ga_4 + 3\ga_5 + 2\ga_6 + \ga_7$
\item[] $E_8(1):$ $\gamma_0 = 2\ga_1 + 3\ga_2 + 4\ga_3 + 6\ga_4 + 5\ga_5 + 4\ga_6 + 2\ga_7 + \ga_8$
\item[] $F_4(4):$ $\gamma_0 = \ga_1 + 2\ga_2 + 4\ga_3 + 2\ga_4$.
\end{enumerate}

\begin{proof}
To prove this theorem we just use the standard character formula 
due to Klimyk (\cite[Corollary]{Klimyk68})
for $\flg \otimes \fz(\fn)$. 
For the details, see Chapter 5 of \cite{KuboThesis}.
\end{proof}

\subsection
{Special constituents}\label{SS540}
Given irreducible constituent $V(\nu)$ in $\fl \otimes \fz(\fn)$,
we build $L$-intertwining map
\begin{equation*}
\ttau_2|_{V(\nu)^*} \in \Hom_{L}(V(\nu)^*, \Cal{P}^2(\fg(1)))
\end{equation*}
with $V(\nu)^*$ the dual of $V(\nu)$ 
with respect to the Killing form $\kappa$.
From $\ttau_2|_{V(\nu)^*}$, we construct 
operator $\Omega_2|_{V(\nu)^*}: V(\nu)^* \to \mathbb{D}(\Cal{L}_{s})^{\bar \fn}$.
To do so, it is necessary to determine
which irreducible constituents $V(\nu)$ have the property 
that $\ttau_2|_{V(\nu)^*} \neq 0$.
Thus, next, by using the above decomposition results,
we shall determine such irreducible constituents.

First we observe the vector space isomorphism
$\Cal{P}^2(\fg(1)) \cong \Sym^2(\fg(1))^*$.
With the natural $L$-action on $\Cal{P}^2(\fg(1))$
and $\Sym^2(\fg(1))^*$, this vector space isomorphism is 
$L$-equivariant.  Thus, if $\ttau_2\big|_{V(\nu)^*}$ is 
a non-zero map
then $V(\nu)$ is an irreducible constituent of 
$\Sym^2(\fg(1)) \subset \fg(1) \otimes \fg(1)$;
in particular,
the weight $\nu$ is of the form $\nu=\mu+\ge$ for some $\ge \in \gD(\fg(1))$,
where $\mu$ is the highest weight of $\fg(1)$.

One can see from the decompositions in Theorem \ref{Thm5.1} that 
$V(\gamma)$ is an irreducible constituent of $\fl \otimes \fz(\fn)$
for any $\fq$ under consideration.
By Lemma \ref{Lem4.3.3}, we have
$\gamma = \mu + \ge$ for some $\ge \in \gD(\fg(1))$.
Now we claim that $\ttau_2|_{V(\gamma)^*}$ is identically zero.
It is well-known that
\begin{equation}\label{Eqn6.2.1}
\fg(1) \otimes \fg(1) = \Sym^2(\fg(1)) \oplus \wedge^2(\fg(1))
\end{equation}
as an $L$-module.
Since each weight space for $\fg(1)$ is one-dimensional
as weights for $\fg(1)$ are roots of $\fg$,
the $L$-module decomposition (\ref{Eqn6.2.1}) is multiplicity free.

\begin{Prop}\label{Prop3.3.20}
The $L$-module $V(\gamma)$ is an irreducible constituent of $\land^2(\fg(1))$.
\end{Prop}

\begin{proof}
Define a linear map $\varphi:\fz(\fn) \to \wedge^2(\fg(1))$ by means of
\begin{equation*}
\varphi(W) = \sum_{\gb \in \gD(\fg(1))}\ad(W)X_{-\gb}\wedge X_\gb.
\end{equation*}
By using an argument similar to that for Lemma \ref{Lem6.5},
one can show that $\varphi$ is $L$-equivariant.
Then, since $\fz(\fn) \cong V(\gamma)$ as an irreducible $L$-module,
it suffices to show that $\varphi$ is a non-zero map.
Write
$\gD_\gamma(\fg(1)) = \{\gb \in \gD(\fg(1)) \; | \; \gamma-\gb \in \gD\}$.
By Lemma \ref{Lem4.3.3}, we have $\gamma - \mu \in \gD$.
Hence $\gD_\gamma(\fg(1)) \neq \emptyset$. 
By writing $\gb' = \gamma - \gb$ for $\gb \in \gD_\gamma(\fg(1))$,
$\varphi(X_\gamma)$ is given by
\begin{equation*}
\varphi(X_\gamma)
=\sum_{\gb \in \gD(\fg(1))}\ad(X_\gamma)X_{-\gb}\wedge X_\gb
=\sum_{\gb \in \gD_\gamma(\fg(1))}N_{\gamma, -\gb}X_{\gb'}\wedge X_\gb.
\end{equation*}
Observe that for each $\gb \in \gD_\gamma(\fg(1))$, we have 
$\gamma-\gb \in \gD_\gamma(\fg(1))$. 
Moreover, by the normalization (H6) of our normalizations in Section \ref{chap:O1},
it follows that $N_{\gamma, -\gb'} = -N_{\gamma, -\gb}$.
Therefore,
\begin{equation}\label{EqnPhi}
N_{\gamma, -\gb}X_{\gb'} \wedge X_{\gb}
+ N_{\gamma, -\gb'}X_{\gb} \wedge X_{\gb'}
=2N_{\gamma, -\gb}X_{\gb'} \wedge X_{\gb}.
\end{equation}
Since $N_{\gamma, -\gb} \neq 0$ for $\gb \in \gD_\gamma(\fg(1))$,
equation (\ref{EqnPhi}) is non-zero. On the other hand,
if $\gb \in \gD_\gamma(\fg(1))$ and $\eta \in \gD_\gamma(\fg(1))$
is so that $\eta \neq \gb, \gb'$ then 
$X_{\gb'} \wedge X_{\gb}$ and $X_{\eta} \wedge X_{\gb}$ are linearly independent.
Hence, $\varphi(X_\gamma) \neq 0$.
\end{proof}

\begin{Def}\label{Def5.2.1}
An irreducible constituent $V(\nu)$
of $\fl \otimes \fz(\fn)$ is called  \textbf{special} 
if $\nu \neq \gamma$ and there exists $\ge \in \gD(\fg(1))$ so that 
$\nu = \mu + \ge$, where $\mu$ and $\gamma$ are 
the highest weights for  $\fg(1)$ and $\fz(\fn)$, respectively.
\end{Def}

\begin{Prop}\label{Prop6.18}
Let $V(\nu)$ be an irreducible constituent of $\fl \otimes \fz(\fn)$.
Then $\ttau_2\big|_{V(\nu)^*}$ is not identically zero only if 
$V(\nu)$ is a special constituent of $\fl\otimes \fz(\fn)$.
\end{Prop}

\begin{proof}
At the beginning of this section we observed that 
if $\ttau_2|_{V(\nu)^*} \neq 0$ then $\nu$ must be of the form
$\nu = \mu + \ge$ for some $\ge \in \gD(\fg(1))$.
Then $V(\nu)$ is either a special constituent or $V(\gamma)$ 
(by Lemma \ref{Lem4.3.3}, $\gamma$ satisfies the form).
However, by Proposition \ref{Prop3.3.20}, it follows that 
$\ttau_2|_{V(\gamma)^*}$ is identically zero. Therefore,
$V(\nu)$ must be a special constituent.
\end{proof}


\vsp

Now we determine all the special constituents of $\fl \otimes \fz(\fn)$.
Since $\fl \otimes \fz(\fn) = 
(\C H_\fq \otimes \fz(\fn) ) \oplus ([\fl,\fl] \otimes \fz(\fn))$
and $\C H_\fq \otimes \fz(\fn) = V(\gamma)$,
it suffices to consider 
$[\fl,\fl] \otimes \fz(\fn) = (\flg \otimes \fz(\fn)) \oplus (\flng \otimes \fz(\fn))$. 
We start by observing that, by Proposition \ref{Lem5.3.1},
$\flng \otimes \fz(\fn) = V(\xing + \gamma)$.

\begin{Prop}\label{Prop5.4.1}
Suppose that $\flng \neq 0$. Then
the irreducible constituent $V(\xing + \gamma)$
is special.
\end{Prop}

\begin{proof}
We need to show that $\xing + \gamma = \mu + \gb$ for some $\gb \in \gD(\fg(1))$.
This is precisely the statement (1) of Lemma \ref{Lem4.3.4}.
\end{proof}

We next consider
the constituent $V(\xig + \gamma)$ of 
$\flg \otimes \fz(\fn) = V(\xig) \otimes V(\gamma)$.

\begin{Lem}\label{Lem5.4.2}
The irreducible constituent $V(\xig +\gamma)$ of $\flg \otimes \fz(\fn)$ 
is not special.
\end{Lem}

\begin{proof}
Lemma \ref{Lem4.3.4} and Remark \ref{Rem4.21} show that
$\xig + \gamma - \mu \notin \gD(\fg(1))$, which implies that
$\xig + \gamma \neq \mu + \gb$ for all $\gb \in \gD(\fg(1))$.
\end{proof}

We determine all the special constituents of $\flg \otimes \fz(\fn)$
in two steps. First we assume that $\fg$ is a classical algebra, and
then consider the case that $\fg$ is an exceptional algebra.

For classical cases the parabolic subalgebras $\fq$ under consideration
are of type 
$B_n(i)\; (3 \leq i \leq n)$, $C_n(i) \; (2 \leq i \leq n-1)$, or $D_n(i)\; (3\leq i \leq n-3)$.
It will be convenient to write $\gb \in \gD(\fg(1))$
in terms of the fundamental weights of $\flg$ and $\flng$.
It is clear from the deleted Dynkin diagrams
that, for each of the cases, 
$\Pi(\flg)$ and $\Pi(\flng)$ are given by
\begin{equation*}
\Pi(\flg) = \{\ga_r \; | \; 1 \leq r \leq i-1\}
\quad \text{and} \quad
\Pi(\flng) =\{\ga_{i+s} \; | \; 1 \leq s \leq n-i\}, 
\end{equation*}
where $\ga_j$ are the simple roots with the standard numbering.
By using the standard realizations of roots,
we have $\ga_r = \varepsilon_r - \varepsilon_{r+1}$ for $1 \leq r \leq i-1$,
$\ga_{i+s} = \varepsilon_{i+s} - \varepsilon_{i+s+1}$ for $1 \leq s \leq n-i-1$, and 
\begin{equation*}
\ga_n = 
\begin{cases}
\varepsilon_n & \text{if $\fg$ is of type $B_n$}\\
2\varepsilon_n & \text{if $\fg$ is of type $C_n$}\\
\varepsilon_{n-1} + \varepsilon_n & \text{if $\fg$ is of type $D_n$.}
\end{cases}
\end{equation*}
The data in Appendix \ref{chap:Data} shows that 
if $\fq$ is of type $B_n(i)$ then
\begin{equation*}
\gD(\fg(1)) =
\{\varepsilon_j \pm \varepsilon_k \; | \; 1 \leq j \leq i \text{ and } i+1 \leq k \leq n\}
 \cup \{\varepsilon_j \; | \; 1\leq j \leq i \}
\end{equation*}
and if $\fq$ is of type $C_n(i)$ or $D_n(i)$ then
\begin{equation*}
\gD(\fg(1)) =
\{\varepsilon_j \pm \varepsilon_k \; | \; 1 \leq j \leq i \text{ and } i+1 \leq k \leq n\}.
\end{equation*}

\noindent 
Since we have two simple algebras $\flg$ and $\flng$, we use the notation
$\varpi_r$ for the fundamental weights of $\ga_r \in \Pi(\flg)$
and $\tilde{\varpi}_s$ for those of $\ga_{i+s} \in \Pi(\flng)$.
Direct computation then shows that each $\gb \in \gD(\fg(1))$ is
exactly one of the following forms:
\begin{align}\label{Eqn5.17}
\gb = 
\begin{cases}
\varpi_1 + \sum_{s=1}^{n-i} \tilde{m}_s \tilde{\varpi}_s, \\
(-\varpi_r +\varpi_{r+1}) + \sum_{s=1}^{n-i} \tilde{m}_s \tilde{\varpi}_s\;
          \text{with $1 \leq r \leq i-2$, or}  \\
 -\varpi_{i-1} + \sum_{s=1}^{n-i} \tilde{m}_s\tilde{\varpi}_s 
\end{cases}
\end{align}
for some $\tilde{m}_s \in \Z$.

\begin{Prop}\label{Prop5.4.3}
Let $V(\nu)$ be an irreducible constituent of $\flg \otimes \cfn$.
\begin{enumerate}
\item[(1)] If $\fq$ is of type $B_n(i)\; (3 \leq i \leq n)$ or $D_n(i) \; (3\leq i \leq n-3)$ 
then $V(\nu)$ is a special constituent if and only if $\nu = 2\varepsilon_1$.
\item[(2)] If $\fq$ is of type $C_n(i) \; (2 \leq i \leq n-1)$ 
then $V(\nu)$ is a special constituent if and only if $\nu = \varepsilon_1 + \varepsilon_2$.
\end{enumerate}
\end{Prop}

\begin{proof}
Suppose that $\fq$ is of type $B_n(i)$, $C_n(i)$, or $D_n(i)$.
By Definition \ref{Def5.2.1}, we need to find all $\nu$ 
of the form $\nu = \mu + \gb$ for some $\gb \in \gD(\fg(1))$.
Here $\mu$, the highest weight for $\fg(1)$, is 
\begin{equation*}
\mu = 
\begin{cases}
\varepsilon_1 + \varepsilon_{i+1} &\quad \text{if $\fq$ is of type $B_n(i)$ with $i \neq n$, $C_n(i)$, or $D_n(i)$} \\
\varepsilon_1 &\quad \text{if $\fq$ is of type $B_n(n)$}.
\end{cases}
\end{equation*}
\noindent 
We write $\mu$ in terms of the fundamental weights of $\flg$ and $\flng$; that is,
\begin{equation}\label{Eqn5.19}
\mu =
\begin{cases}
\varpi_1 + \tilde{\varpi}_1 &\quad \text{if $\fq$ is of type $B_n(i)$ 
                                        with $i \neq n$, $C_n(i)$, or $D_n(i)$} \\
\varpi_1 &\quad \text{if $\fq$ is of type $B_n(n)$},
\end{cases}
\end{equation}
where $\varpi_1$ and $\tilde{\varpi}_1$ are the fundamental weights of
$\ga_1 = \varepsilon_1-\varepsilon_2$ and $\ga_{i+1} = \varepsilon_{i+1} - \varepsilon_{i+2}$, respectively.
As $\flng$ acts trivially on both $\flg$ and $\fz(\fn)$,
the highest weight $\nu$ for a constituent $V(\nu) \subset \flg \otimes \fz(\fn)$
is of the form
\begin{equation}\label{Eqn5.21}
\nu = \sum_{j = 1}^{i-1} n_j \varpi_j \quad \text{for $n_j \in \Z_{\geq 0}$.}
\end{equation}
\noindent
If there exists $\gb \in \gD(\fg(1))$ so that $\nu = \mu + \gb$ then
(\ref{Eqn5.19}) and (\ref{Eqn5.21}) imply that $\gb = \nu-\mu$ is of the form
\begin{equation} \label{Eqn5.2.4}
\begin{cases}
(n_1-1)\varpi_1 + \sum_{j = 2}^{i-1} n_j \varpi_j - \tilde{\varpi}_1 
& \text{if $\fq$ is of type $B_n(i)$ $i \neq n$, $C_n(i)$, or $D_n(i)$} \\
(n_1-1)\varpi_1 + \sum_{j = 2}^{i-1} n_j \varpi_j
& \text{if $\fq$ is of type $B_n(n)$}
\end{cases}
\end{equation}
for $n_j \in \Z_{\geq 0}$.
On the other hand, we observed that the root $\gb$ must be 
one of the forms in (\ref{Eqn5.17}).
Then observation shows that 
if $\gb$ satisfies both (\ref{Eqn5.17}) and (\ref{Eqn5.2.4}) then
$\gb$ must be
\begin{equation*}
\begin{cases}
\varpi_1 - \tilde{\varpi}_1 \text{ or } (-\varpi_1 + \varpi_2) - \tilde{\varpi}_1
& \text{if $\fq$ is of type $B_n(i)$ $i \neq n$, $C_n(i)$, or $D_n(i)$} \\
\varpi_1 \text{ or } (-\varpi_1 + \varpi_2)
& \text{if $\fq$ is of type $B_n(n)$}.
\end{cases}
\end{equation*}
Therefore $\nu = \mu + \gb$ is 
$\nu = 2\varpi_1 \text{ or } \varpi_2$,
which shows that $\nu = 2\varepsilon_1$ or $\varepsilon_1 + \varepsilon_2$.
As $\xig = \varepsilon_1 - \varepsilon_{i}$ for $\fq$ 
of type $B_n(i)$, $C_n(i)$, or $D_n(i)$, 
Theorem \ref{Thm5.1} shows that
both $V(2\varepsilon_1)$ and $V(\varepsilon_1 + \varepsilon_2)$
occur in  $\flg \otimes \fz(\fn)$.
Now the assertions follow from the fact that
the highest root $\gamma$ of $\fg$ is 
$\gamma = \varepsilon_1 + \varepsilon_2$ if $\fg$ is of type $B_n$ or $D_n$, 
and $\gamma = 2\varepsilon_1$ if $\fg$ is of type $C_n$.
\end{proof}

If $\fg$ is an exceptional algebra
then the parabolic subalgebras $\fq$ 
under consideration are 
\begin{equation}\label{Eqn5.4.4}
E_6(3), E_6(5), E_7(2), E_7(6), E_8(1), \text{and } F_4(4).
\end{equation}

\begin{Lem}\label{Lem5.4.4}
If $\fq$ is of exceptional type as in (\ref{Eqn5.4.4}) then
$V(\xig + \gamma_0)$ in Theorem \ref{Thm5.1}
is a special constituent.
\end{Lem}

\begin{proof}
This is done by a direct computation.
The roots $\geg$ in $\gD(\fg(1))$ so that 
$\xig+\gamma_0 = \mu+\geg$ are given in Table \ref{T54} below.
\end{proof}


\begin{Prop}\label{Prop5.4.5}
There exists a unique special constituent in $\flg \otimes \fz(\fn)$. 
\end{Prop}

\begin{proof}
If $\fq$ is of classical type then this proposition follows from  
Proposition \ref{Prop5.4.3}.
For $\fq$ of exceptional type,
by Theorem \ref{Thm5.1},
the tensor product $\flg \otimes \fz(\fn)$ decomposes into
\begin{equation*}
\flg \otimes \fz(\fn) = 
V( \xig + \gamma) \oplus
V(\gamma) \oplus
V( \xig + \gamma_0)
\end{equation*}
with $\gamma_0 \in \gD(\fn)$ as in Theorem \ref{Thm5.1}.
Then Lemma \ref{Lem5.4.2} and Lemma \ref{Lem5.4.4} show
that $V(\xig + \gamma_0)$ is the unique special constituent.  
\end{proof}

Since the weight $\ge \in \gD(\fg(1))$ so that $\mu + \ge$ 
is the highest weight of a special constituent will play a role later,
we introduce the notation related to $\ge$.

\begin{Def}\label{Def5.2.7}
We denote by $\geg$ the root contributing to $\fg(1)$ 
so that $V(\mu + \ge_\gamma)$ 
is the special constituent of $\flg \otimes \cfn$. 
Similarly, we denote by $\geng$ the root for $\fg(1)$ so that
$V(\mu +\geng) = \flng \otimes \cfn$.
\end{Def}

We summarize data on the special constituents 
in Table \ref{T51}, Table \ref{T52}, Table \ref{T53}, and Table \ref{T54} below.
A dash indicates that no special constituent of the type exists for the case.

By Proposition \ref{Prop6.18},
only special constituents could contribute to 
the construction of the $\Omega_2$ systems. 
Next we want to show that 
$\ttau_{2}|_{V^*} \neq 0$ when $V$ is a special constituent.
An observation on the highest weights for the
special constituents will simplify the argument. 
We classify them
by their highest weights and call them type 1a, type 1b, type 2, and type 3.

\begin{Def}\label{Def5.2.8}
Let $\mu$ be the highest weight for $\fg(1)$, and let 
$\ge = \geg$ or $\ge = \geng$. (See Definition \ref{Def5.2.7}.)
We say that a special constituent $V(\mu+\ge)$ of $\fl \otimes \fz(\fn)$ is of \\
\noindent (1) \textbf{type 1a} if $\mu + \ge$ is not a root with $\ge \neq \mu$ and both 
$\mu$ and $\ge$ are long roots,\\
\noindent (2) \textbf{type 1b} if $\mu + \ge$ is not a root with $\ge \neq \mu$ and either
$\mu$ or $\ge$ is a short root,\\
\noindent (3) \textbf{type 2} if $\mu + \ge = 2\mu$ is not a root, or\\
\noindent (4) \textbf{type 3} if $\mu + \ge$ is a root.
\end{Def}

\begin{table}[h]
\caption{Highest Weights for Special Constituents (Classical Cases)}
\begin{center}
\begin{tabular}{ccccc} 
\hline
Type        &$V(\mu + \geg)$ &$V(\mu + \geng)$ \\
\hline
$B_n(i),\;  3\leq i \leq n-2$ &$2\varepsilon_1$&$\varepsilon_1 + \varepsilon_2 + \varepsilon_{i+1} + \varepsilon_{i+2}$\\
$B_n(n-1)$ &$2\varepsilon_1$&$\varepsilon_1 + \varepsilon_2 + \varepsilon_n$\\
$B_n(n)$ &$2\varepsilon_1$&$-$\\
$C_n(i), \; 2\leq i \leq n-1$ &$\varepsilon_1 + \varepsilon_2$&$2\varepsilon_1 + 2\varepsilon_{i+1}$\\
$D_n(i), \; 3\leq i \leq n-3$ &$2\varepsilon_1$&$\varepsilon_1 +\varepsilon_2 + \varepsilon_{i+1} + \varepsilon_{i+2}$\\
\hline
\end{tabular} \label{T51}
\end{center}
\end{table}

\begin{table}[t]
\caption{Highest Weights for Special Constituents (Exceptional Cases)}
\begin{center}
\makebox[0.93 \width][r]{ 
\begin{tabular}{cc} 
\hline
Type        &$V(\mu + \geg)$ \\
\hline
$E_6(3)$ &$\ga_1 + 2\ga_2 + 2\ga_3 + 4\ga_4 + 3\ga_5 + 2\ga_6$\\
$E_6(5)$ &$2\ga_1 + 2\ga_2 + 3\ga_3 + 4\ga_4 + 2\ga_5 + \ga_6$\\
$E_7(2)$ &$2\ga_1 + 2\ga_2 + 4\ga_3 + 5\ga_4 + 4\ga_5 + 3\ga_6 + 2\ga_7$\\
$E_7(6)$ &$2\ga_1 + 3\ga_2 + 4\ga_3 + 6\ga_4 + 4\ga_5 + 2\ga_6 + \ga_7$ \\
$E_8(1)$ &$2\ga_1 + 4\ga_2  + 5\ga_3 + 8\ga_4 + 7\ga_5 + 6\ga_6 + 4\ga_7 + 2\ga_8$\\
$F_4(4)$ & $2\ga_1+ 4\ga_2 + 6\ga_3 + 2\ga_4$\\
\hline
Type       &$V(\mu + \geng)$ \\
\hline
$E_6(3)$ &$2\ga_1 + 2\ga_2 + 2\ga_3 + 3\ga_4 + 2\ga_5 + \ga_6$\\
$E_6(5)$ &$\ga_1 + 2\ga_2 + 2\ga_3 + 3\ga_4 + 2\ga_5 + 2\ga_6$ \\
$E_7(2)$ &$-$\\
$E_7(6)$ &$2\ga_1 + 2\ga_2 + 3\ga_3 + 4\ga_4 + 3\ga_5 + 2\ga_6 + 2\ga_7$ \\
$E_8(1)$ &$-$\\
$F_4(4)$ &$-$\\
\hline
\end{tabular} \label{T52}
}
\end{center}
\end{table}

\begin{table}[h]
\caption{Roots $\mu$, $\geg$, and $\geng$ (Classical Cases)}
\begin{center}
\begin{tabular}{ccccc} 
\hline
Type    &$\mu$    &$\geg$ &$\geng$ \\
\hline
$B_n(i),\;  3\leq i \leq n-2$ &$\varepsilon_1 + \varepsilon_{i+1}$&$\varepsilon_1 - \varepsilon_{i+1}$&$\varepsilon_2 + \varepsilon_{i+2}$\\
$B_n(n-1)$                       &$\varepsilon_1 + \varepsilon_n$     &$\varepsilon_1-\varepsilon_n$      &$\varepsilon_2$\\
$B_n(n)$                          &$\varepsilon_1$              &$\varepsilon_1$&$-$\\
$C_n(i), \; 2\leq i \leq n-1$  &$\varepsilon_1 + \varepsilon_{i+1}$&$\varepsilon_2-\varepsilon_{i+1}$&$\varepsilon_1 + \varepsilon_{i+1}$\\
$D_n(i), \; 3\leq i \leq n-3$  &$\varepsilon_1 + \varepsilon_{i+1}$&$\varepsilon_1 - \varepsilon_{i+1}$&$\varepsilon_2 + \varepsilon_{i+2}$\\
\hline
\end{tabular} \label{T53}
\end{center}
\end{table}

\begin{table}[!]
\caption{Roots $\mu$, $\geg$, and $\geng$ (Exceptional Cases)}
\begin{center}
\begin{tabular}{cc} 
\hline
Type        &$\mu$ \\
\hline
$E_6(3)$ &$\ga_1+ \ga_2 + \ga_3 + 2\ga_4 +2\ga_5 + \ga_6$\\
$E_6(5)$ &$\ga_1+ \ga_2 + 2\ga_3 + 2\ga_4 +\ga_5 + \ga_6$\\
$E_7(2)$ &$\ga_1+ \ga_2 + 2\ga_3 + 3\ga_4 +3\ga_5 + 2\ga_6 + \ga_7$\\
$E_7(6)$ &$\ga_1+ 2\ga_2 + 2\ga_3 + 3\ga_4 +2\ga_5 + \ga_6 + \ga_7$\\
$E_8(1)$ &$\ga_1+ 3\ga_2 + 3\ga_3 + 5\ga_4 +4\ga_5 + 3\ga_6 + 2\ga_7+\ga_8$\\
$F_4(4)$ &$\ga_1+ 2\ga_2 + 3\ga_3 + \ga_4$\\
\hline
Type        &$\geg$ \\
\hline
$E_6(3)$ &$\ga_2 + \ga_3+2\ga_4 + \ga_5 + \ga_6$\\
$E_6(5)$ &$\ga_1 + \ga_2 + \ga_3 + 2\ga_4 + \ga_5$\\
$E_7(2)$ &$\ga_1 + \ga_2 + 2\ga_3 + 2\ga_4 +\ga_5 + \ga_6 + \ga_7$\\
$E_7(6)$ &$\ga_1 + \ga_2 + 2\ga_3 + 3\ga_4 + 2\ga_5 + \ga_6$\\
$E_8(1)$ &$\ga_1 + \ga_2 + 2\ga_3 + 3\ga_4 + 3\ga_5 + 3\ga_6 + 2\ga_7 +\ga_8$\\
$F_4(4)$ &$\ga_1+ 2\ga_2 + 3\ga_3 + \ga_4$\\
\hline
Type       &$\geng$ \\
\hline
$E_6(3)$ &$\ga_1 + \ga_2 + \ga_3 + \ga_4$\\
$E_6(5)$ &$\ga_2 + \ga_4 + \ga_5 + \ga_6$\\
$E_7(2)$ &$-$\\
$E_7(6)$ &$\ga_1 + \ga_3 + \ga_4 + \ga_5 + \ga_6 + \ga_7$\\
$E_8(1)$ &$-$\\
$F_4(4)$ &$-$\\
\hline
\end{tabular} \label{T54}
\end{center}
\end{table}

Table \ref{T55} summarizes the types of special constituents for each parabolic 
subaglebra $\fq$. One may want to observe that almost all the special constituents 
are of type 1a. We regard any roots as long roots, when $\fg$ is simply laced.
A dash indicates that no special constituent of the type exists in the case.

\begin{table}[h]
\caption{Types of Special Constituents}
\begin{center}
\begin{tabular}{ccccc} 
\hline
Type        &$V(\mu+\geg)$ &$V(\mu+\geng)$ \\
\hline
$B_n(i),\;  3\leq i \leq n-2$ &Type 1a&Type 1a\\
$B_n(n-1)$                      &Type 1a&Type 1b\\
$B_n(n)$                         &Type 2&$-$\\
$C_n(i), \; 2\leq i \leq n-1$ &Type 3&Type 2\\
$D_n(i), \; 3\leq i \leq n-3$ &Type 1a&Type 1a\\
$E_6(3)$                          &Type 1a&Type 1a\\
$E_6(5)$                          &Type 1a&Type 1a\\
$E_7(2)$                          &Type 1a&$-$\\
$E_7(6)$                          &Type 1a&Type 1a\\
$E_8(1)$                          &Type 1a&$-$\\
$F_4(4)$                          &Type 2&$-$\\
\hline
\end{tabular} \label{T55}
\end{center}
\end{table}


\begin{Rem}\label{Rem5.4.6}
It is observed from Table \ref{T53} and Table \ref{T54}
that we have 
$\mu \pm \ge \notin \gD$, unless $V(\mu + \ge)$ is of type 3.
In particular, if $V(\mu+\ge)$ is of type 1a then $\IP{\mu}{\ge} = 0$.
\end{Rem}

\begin{Rem}\label{Rem5.2.11}
Table \ref{T55} shows that 
when $V(\mu+\ge)$ is a special constituent of type 1a,
$\fq$ is of type 
$B_n(i)$ $(3 \leq i \leq n-1)$, $D_n(i)$, $E_6(3)$, 
$E_6(5)$, $E_7(2)$, $E_7(6)$, or $E_8(1)$.
The data in Appendix \ref{chap:Data} shows that 
when $\fq$ is of type $B_n(i)$ for $3 \leq i \leq n-1$,
the simple root $\ga_\fq = \varepsilon_i - \varepsilon_{i+1}$ 
that parametrizes $\fq$ is a long root
and that the set $\gD(\fz(\fn))$ contains solely long roots.
Since we regard any roots as long roots for $\fg$ simply laced,
it follows that when $V(\mu+\ge)$ is of type 1a,
the simple root $\ga_\fq$ and any root $\gamma_j \in \gD(\fz(\fn))$ 
are all long roots. 
\end{Rem}

\subsection{Computations for structure constants}\label{SS542}

For the rest of this section,
we collect technical results on the special constituents and 
structure constatnts,
so that certain arguments will go smoothly when 
we find the special values for the $\Omega_2$ systems.
The root vectors $X_\ga$ and the structure constants 
$N_{\ga, \gb}$ are normalized as in Section \ref{chap:O1}.

\begin{Lem}\label{Lem5.25}
Let $V(\mu+\ge$) be a special constituent $\fl \otimes \fz(\fn)$ of type 1a,
and $\ga \in \gD^+(\fl)$.
If $\ge + \ga \in \gD$ then $\mu - \ga \in \gD$.
\end{Lem}

\begin{proof}
We show that $\IP{\mu}{\ga} > 0$.
Since $\mu+\ge$ is the highest weight of an irreducible $\fl$-module, 
it is $\gD(\fl)$-dominant. Thus,
\begin{equation}\label{Eqn5.261}
\IP{\mu+\ge}{\ga} = \IP{\mu}{\ga}+\IP{\ge}{\ga} \geq 0.
\end{equation}
\noindent Observe that,
as $\mu+\ge$ is of type 1a,  
$\ge$ is a long root of $\fg$.
Since $\ga+ \ge$ is assumed to be a root, 
Lemma \ref{Lem3.1.3} implies that $\IP{\ga}{\ge^{\vee}} = -1$;
in particular, $\IP{\ge}{\ga} < 0$. Now, by (\ref{Eqn5.261}), we have 
\beu
\IP{\mu}{\ga} \geq - \IP{\ge}{\ga} > 0.
\end{equation*} 
\end{proof}

\begin{Lem}\label{Lem5.26}
Let $V(\mu+\ge)$ be a special constituent of $\fl \otimes \fz(\fn)$ of type 1a.
If $\ga \in \gD^+(\fl)$ with $\ga + \ge \in \gD$ then, 
for all $\gamma_j \in \gD(\cfn)$, we have
\begin{equation*}
\emph{\ad}(X_\mu)\emph{\ad}(X_{\ga + \ge})X_{-\gamma_j} = 0.
\end{equation*}
\end{Lem}

\begin{proof}

If $(\ga+\ge) -\gamma_j \notin \gD$ then there is nothing to prove. 
So we assume that $(\ga+\ge)-\gamma_j \in \gD$ and 
$\mu + (\ga+\ge)- \gamma_j \in \gD$. 
Since $\mu+\ge$ is assumed to be of type 1a, the root $\mu$ is long.
Lemma \ref{Lem3.1.3} then implies that 
\begin{equation}\label{Eqn5.26}
\IP{(\ga+\ge)-\gamma_j}{\cmu} = -1.
\end{equation}
By Remark \ref{Rem5.4.6}, we have $\IP{\ge}{\cmu} = 0$. 
Thus (\ref{Eqn5.26}) becomes
\begin{equation}\label{Eqn5.27}
\IP{\ga}{\cmu} - \IP{\gamma_j}{\cmu} = -1.
\end{equation}
Since $\mu$ is the highest weight for $\fg(1)$,
$\gamma_j \in \gD(\cfn)$, and $\ga \in \gD^+(\fl)$,
neither $\mu +\ga$ nor  $\gamma_j + \mu$ is a root.
Then, as $\mu$ is a long root, (\ref{Eqn5.27}) holds if and only if
$\IP{\ga}{\cmu} = 0$ and $\IP{\gamma_j}{\cmu} = 1$. 
On the other hand, since $\ga+\ge$ is a root by hypothesis
and by Lemma \ref{Lem5.25}, $\mu-\ga$ is a root.
In particular, by Lemma \ref{Lem3.1.3}, 
$\IP{\ga}{\cmu} = 1$.
Now we have $\IP{\ga}{\cmu} = 1$ and 
$\IP{\ga}{\cmu} = 0$, which is a contradiction.
\end{proof}

\vsp

For any $\ad(\fh)$-invariant subspace
$W \subset \fg$ and any weight $\nu \in \fh^*$, we write 
\begin{equation*}
\gD_{\nu}(W) = \{ \ga \in \gD(W) \; | \; \nu - \ga \in \gD\}.
\end{equation*}
In Section \ref{chap:O2}, we will construct the 
$\Omega_2|_{V(\mu+\ge)^*}$ systems and find 
their special values,
when $V(\mu+\ge)$ is of either type 1a or type 2.
When we do so,
the roots $\gb \in \gD_{\mu+\ge}(\fg(1))$ and 
$\gamma_j \in \gD_{\mu+\ge}(\fz(\fn))$ will play a role.
Therefore, for the rest of this section, 
we shall show several technical results about
those roots, so that certain argument will become simple.

First of all, we need check that
$\gD_{\mu+\ge}(\fg(1))$ and $\gD_{\mu+\ge}(\fz(\fn))$ 
are not empty.
It is clear that $\gD_{\mu+\ge}(\fg(1)) \neq \emptyset$,
since $\mu$, $\ge \in \gD_{\mu+\ge}(\fg(1))$.
Moreover, Lemma \ref{Lem5.2.141} below shows that 
when $V(\mu+\ge)$ is of type 2,
we have $\gD_{\mu+\ge}(\fg(1))=\{\mu\}$.

\begin{Lem}\label{Lem5.2.141}
If $V(\mu+\ge)$ is a special constituent of $\fl \otimes \fz(\fn)$
of type 2 then $\gD_{\mu+\ge}(\fg(1)) = \{\mu\}$.
\end{Lem}

\begin{proof}
First we claim that $\mu$ has the maximum height among the roots 
$\gb \in \gD(\fg(1))$. As $\fg(1)$ is the irreducible $L$-module with highest weight $\mu$,
any root $\gb \in \gD(\fg(1))$ is of the form 
$\gb = \mu - \sum_{\ga \in \Pi(\fl)}n_\ga \ga$ with $n_\ga \in \Z_{\geq 0}$.
Then if $\Ht(\mu)$ and $\Ht(\gb)$ denote the heights of $\mu$ and $\gb$,
respectively, then 
\begin{equation*}
\Ht(\mu) = \Ht(\gb) + \sum_{\ga \in \Pi(\fl)} n_\ga \geq \Ht(\gb).
\end{equation*}

Now as $V(\mu+\ge)$ is of type 2, by definition,
we have $\mu+\ge = 2\mu$. If $\gb \in \gD_{2\mu}(\fg(1))$
then $2\mu-\gb \in \gD(\fg(1))$. In particular, 
the height $\Ht(2\mu-\gb)$ satisfies $\Ht(\mu) \geq \Ht(2\mu-\gb)$.
If $\gb = \mu - \sum_{\ga \in \Pi(\fl)}n_\ga \ga$ with $n_\ga \in \Z_{\geq 0}$
then 
\begin{equation*}
\Ht(\mu) \geq \Ht(2\mu-\gb) = 2\Ht(\mu) - \Ht(\gb)
=\Ht(\mu)  + \sum_{\ga \in \Pi(\fl)}n_\ga.
\end{equation*}
This forces that $\sum_{\ga \in \Pi(\fl)}n_\ga = 0$. Therefore $\gb = \mu$.
\end{proof}

\begin{Lem}\label{Lem5.2.13}
If $V(\mu+\ge)$ is a special constituent of $\fl \otimes \fz(\fn)$ then
$\gD_{\mu+\ge}(\fz(\fn)) \neq \emptyset$.
\end{Lem}

\begin{proof}
Observe that
the highest weight $\mu+\ge$ of $V(\mu+\ge) \subset \fl \otimes \fz(\fn)$
must be of the form
\begin{equation*}
\mu + \ge =
\begin{cases}
\xig + \gamma' &\text{if $V(\mu+\ge) \subset \flg \otimes \fz(\fn)$}\\
\xing+\gamma'' &  \text{if $V(\mu+\ge) = \flng \otimes \fz(\fn)$}
\end{cases}
\end{equation*}
for some $\gamma',\gamma'' \in \gD(\fz(\fn))$, where $\xig$ and $\xing$ 
are the highest weights for $\flg$ and $\flng$, respectively.
Then we have $\gamma',\gamma'' \in \gD_{\mu+\ge}(\fz(\fn))$.
\end{proof}

The following simple technical lemma will simplify 
an argument in later proofs.

\begin{Lem}\label{Lem2.6}
Let $\ga, \gb, \gd \in \gD$ with $\ga$, $\gb \neq \gd$.
If $\ga + \gb \notin \gD$ and $\ga + \gb - \gd \in \gD$
then the following hold:
\begin{enumerate}
\item[(1)] $\ga - \gd,\; \gb - \gd \in \gD$, and 
\item[(2)] $N_{\gb, \ga-\gd}N_{\ga, -\gd} = N_{\ga, \gb-\gd}N_{\gb, -\gd}$.
\end{enumerate}
\end{Lem}

\begin{proof}
These simply follow from the structure of the complex simple Lie algebras.
\end{proof}

\begin{Lem}\label{Lem5.2.14}
Let $W$ be any $\ad(\fh)$-invariant subspace of $\fg$
with the condition $\gD_{\mu+\ge}(W)\backslash\{\mu,\ge\} \neq \emptyset$. 
If $V(\mu+\ge)$ is a special constituent of $\fl \otimes \fz(\fn)$ of 
type 1a, type 1b, or type 2 then, 
for any $\gd \in \gD_{\mu+\ge}(W) \backslash \{\mu,\ge\}$,
we have $\gd -\mu$, $\gd-\ge \in \gD$.
\end{Lem}

\begin{proof}
If $V(\mu+\ge)$ is of type 1a, type 1b, or type 2 then, by definition,
$\mu+\ge$ is not a root. Then this lemma simply follows from 
Lemma \ref{Lem2.6}
\end{proof}

\begin{Rem}\label{Rem736}
A direct observation shows that if $V(\mu+\ge)$ is a special constituent
of type 1a then $\gD_{\mu+\ge}(\fg(1))\backslash \{\mu,\ge\} \neq \emptyset$.
\end{Rem}

\begin{Lem}\label{Lem5.2.15}
If $V(\mu+\ge)$ is a special constituent of $\fl \otimes \fz(\fn)$ of type 1a
then, for any $\ga \in \gD_{\mu+\ge}(\fg(1))$ and 
any $\gamma_j \in \gD_{\mu+\ge}(\fz(\fn))$, we have 
$\gamma_j - \ga \in \gD$.
\end{Lem}

\begin{proof}
By Lemma \ref{Lem2.6}, we have $\gamma_j -\mu, \gamma_j -\ge \in \gD$.
So, let $\ga \neq \mu,\ge$.
We show that $\IP{\gamma_j}{\ga} > 0$. Observe that since
$\ga \in \gD(\fg(1))$ and $\gamma_j \in \gD(\fz(\fn))$, we have 
$\gamma_j + \ga \notin \gD$. Thus $\IP{\gamma_j}{\ga} \geq 0$.
Since $\ga \in \gD_{\mu+\ge}(\fg(1))\backslash \{\mu,\ge\}$ 
and $\gamma_j \in \gD_{\mu+\ge}(\fz(\fn))$, 
by Lemma \ref{Lem5.2.14}, we have 
$\mu-\ga$, $\ge - \gamma_j \in \gD$. 
Then we first claim that if 
$\IP{\gamma_j}{\ga}=0$ then 
$(\mu-\ga) + (\ge-\gamma_j) \in \gD$.
Since $V(\mu+\ge)$ is assumed to be of type 1a, 
both $\mu$ and $\ge$ are long roots.
Thus, by Lemma \ref{Lem3.1.3}, 
$\IP{\gamma_j}{\cmu} = \IP{\ga}{\cge} = 1$; in particular, 
$\IP{\gamma_j}{\mu}$, $\IP{\ga}{\ge} > 0$.
By Remark \ref{Rem5.4.6}, we have $\IP{\mu}{\ge} = 0$.
Then,
\begin{equation*}
\IP{\mu-\ga}{\ge-\gamma_j} = -\IP{\mu}{\gamma_j} - \IP{\ga}{\ge} < 0.
\end{equation*}
Therefore, as $\mu-\ga, \ge-\gamma_j \in \gD$,
it follows that $(\mu-\ga)+(\ge-\gamma_j) \in \gD$.
On the other hand, since $\IP{\mu}{\ge}=0$
and $\IP{\gamma_j}{\ga}$ is assumed to be 0,
we have
\begin{align*}
&||(\mu-\ga) + (\ge-\gamma_j)||^2\\
&=||\mu||^2 + ||\ga||^2 + ||\ge||^2 + ||\gamma_j||^2
-2\IP{\ga}{\mu} - 2\IP{\ga}{\ge} - 2\IP{\gamma_j}{\mu}
-2\IP{\gamma_j}{\ge}.
\end{align*}
For $\nu=\ga, \gamma_j$ and 
$\zeta = \mu, \ge$, by Lemma \ref{Lem3.1.3}, we have 
$\IP{\nu}{\zeta^\vee} = 2\IP{\nu}{\zeta}/||\zeta||^2 = 1$,
as $\mu$ and $\ge$ are long roots.
Therefore,
$2\IP{\nu}{\zeta} = ||\zeta||^2$, and so, 
\begin{equation*}
||(\mu-\ga) + (\ge-\gamma_j)||^2
= ||\ga||^2 + ||\gamma_j||^2 - ||\mu||^2 - ||\ge||^2.
\end{equation*}
Since $\mu$ and $\ge$ are assumed to be long roots, this shows that
$||(\mu-\ga) + (\ge-\gamma_j)||^2 \leq 0$, which contradicts that
$(\mu-\ga) + (\ge-\gamma_j)$ is a root.
Hence, $\IP{\gamma_j}{\ga} >0$.
\end{proof}

\begin{Lem}\label{Lem5.2.16}
If $V(\mu+\ge)$ is a special constituent of $\fl \otimes \fz(\fn)$ of type 1a
or type 2 then, for any $\gamma_j \in \gD_{\mu+\ge}(\fz(\fn))$,
\begin{equation*}
\gD_{\mu+\ge}(\fg(1)) \subset \gD_{\gamma_j}(\fg(1)).
\end{equation*}
In particular, $\gD_{\gamma_j}(\fg(1)) \neq \emptyset$
for any $\gamma_j \in \gD_{\mu+\ge}(\fz(\fn))$.
\end{Lem}

\begin{proof}
It is clear that the type 1a case follows from Lemma \ref{Lem5.2.15}.
The type 2 case follows from Lemma \ref{Lem5.2.141} and Lemma \ref{Lem5.2.14}.
\end{proof}

If $V(\mu+\ge)$ is a special constituent of $\fl \otimes \fz(\fn)$
then, for $\gb \in \gD$, we write
\begin{equation*}
\gt(\gb) = (\mu+\ge)-\gb.
\end{equation*}

\begin{Lem}\label{Lem5.2.191}
If $V(\mu+\ge)$ is a special constituent of $\fl \otimes \fz(\fn)$ of type 1a or type 2
then, for any $\gamma_j \in \gD_{\mu+\ge}(\fz(\fn))$, 
\begin{equation*}
\gD_{\gt(\gamma_j)}(\fg(1)) \neq \emptyset.
\end{equation*}
\end{Lem}

\begin{proof}
This simply follows from Lemma \ref{Lem2.6}.
\end{proof}

\begin{Lem}\label{Lem5.2.17}
If $V(\mu+\ge)$ is a special constituent of type 1a or type 2 then 
\begin{equation*}
\sum_{\gamma_j \in \gD_{\mu+\ge}(\fz(\fn))}
N_{\mu, \ge-\gamma_j} N_{-\mu, \gamma_j-\ge}
N_{\ge,-\gamma_j}N_{-\ge,\gamma_j}
> 0,
\end{equation*}
where $N_{\ga,\gb}$ are the structure constants for $\ga,\gb \in \gD$
defined in Section \ref{chap:O1}.
\end{Lem}

\begin{proof}
It follows from the normalization (H7) in Section \ref{chap:O1}
that
\begin{equation*}
N_{\mu, \ge-\gamma_j} N_{-\mu, \gamma_j-\ge}=
-\frac{q_{\mu, \ge-\gamma_j}(1+p_{\mu,\ge-\gamma_j})}{2}||\mu||^2
\end{equation*}
and
\begin{equation*}
N_{\ge,-\gamma_j}N_{-\ge,\gamma_j}=
-\frac{q_{\ge, -\gamma_j}(1+p_{\ge,-\gamma_j})}{2}||\ge||^2.
\end{equation*}
In particular, by (\ref{Eqn3.1.1}) in Section \ref{chap:O1},
$N_{\mu, \ge-\gamma_j} N_{-\mu, \gamma_j-\ge} \leq 0$ and  
$N_{\ge,-\gamma_j}N_{-\ge,\gamma_j} \leq 0$.
By Lemma \ref{Lem5.2.13} and Lemma \ref{Lem5.2.14},
$\gD_{\mu+\ge}(\fz(\fn)) \neq \emptyset$ and 
$\gamma_j -\ge \in \gD$ for any $\gamma_j \in \gD_{\mu+\ge}(\fz(\fn))$.
Therefore,  for all $\gamma_j \in \gD_{\mu+\ge}(\fz(\fn))$, we have
\begin{equation*}
N_{\mu, \ge-\gamma_j} N_{-\mu, \gamma_j-\ge}
N_{\ge,-\gamma_j}N_{-\ge,\gamma_j} > 0.
\end{equation*}
\end{proof}

\begin{Lem}\label{Lem5.2.231}
If $V(\mu+\ge)$ is a special constituent of type 1a then,
for any $\ga \in \gD_{\mu+\ge}(\fg(1))\backslash \{\mu,\ge\}$ 
and any $\gamma_j \in \gD_{\mu+\ge}(\fz(\fn))$, we have the following:
\begin{enumerate}
\item[(1)] $[X_{-\gamma_j}, X_{\ga-\mu}] = [X_{\gt(\gamma_j)}, X_{\ga-\mu}] = 0$.
\item[(2)] $N_{\mu-\gamma_j, \ga-\mu} N_{-(\mu-\gamma_j), -(\ga-\mu)}
=-\frac{||\mu-\gamma_j||^2}{2}$.
\end{enumerate}
\end{Lem}

\begin{proof}
To prove (1), we show that $-\gamma_j + \ga -\mu$ and 
$\gt(\gamma_j) + \ga-\mu$ are neither zero nor roots.
First of all, if $-\gamma_j + \ga -\mu = 0$ then 
$\gamma_j = \mu-\ga \in \gD(\fl)$, which contradicts
that $\gamma_j \in \gD(\fz(\fn))$.
Next, if $\gt(\gamma_j) + \ga -\mu = 0$ then since 
$\gt(\gamma_j) + \ga - \mu = \ge + \ga -\gamma_j$,
we would have
$\ga+\ge=\gamma_j \in \gD$. On the other hand,
as $V(\mu+\ge)$ is assumed to be of type 1a, $\ge$ is a long root. 
As $\ga \in \gD_{\mu+\ge}(\fg(1))\backslash \{\mu, \ge\}$, 
by Lemma \ref{Lem5.2.14}, we have $\ga-\ge \in \gD$.
Then, by Lemma \ref{Lem3.1.3}, it follows that $\ga+\ge \not \in \gD$,
which is a contradiction.
To show $\gamma_j + \ga - \mu$ is not a root,
observe that, by Lemma \ref{Lem3.1.3}, we have
\begin{equation*}
\IP{-\gamma_j + \ga-\mu}{\cmu}
=-1 + 1 - 2 = -2.
\end{equation*}
Thus, if $-\gamma_j +\ga -\mu \in \gD$ then
$(-\gamma_j + \ga -\mu) + 2\mu$ would be a root.
However, since $\mu$ is a long root, 
it is impossible.
The fact that $\gt(\gamma_j) + \ga-\mu \notin \gD$
can be shown in a similar manner.

By the normalization (H7) in Section \ref{chap:O1},
to show (2), 
it suffices to show that 
$p_{\mu-\gamma_j, \ga-\mu} = 0$ and  
$q_{\mu-\gamma_j, \ga-\mu} = 1$.
Observe that,
by Lemma \ref{Lem5.2.15},
$(\ga-\mu) + (\mu-\gamma_j) = \gamma_j - \ga$ is a root.
As $V(\mu+\ge)$ is assumed to be of type 1a,
$\mu$ is a long root.
By Remark \ref{Rem5.2.11},
the root $\gamma_j$ is also a long root.
Therefore  $\mu-\gamma_j$ is a long root.
Now the proposed equality follows from Lemma \ref{Lem3.1.3}.
\end{proof}

\begin{Lem}\label{Lem5.2.21}
If $V(\mu+\ge)$ is a special constituent of type 1a then,
for any $\ga \in \gD_{\mu +\ge}(\fg(1))\backslash\{\mu,\ge\}$ and 
any $\gamma_j \in \gD_{\mu+\ge}(\fz(\fn))$, we have the following:
\begin{enumerate}
\item[(1)] $N_{\ga, -\gamma_j} N_{\mu, -\ga}
= N_{\mu, -\gamma_j}N_{\ga-\mu, \mu-\gamma_j}$, and
\item[(2)] $N_{-\gt(\gamma_j), \gt(\ga)}N_{-\gt(\ga), \gt(\mu)}
=N_{-\gt(\gamma_j), \gt(\mu)}N_{-(\mu-\gamma_j), -(\ga-\mu)}$.
\end{enumerate}
\end{Lem}

\begin{proof}
It follows from Lemma \ref{Lem5.2.14} that $\ga -\mu \in \gD$.
Therefore, we have $X_\ga = (1/N_{\ga-\mu,\mu})[X_{\ga-\mu}, X_\mu]$.
Now the assertion (1) follows from the Jacobi identity and the normalization (H6)
with Lemma \ref{Lem5.2.15} and Lemma \ref{Lem5.2.231} (1).
The assertion (2) can be shown similarly.
\end{proof}

\begin{Lem}\label{Lem5.2.22}
Let $\fq$ be a parabolic subalgebra
of quasi-Heisenberg type, listed in (\ref{Eqn4.0.1}) or (\ref{Eqn4.0.2}),
and $\ga_\fq$ be the simple root that parametrizes the parabolic subalgebra $\fq$.
If $V(\mu+\ge)$ is a special constituent of type 1a
then, for any $\ga \in \gD_{\mu +\ge}(\fg(1))\backslash\{\mu,\ge\}$ and 
any $\gamma_j \in \gD_{\mu+\ge}(\fz(\fn))$,
\begin{equation}\label{Eqn9.3.3}
N_{\ga, -\gamma_j}N_{\mu, -\ga}N_{-\gt(\gamma_j), \gt(\ga)}N_{-\gt(\ga), \gt(\mu)}
=N_{\mu, \ge-\gamma_j}N_{\ge, -\gamma_j} \frac{||\ga_\fq||^2}{2}.
\end{equation}
\end{Lem}

\begin{proof}
By Lemma \ref{Lem5.2.21}, we have
\begin{align*}
N_{\ga, -\gamma_j}N_{\mu, -\ga}
N_{-\gt(\gamma_j), \gt(\ga)}N_{-\gt(\ga), \gt(\mu)}\nonumber
&=N_{\mu, -\gamma_j}N_{\ga-\mu, \mu-\gamma_j}
    N_{-\gt(\gamma_j), \gt(\mu)}N_{-(\mu-\gamma_j), -(\ga-\mu)}\\
&=N_{\mu, -\gamma_j}N_{-\gt(\gamma_j), \gt(\mu)}
    N_{\ga-\mu, \mu-\gamma_j} N_{-(\mu-\gamma_j), -(\ga-\mu)}\\
&=N_{\mu, -\gamma_j}N_{-\gt(\gamma_j), \gt(\mu)} \frac{||\mu-\gamma_j||^2}{2}.
\end{align*}
Note that Lemma \ref{Lem5.2.231} (2) is applied from line two to line three.
Since $-\gt(\gamma_j) + \gt(\mu) + (\mu-\gamma_j) = 0$ 
with $\gt(\mu)=(\mu+\ge)-\mu =\ge$,
by the normalization (H6),
we have $N_{-\gt(\gamma_j), \gt(\mu)} = N_{\ge, \mu-\gamma_j}$.
By Lemma \ref{Lem2.6} with $\ga = \mu$, $\gb = \ge$, and $\gd = \gamma_j$, 
it follows that
$N_{\ge, \mu-\gamma_j}N_{\mu, -\gamma_j}
 =N_{\mu, \ge-\gamma_j}N_{\ge, -\gamma_j}$.
Therefore,
\begin{equation*}
N_{\mu, -\gamma_j}N_{-\gt(\gamma_j), \gt(\mu)}
= N_{\mu, -\gamma_j}N_{\ge, \mu-\gamma_j}
 =N_{\mu, \ge-\gamma_j}N_{\ge, -\gamma_j}.
\end{equation*}
Remark \ref{Rem5.2.11} shows that 
$\gamma_j$ and $\ga_\fq$ are long roots, when $V(\mu+\ge)$ is of type 1a.
Since $\mu$ is assumed to be a long root, the root $\mu-\gamma_j$ 
is a long root. Thus $||\mu-\gamma_j||^2 = ||\ga_\fq||^2$. Hence,
\begin{align*}
N_{\ga, -\gamma_j}N_{\mu, -\ga}
N_{-\gt(\gamma_j), \gt(\ga)}N_{-\gt(\ga), \gt(\mu)}
&=N_{\mu, -\gamma_j}N_{-\gt(\gamma_j), \gt(\mu)} 
\frac{||\mu-\gamma_j||^2}{2}\\
&=N_{\mu, \ge-\gamma_j}N_{\ge, -\gamma_j} \frac{||\ga_\fq||^2}{2}.
\end{align*}

\end{proof}

%% file: SS_O2.tex
\section
{The $\Omega_2$ Systems}\label{chap:O2}

We continue with $\fq = \fl \oplus \fg(1) \oplus \fz(\fn)$
a maximal parabolic subalgebra of quasi-Heisenberg type,
listed in (\ref{Eqn4.0.1}) or (\ref{Eqn4.0.2}).
In this section, by using the preliminary results 
from Section \ref{chap:Sp},
we shall determine the complex parameter $s_2 \in \C$ for the line bundle $\Cal{L}_{s}$
so that the $\Omega_2$ systems are conformally invariant on $\Cal{L}_{s_2}$.
This is done in Theorem \ref{Thm8.3.1}.

\subsection{Covariant map $\tau_2$}\label{SS811}

As we have observed in Subsection \ref{SS7211}, 
to construct the $\Omega_2|_{V^*}$ system,
we use the covariant map $\tau_2$ and the associated $L$-intertwining 
operator $\ttau_2|_{V^*}$, where $V^*$ is an irreducible constituents
of $\fl^* \otimes \fz(\fn)^*=\fg(0)^* \otimes \fg(2)^*$. 
We first show that the covariant map $\tau_2$
is not identically zero, and also that 
the $L$-intertwining operators $\ttau_2|_{V^*}$
are not identically zero for certain irreducible constituents $V$.
We keep using the normalizations from Section \ref{chap:O1}.

We start by showing that $\tau_2$ is not identically zero.
The covariant map $\tau_2$ is given by 
\begin{align*}
\tau_2 : \fg(1) &\to \fl\otimes \fz(\fn)\\
X &\mapsto \frac{1}{2} \; \big(\ad(X)^2 \otimes \text{Id} \big) \omega
\end{align*}
with $\omega = \sum_{\gamma_j \in \gD(\fz(\fn))}X_{-\gamma_j} \otimes X_{\gamma_j}$.
The following technical lemma 
will make a certain argument simpler in later proofs.

\begin{Lem}\label{Lem6.11}
If $V(\mu + \ge)$ is a special constituent of 
type 1a or type 2 then
\begin{equation}\label{Eqn6.1.3}
\tau_2(X_\mu + X_\ge) 
=a_{\mu,\ge}\; \emph{\ad}(X_{\mu})\;\emph{\ad}(X_\ge)\;\omega,
\end{equation}
\noindent
where $a_{\mu,\ge} = 1 + \gd_{\mu, \ge}$ with $\gd_{\mu, \ge}$ the Kronecker delta.
\end{Lem}

\begin{proof}
It is clear that (\ref{Eqn6.1.3}) holds if $\mu+\ge$ is of type 2.
Indeed, if $\ge=\mu$ then we have
\begin{equation*}
\tau_2(2X_\mu)
= 4\tau_2(X_\mu) = 2 \ad(X_\mu)^2\omega.
\end{equation*}
If $\mu+\ge$ is of type 1a then,
by definition, $\mu+\ge \notin \gD$ and 
both $\mu$ and $\ge$ are long roots.
Thus, in the case, $\ad(X_\mu)\;\ad(X_\ge)\; = \ad(X_\ge)\;\ad(X_\mu)$.
Moreover, by Lemma \ref{Lem3.1.3}, we have
$\ad(X_\mu)^2X_{-\gamma_j} = \ad(X_\ge)^2X_{-\gamma_j} = 0$
for any $\gamma_j \in \gD(\fz(\fn))$.
Hence,
\begin{equation*}
\tau_2(X_\mu + X_\ge) 
=(1/2)(2\; \ad(X_\mu)\;\ad(X_\ge))\;\omega
=\ad(X_\mu)\;\ad(X_\ge)\;\omega.
\end{equation*}
\end{proof}

\begin{Prop}\label{Prop6.12}
Let $\fq$ be a maximal parabolic subalgebra of 
quasi-Heisenberg type listed in (\ref{Eqn4.0.1}) or (\ref{Eqn4.0.2}).
Then the covariant map $\tau_2$ is not identically zero.
\end{Prop}

\begin{proof}
To prove that $\tau_2$ is not identically zero,
it suffices to show that there exists a vector $X \in \fg(1)$
so that $\tau_2(X) \neq 0$.
Observe that, for each $\fq$ under consideration, 
$\fl \otimes \cfn$ has at least one special constituent $V(\mu+\ge)$ of type 1a or type 2
(see Table \ref{T55} in Subsection \ref{SS540}). Therefore, $\gD(\fg(1))$ always contains 
a root $\ge$ so that $V(\mu+\ge)$ is such a special constituent.
Then, to prove this proposition,
we show that $\tau_2(X_\mu + X_\ge) \neq 0$, 
where $X_\mu$ and $X_\ge$ are root vectors for $\mu$ and $\ge$,
respectively, with $\mu+\ge$ the highest weight for a special constituent
of type 1a or type 2.

Let $\mu + \ge$ be the highest weight of a special constituent of type 1a or type 2. 
By Lemma \ref{Lem6.11} we have
\begin{equation*}\label{Eqn6.1.4}
\tau_2(X_\mu + X_\ge)
=a_{\mu,\ge}\; \ad(X_\mu)\;\ad(X_\ge)\;\omega \nonumber
=a_{\mu,\ge} \sum_{\gamma_j \in \gD(\fz(\fn))}
\ad(X_\mu)\;\ad(X_\ge)\;X_{-\gamma_j}\otimes X_{\gamma_j}
\end{equation*}
with $a_{\mu,\ge} = 1+\gd_{\mu,\ge}$.
If there were a root $\gamma_j \in \gD(\fz(\fn))$ such that 
$\ge - \gamma_j = -\mu$ then $\mu+\ge = \gamma_j \in \gD$,
which contradicts the assumption that $\mu+\ge$ is of type 1a or type 2.
By Lemma \ref{Lem5.2.14},
if $\mu+\ge - \gamma_j \in \gD$
then $\ge-\gamma_j \in \gD$.
Then, for all $\gamma_j \in \gD(\fz(\fn))$,
\begin{equation*}
\ad(X_\mu)\;\ad(X_\ge)\;X_{-\gamma_j}=
\begin{cases}
N_{\mu,\ge-\gamma_j}N_{\ge, -\gamma_j}X_{\mu+\ge-\gamma_j}
& \text{if $\mu+\ge-\gamma_j \in \gD$}\\
0 &\text{otherwise.}
\end{cases}
\end{equation*}
Therefore, we have
\begin{align*}\label{Eqn6.1.5}
\tau_2(X_\mu + X_\ge)
&=a_{\mu,\ge} \sum_{\gamma_j \in \gD(\fz(\fn))}
\ad(X_\mu)\;\ad(X_\ge)\;X_{-\gamma_j}\otimes X_{\gamma_j} \nonumber\\
&=a_{\mu,\ge}\sum_{\gamma_j \in \gD_{\mu+\ge}(\fz(\fn))}
N_{\mu,\ge-\gamma_j}N_{\ge, -\gamma_j}X_{\mu+\ge-\gamma_j}
\otimes X_{\gamma_j}.
\end{align*}
Since $\{ X_{\mu+\ge-\gamma_j} \otimes X_{\gamma_j}  \; | \; 
\gamma_j \in \gD_{\mu+\ge}(\fz(\fn)) \}$ is a linearly independent set,
this shows that $\tau_2(X_\mu + X_\ge) \neq 0$.
\end{proof}

\vsp

Next we identify irreducible constituent $V(\nu)^*$
so that $\ttau_2|_{V(\nu)^*}$ is not identically zero.
In Subsection \ref{SS7211}, we observed that,
given an irreducible constituent $V(\nu)^*$, 
the $L$-intertwining operator 
$\ttau_2|_{V(\nu)^*} \in \Hom_{L}(V(\nu)^*, \Cal{P}^2(\fg(1)))$
is given by
\begin{equation}
\ttau_2|_{V(\nu)^*}(Y^*)(X) = Y^*(\tau_2(X)),
\end{equation}
where $\Cal{P}^2(\fg(1))$ is the space of polynomials on $\fg(1)$ of degree 2.
By Proposition \ref{Prop6.18}, we know that if $\ttau_2|_{V(\nu)^*}$
is not identically zero then $V(\nu)$ is a special constituent of $\fl\otimes \fz(\fn)$.
We now show that the converse of Proposition \ref{Prop6.18} also holds
for special constituents $V(\nu)$ of type 1a or type 2.
If $l \in L$ and $Z \in \fl$ then we denote the action of the group
and its Lie algebra on $X_\ga \otimes X_{\gamma_j}$ by 
$l \cdot (X_\ga \otimes X_{\gamma_j})$ and 
$Z \cdot (X_\ga \otimes X_{\gamma_j})$, respectively.

\begin{Prop}\label{Prop6.19} 
If $V(\mu+\ge)$ is a special constituent of $\fl \otimes \fz(\fn)$
of type 1a or type 2 then the following hold:
\begin{enumerate}
\item[(1)] The vector $\tau_2(X_\mu + X_\ge)$ 
is a highest weight vector for $V(\mu+\ge)$.
\item[(2)] The $L$-intertwining operator
$\ttau_2|_{V(\mu+\ge)^*}$ is not identically zero.
\end{enumerate}
\end{Prop}

\begin{proof}
We have shown that in the proof for Proposition \ref{Prop6.12}
that $\tau_2(X_\mu + X_\ge) \neq 0$.
Moreover, Lemma \ref{Lem6.11} gives that 
$\tau_2(X_\mu + X_\ge) =a_{\mu,\ge}\; \ad(X_{\mu})\;\ad(X_\ge)\omega$
with $a_{\mu,\ge} = 1 + \gd_{\mu,\ge}$.
For $l \in L$, we have $l \cdot \omega = \omega$ (see Lemma \ref{Cor6.5}) 
and so
\begin{equation*}
l \cdot \tau_2(X_\mu + X_\ge)
= a_{\mu,\ge}\; \big(\ad(\Ad(l)X_\mu)\;\ad(\Ad(l)X_{\ge}) 
\otimes \text{Id} \big)\;\omega.
\end{equation*}
By replacing $l$ by $\exp(tZ)$ with $Z \in \fl$, differentiating, and setting $t = 0$,
we obtain
\begin{equation}\label{Eqn714}
Z \cdot \tau_2(X_\mu + X_\ge)
= a_{\mu,\ge}\;
\big(\big(\ad([Z,X_\mu])\;\ad(X_{\ge})
+\ad(X_\mu)\;\ad([Z,X_{\ge}])\big) \otimes \text{Id} \big)\;\omega.
\end{equation}
In particular, if $Z = H \in \fh$ in (\ref{Eqn714}) then
\begin{equation*}
H \cdot \tau_2(X_\mu + X_\ge) = (\mu + \ge)(H)\tau_2(X_\mu + X_\ge).
\end{equation*}
Therefore $\tau_2(X_\mu + X_\ge)$ is a weight vector 
with weight $\mu+\ge$. 
To show that $\tau_2(X_\mu + X_\ge)$ is a highest weight vector,
we replace $Z$ in (\ref{Eqn714}) by $X_\ga$ with $\ga \in \gD^+(\fl)$.
Since $\mu$ is the highest weight for $\fg(1)$, we have
\begin{equation*}
X_\ga \cdot \tau_2(X_\mu + X_\ge)
=a_{\mu,\ge}\; 
\big(\ad(X_\mu)\; \ad([X_\ga, X_\ge]) \otimes \text{Id}\big)\; \omega.
\end{equation*}
If $\mu+\ge$ is of type 2 then, as $\ge = \mu$ in the case,
clearly $X_\ga \cdot \tau_2(X_\mu + X_\ge) = 0$. 
The case that $\mu+\ge$ is of type 1a follows from Lemma \ref{Lem5.26}.

To prove the second statement, it is enough to show that
there exist $Y^* \in V(\mu+\ge)^*$ and $X \in \fg(1)$
so that $\ttau_2(Y^*)(X) \neq 0$.
Let $Y^*_l$ be a lowest weight vector for $V(\mu+\ge)^*$.
Observe that if $Y_h$ is a highest weight vector for $V(\mu+\ge)$ 
then $Y_l^*(Y_h) \neq 0$.
Since $\tau_2(X_\mu +X_\ge)$ is a highest weight vector for $V(\mu + \ge)$,
we have
\begin{equation*}
\ttau_2|_{V(\mu+\ge)^*}(Y^*_l)(X_\mu+X_\ge) 
= Y^*_l(\tau_2(X_\mu + X_\ge)) \neq 0.
\end{equation*}
\end{proof}

\subsection{The $\Omega_2|_{V(\mu+\ge)^*}$ systems}\label{SS83}

Proposition \ref{Prop6.19} shows that 
the $L$-intertwining operator $\ttau_2|_{V(\mu+\ge)^*}$ 
is not identically zero, when $V(\mu +\ge)$ is a special constituent of 
$\fl \otimes \fz(\fn)$ of type 1a or type 2.
We thus construct the $\Omega_2|_{V(\mu+\ge)^*}$ system
corresponding to irreducible constituents $V(\mu+\ge)$ of type 1a or type 2.
Here it may be helpful to recall
some notation introduced in Subsection \ref{SS542}.
For any $\ad(\fh)$-invariant subspace
$W \subset \fg$ and any weight $\nu \in \fh^*$, we write 
\begin{equation*}
\gD_{\nu}(W) = \{ \ga \in \gD(W) \; | \; \nu - \ga \in \gD\}.
\end{equation*}
When $V(\mu+\ge)$ is a special constituent of $\fl \otimes \fz(\fn)$,
we write
\begin{equation*}
\gt(\gb) = (\mu+\ge)-\gb.
\end{equation*}

As indicated in Subsection \ref{SS7211},
the $L$-intertwining operator $\ttau_2|_{V(\mu+\ge)^*}$ 
yields a system of differential operators.
We have denoted such operators by $\Omega_2(Y^*)=\Omega_2|_{V(\mu+\ge)^*}(Y^*)$
with $Y^* \in V(\mu+\ge)^*$, where 
$\Omega_2|_{V(\mu+\ge)^*}: V(\mu+\ge)^* \to \mathbb{D}(\Cal{L}_{s})^{\bar\fn}$
is $\Cal{U}(\fl)$-equivariant. Because of such equivariance,
the system is totally determined, once $\Omega_2(Y^*_l)$ 
is constructed, where $Y^*_l$ is a lowest weight vector in $V(\mu+\ge)^*$. 

The first step is to explicitly describe $Y^*_l \in V(\mu+\ge)^*$.
Observe that we have a non-zero map
\begin{align*}
\btau_2:\fg(-1) &\to \fl \otimes \fz(\bar \fn)\\
       \bar{X} & \mapsto \frac{1}{2} \big(\ad(\bar{X})^2 \otimes \text{Id}\big) \bar{\omega}  
\end{align*}
with $\bar{\omega} = \sum_{\gamma_t \in \gD(\fz(\fn))}X_{\gamma_t} \otimes X_{-\gamma_t}$.
One checks, as in the proofs for Lemma \ref{Lem6.5} and Proposition \ref{Prop6.12}, 
that $\btau_2$ is a non-zero $L$-equivariant map. Moreover, if 
$V(\mu+\ge)$ is a special constituent of type 1a or type 2 then
\begin{equation*}
\btau_2(X_{-\mu}+X_{-\ge}) = a_{\mu, \ge}\;
 \big( \ad(X_{-\mu})\; \ad(X_{-\ge})\otimes \text{Id} \big) \bar{\omega}
\end{equation*}
with $a_{\mu,\ge} = 1 + \gd_{\mu,\ge}$.
Arguing as in Proposition \ref{Prop6.19}, we can show that $\btau_2(X_{-\mu} + X_{-\ge})$
is a lowest weight vector for $V(\mu+\ge)^*$ with lowest weight $-\mu-\ge$.
Thus,
\begin{equation}\label{Eqn8.2.1}
Y^*_l 
= \big( \ad(X_{-\mu})\; \ad(X_{-\ge})\otimes \text{Id} \big) \bar{\omega} 
=\sum_{\gamma_t \in \gD_{\mu+\ge}(\fz(\fn))}
N_{-\mu, \gamma_t - \ge}N_{-\ge, \gamma_t}
X_{-\gt(\gamma_t)} \otimes X_{-\gamma_t}
\end{equation}
is a lowest weight vector for $V(\mu+\ge)^*$.
Observe that, by Lemma \ref{Lem5.2.14},
we have $\gamma_t - \ge \in \gD$ for $\gamma_t \in \gD_{\mu+\ge}(\fz(\fn))$.
Then, by (\ref{Eqn8.2.1}), we have 
\begin{align}\label{Eqn8.2.2}
Y^*_l(\tau_2(X)) &=\frac{1}{2}\sum_{
\substack{\gamma_t \in \gD_{\mu+\ge}(\fz(\fn))\\\gamma_j \in \gD(\fz(\fn))}}
N_{-\mu,\gamma_t - \ge}N_{-\ge, \gamma_t}
\kappa(X_{-\gt(\gamma_t)}, \ad(X)^2X_{-\gamma_j})\kappa(X_{-\gamma_t},X_{\gamma_j})
\nonumber\\
&=\frac{1}{2}\sum_{\gamma_t \in \gD_{\mu+\ge}(\fz(\fn))}
N_{-\mu,\gamma_t - \ge}N_{-\ge, \gamma_t}
\kappa(X_{-\gt(\gamma_t)},\ad(X)^2X_{-\gamma_t}).
\end{align}
Write $X= \sum_{\ga \in \gD(\fg(1))}\eta_\ga X_\ga$
and let $\gamma_t \in \gD_{\mu+\ge}(\fz(\fn))$. Then,
\begin{align*}
\kappa(X_{-\gt(\gamma_t)},\ad(X)^2X_{-\gamma_t})
&=\sum_{\ga,\gb\in\gD(\fg(1))}\eta_\ga \eta_\gb 
\kappa(X_{-\gt(\gamma_t)}, [X_\gb,[X_\ga, X_{-\gamma_t}]]) \nonumber\\
&=\sum_{\ga,\gb\in\gD(\fg(1))}\eta_\ga \eta_\gb 
\kappa([X_{-\gt(\gamma_t)},X_\gb],[X_\ga, X_{-\gamma_t}])\nonumber\\
&=\sum_{\substack{\ga \in \gD_{\gamma_t}(\fg(1))\\ \gb\in\gD_{\gt(\gamma_t)}(\fg(1))}}
\eta_\ga \eta_\gb N_{\ga, -\gamma_t}N_{-\gt(\gamma_t),\gb}
\kappa(X_{\gb-\gt(\gamma_t)},X_{\ga-\gamma_t}).
\end{align*}

\noindent Observe that, by Lemma \ref{Lem5.2.16} and Lemma \ref{Lem5.2.191}, 
the sets $\gD_{\gamma_t}(\fg(1))$ 
and $\gD_{\gt(\gamma_t)}(\fg(1))$ are non-empty.
By the normalization (H3) in Section \ref{chap:O1},
 if $\kappa(X_{\gb - \gt(\gamma_t)}, X_{\ga-\gamma_t}) \neq 0$
then $\gb - \gt(\gamma_t) = \gamma_t -\ga$.
Thus $\kappa(X_{\gb - \gt(\gamma_t)}, X_{\ga-\gamma_t})=0$
unless $\gb =(\mu+\ge)-\ga= \gt(\ga)$. Therefore,
\begin{align}\label{Eqn9.2.1}
\kappa(X_{-\gt(\gamma_t)},\ad(X)^2X_{-\gamma_t})
&=\sum_{\substack{\ga \in \gD_{\gamma_t}(\fg(1))\\ \gb\in\gD_{\gt(\gamma_t)}(\fg(1))}}
\eta_\ga \eta_\gb N_{\ga, -\gamma_t}N_{-\gt(\gamma_t),\gb}
\kappa(X_{\gb-\gt(\gamma_t)},X_{\ga-\gamma_t}) \nonumber \\
&= \sum_{\ga \in \gD_{\gamma_t}(\fg(1)) \cap \gD_{\mu+\ge}(\fg(1))}
N_{\ga, -\gamma_t} N_{-\gt(\gamma_t), \gt(\ga)}\eta_{\ga}\eta_{\gt(\ga)} \nonumber\\
&= \sum_{\ga \in \gD_{\mu+\ge}(\fg(1))}
N_{\ga, -\gamma_t} N_{-\gt(\gamma_t), \gt(\ga)}
\kappa(X, X_{-\ga})\kappa(X,X_{-\gt(\ga)}).
\end{align}

\noindent Lemma \ref{Lem5.2.16} is used in line three to show that 
$\gD_{\gamma_t}(\fg(1)) \cap \gD_{\mu+\ge}(\fg(1)) =\gD_{\mu+\ge}(\fg(1))$.
Hence, by (\ref{Eqn8.2.2}) and (\ref{Eqn9.2.1}),
$\ttau_2|_{V(\mu+\ge)^*}(Y^*_l)(X) = Y_l^*(\tau_2(X))$ is
\begin{align*}
&\ttau_2|_{V(\mu+\ge)^*}(Y^*_l)(X) \\
&=\frac{1}{2}
\sum_{\substack{\ga \in \gD_{\mu+\ge}(\fg(1))\\
\gamma_t \in \gD_{\mu+\ge}(\fz(\fn))}}
(N_{-\mu,\gamma_t - \ge}N_{-\ge, \gamma_t})
(N_{\ga, -\gamma_t} N_{-\gt(\gamma_t), \gt(\ga)})
\kappa(X, X_{-\ga})\kappa(X,X_{-\gt(\ga)}).
\end{align*}
Now, via the composition of maps
\begin{equation*}
V(\mu+\ge)^*\stackrel{\ttau_2|_{V(\mu+\ge)^*}}{\to} \Cal{P}^2(\fg(1))
\to \Sym^2(\fg(-1)) \stackrel{\gs}{\hookrightarrow} 
\Cal{U}(\bar \fn) \stackrel{R}{\to} \mathbb{D}(\Cal{L}_{s})^{\bar \fn},
\end{equation*}
for $Y^*_l \in V(\mu+\ge)^*$,
the second-order differential operator 
$\Omega_2(Y^*_l) \in \mathbb{D}(\Cal{L}_{s})^{\bar \fn}$
is given by
\begin{equation*}
\Omega_2(Y^*_l) =\frac{1}{2}
\sum_{\substack{\ga \in \gD_{\mu+\ge}(\fg(1))\\
\gamma_t \in \gD_{\mu+\ge}(\fz(\fn))}}
(N_{-\mu,\gamma_t - \ge}N_{-\ge, \gamma_t})
(N_{\ga, -\gamma_t} N_{-\gt(\gamma_t), \gt(\ga)})
\gs\big( R(X_{-\ga})R(X_{-\gt(\ga)})\big),
\end{equation*}
where $\gs(ab) = (1/2)(ab+ba)$. 
As the special constituent $V(\mu+\ge)$ is assumed to be 
of type 1a or type 2, we have 
$-\ga-\gt(\ga) = -(\mu+\ge) \not \in \gD$.
Thus symmetrization is unnecessary.
Therefore we obtain
\begin{equation} \label{Eqn82}
\Omega_2(Y^*_l) =\frac{1}{2}
\sum_{\substack{\ga \in \gD_{\mu+\ge}(\fg(1))\\
\gamma_t \in \gD_{\mu+\ge}(\fz(\fn))}}
(N_{-\mu,\gamma_t - \ge}N_{-\ge, \gamma_t})
(N_{\ga, -\gamma_t} N_{-\gt(\gamma_t), \gt(\ga)})
R(X_{-\ga})R(X_{-\gt(\ga)}).
\end{equation}

\subsection[Special values]
{Special Values of the $\Omega_2|_{V(\mu+\ge)^*}$ Systems}\label{SS92}

Now we determine the special values of the line bundle $\Cal{L}_{s}$
for which the $\Omega_2|_{V(\mu+\ge)^*}$ system is conformally invariant,
under the assumption that $V(\mu+\ge)$ is a special constituent of type 1a or type 2.

Choose a basis $\{Y^*_1, \ldots, Y^*_n\}$ for $V(\mu+\ge)^*$.
To show that the list of differential operators
$\Omega_2(Y^*_1)$, $\ldots$, $\Omega_2(Y^*_n)$ 
is conformally invariant
on the bundle $\Cal{L}_{s}$, we need to prove that
in $\mathbb{D}(\Cal{L}_{s})^{\bar \fn}$,
\begin{equation}\label{EqnCIS1}
[\pi_{s}(X), \Omega_2(Y^*_i)]
\in \text{span}_{C^{\infty}(\bar{N}_0)}
\{\Omega_2(Y^*_1), \ldots, \Omega_2(Y^*_n) \}
\end{equation}
for all $X\in \fg$ and all $i$. By Proposition \ref{Prop7.2.3},
(\ref{EqnCIS1}) holds if 
\begin{equation}\label{EqnCIS2}
[\pi_{s}(X), \Omega_2(Y^*_i)]_e
\in \text{span}_{\C}
\{\Omega_2(Y^*_1)_e, \cdots, \Omega_2(Y^*_n)_e \}
\end{equation}
holds for all $X \in \fg$ and all $i$.
Here, for $D \in \mathbb{D}(\Cal{L}_{s})$, $D_{\nbar}$
denotes the linear functional $f \mapsto (D\acts f)(\nbar)$ for 
$f \in C^\infty(\bar{N}_0, \C_{\chi^{s}})$. 
We show that a simplification of (\ref{EqnCIS2})
implies (\ref{EqnCIS1}).

\begin{Prop}\label{Prop7.1.2}
Let $V(\mu+\ge)^*$ be the dual module of a special constituent $V(\mu+\ge)$
of $\fl\otimes \fz(\fn)$ with respect to the Killing form. Suppose that the operator 
$\Omega_2|_{V(\mu+\ge)^*}: V(\mu+\ge)^* \to \mathbb{D}(\Cal{L}_{s})^{\bar \fn}$ 
is non-zero. If $X_h$ is a highest weight vector for $\fg(1)$ and if we have
\begin{equation*}
[\pi_s(X_h), \Omega_2(Y^*_l)]_e \in 
\emph{\text{span}}_\C
\{\Omega_2(Y^*_1)_e, \cdots, \Omega_2(Y^*_n)_e \}
\end{equation*}
for a lowest weight vector $Y_l^*$
and a basis $\{Y^*_1, \ldots, Y^*_n\}$ for $V(\mu+\ge)^*$ 
then the $\Omega_2|_{V(\mu+\ge)^*}$ system 
is a conformally invariant system.
\end{Prop}

\begin{proof}
By Remark \ref{Rem2.3.6}, the $\Omega_k|_{V(\mu+\ge)^*}$ system
satisfies the condition (S1) of Definition \ref{Def1.1.3}.
We need to prove that (\ref{EqnCIS2}) holds for all 
$X \in \fg = \bar \fn \oplus \fl \oplus \fn$.
Note that, by definition, we have 
$\Omega_2(Y^*_i) \in \mathbb{D}(\Cal{L}_{s})^{\bar \fn}$. 
Hence (\ref{EqnCIS2}) holds for $X \in \bar \fn$ trivially.
The $L_0$-equivariance of $\Omega_2|_{V(\mu+\ge)^*}$ shows that 
(\ref{EqnCIS2}) holds for $X \in \fl$. 
Furthermore, Lemma \ref{Lem831} established (\ref{EqnCIS2})
when $X \in \fg(1)$. Now we handle the case when $X \in \fz(\fn)$.

If $X \in \fz(\fn)$ then, since $\fz(\fn) = [\fg(1), \fg(1)]$, it is of the form
$X = [X_1, X_2]$ for some $X_1, X_2 \in \fg(1)$. Then, by the Jacobi identity,
we have 
\begin{equation*}
[\spi(X), \Omega_2(Y^*_i)]
=[\spi(X_1), [\spi(X_2), \Omega_2(Y^*_i)]]
-[\spi(X_2), [\spi(X_1), \Omega_2(Y^*_i)]].
\end{equation*}
By (\ref{Eqn7.2.2}), we have $\spi(X_j)_e = 0$ for $j=1,2$.
It follows from Lemma \ref{Lem831} that for $j=1,2$ and all $i$, we have
\begin{equation*}
\big [\spi(X_j), \Omega_2(Y^*)]_e \in 
\text{span}_\C\{\Omega_2(Y^*_1)_e, 
\ldots, \Omega_2(Y^*_n)_e\}.
\end{equation*}
Therefore, by Lemma \ref{Lem832},
\begin{align*}
[\spi(X), \Omega_2(Y^*_i)]_e
&=[\spi(X_1), [\spi(X_2), \Omega_2(Y^*_i)]]_e
-[\spi(X_2), [\spi(X_1), \Omega_2(Y^*_i)]]_e\\
&\in \text{span}_\C\{\Omega_2(Y^*_1)_e,
\ldots, \Omega_2(Y^*_k)_e \}.
\end{align*}
\end{proof}

\begin{Prop}\label{Prop83}
If $\mu$ is the highest weight for $\fg(1)$
and $\ga, \gb \in \gD(\fg(1))$ then
\begin{align*}
&[\pi_s(X_\mu), R(X_{-\ga})R(X_{-\gb})]_e\\
&= R([[X_\mu, X_{-\ga}], X_{-\gb}]])_e
-s\gl_\fq([X_\mu, X_{-\ga}])R(X_{-\gb})_e
-s\gl_\fq([X_\mu, X_{-\gb}])R(X_{-\ga})_e.
\end{align*}
\end{Prop}

\begin{proof}
This simply follows by substituting $Y=X_\mu$, $X_1=X_{-\ga}$,
and $X_2=X_{-\gb}$ in
Proposition \ref{Prop7.2.5}, and evaluating 
at $\nbar = e$.
\end{proof}

If $V(\mu+\ge)$ is a special constituent of $\fl \otimes \fz(\fn)$
of type 1a or type 2 then we write
\begin{equation}\label{Eqn:Constant}
C(\mu,\ge) 
:= \sum_{\gamma_t \in \gD_{\mu+\ge}(\fz(\fn))}
N_{\mu, \ge-\gamma_t} N_{-\mu, \gamma_t-\ge}
N_{\ge,-\gamma_t}N_{-\ge,\gamma_t}.
\end{equation}
By Lemma \ref{Lem5.2.17}, we have $C(\mu,\ge) \neq 0$.

\begin{Thm}\label{Thm8.3.1}
Let $\fg$ be a complex simple Lie algebra and 
$\fq$ be a maximal parabolic subalgebra of
quasi-Heisenberg type, listed in (\ref{Eqn4.0.1}) or (\ref{Eqn4.0.2}).
If $Y^*_l$ is the lowest weight vector defined in (\ref{Eqn8.2.1})
for the dual module $V(\mu+\ge)^*$ of
a special constituent $V(\mu+\ge)$ of type 1a or type 2,
and if $\ga_\fq$ is the simple root that determines $\fq$ 
then the following hold:
\begin{enumerate}
\item[(1)] If $V(\mu+\ge)$ is of type 1a then
\begin{equation}\label{Eqn9.2.3}
[\pi_s(X_\mu), \Omega_2(Y^*_l )]_e 
= -\frac{||\ga_\fq||^2}{2}C(\mu,\ge)(s-s_2)R(X_{-\ge})_e,
\end{equation}
\noindent with $\displaystyle{s_2 = \frac{|\gD_{\mu+\ge}(\fg(1))|}{2}-1}$,
where $|\gD_{\mu+\ge}(\fg(1))|$ is the number of elements in $\gD_{\mu+\ge}(\fg(1))$.\\

\item[(2)] If $V(\mu+\ge)$ is of type 2 then
\begin{equation}\label{Eqn924}
[\pi_s(X_\mu), \Omega_2(Y^*_l )]_e 
= -\frac{||\ga_\fq||^2}{2}C(\mu,\mu)(s+1)R(X_{-\mu})_e.
\end{equation}
\end{enumerate}
\end{Thm}

\begin{proof}
We start by showing that (\ref{Eqn9.2.3}) holds.
It follows from (\ref{Eqn82}) that
\begin{align} \label{EqnCF}
&[\pi_s(X_\mu), \Omega_2(Y^*_l)]_e\nonumber\\
&=\frac{1}{2}\sum_{\substack{
\ga \in\gD_{\mu+\ge}(\fg(1))\\
\gamma_t \in \gD_{\mu+\ge}(\fz(\fn))}}
(N_{-\mu,\gamma_t - \ge}N_{-\ge, \gamma_t})
(N_{\ga, -\gamma_t} N_{-\gt(\gamma_t), \gt(\ga)})
 [\pi_s(X_\mu), R(X_{-\ga})R(X_{-\gt(\ga)})]_e. 
\end{align}

\noindent We use Proposition \ref{Prop83} to compute
$[\pi_s(X_\mu), R(X_{-\ga})R(X_{-\gt(\ga)})]_e$. This is
\begin{align*}\label{Eqn9.4.01}
&[\pi_s(X_\mu), R(X_{-\ga})R(X_{-\gt(\ga)})]_e= \nonumber\\
&R([[X_\mu, X_{-\ga}], X_{-\gt(\ga)}])_e
-s\gl_\fq([X_\mu, X_{-\ga}])R(X_{-\gt(\ga)})_e
-s\gl_\fq([X_\mu, X_{-\gt(\ga)}])R(X_{-\ga})_e.
\end{align*}
We consider the contributions from each term in (\ref{EqnCF}), separately.
Recall here that, as we defined in Section \ref{SS41},
 our parabolic subalgebra $\fq$ is parametrized by 
the simple root $\ga_\fq \in \Pi$ and that $\gl_\fq$ is the fundamental weight for $\ga_\fq$.

First we study the contribution from the second term. It is
\begin{equation*}
T_2=-\frac{s}{2}
\sum_{\substack{
\ga \in\gD_{\mu+\ge}(\fg(1))\\
\gamma_t \in \gD_{\mu+\ge}(\fz(\fn))}}
(N_{-\mu,\gamma_t - \ge}N_{-\ge, \gamma_t})
(N_{\ga, -\gamma_t} N_{-\gt(\gamma_t), \gt(\ga)})
\gl_\fq([X_\mu, X_{-\ga}])R(X_{-\gt(\ga)})_e.
\end{equation*}
As $\fg(1)$ is the 1-eigenspace of $\ad(H_\fq)$ with $H_\fq$ defined in (\ref{Eqn2.1.14}),
the set $\gD(\fg(1))$ is
$\gD(\fg(1))=\big\{\gb \in \gD \; | \; 2\IP{\gl_\fq}{\gb}/||\ga_\fq||^2 =1 \big\}$.
Therefore, by the normalization (H4) in Section \ref{chap:O1},
for $\gb \in \gD(\fg(1))$, we have 
$\gl_\fq(H_\gb) = \IP{\gl_\fq}{\gb} = ||\ga_\fq||^2/2$.
Thus, 
\begin{equation}\label{EqnLambda}
\gl_\fq([X_\mu, X_{-\ga}]) = \frac{||\ga_\fq||^2}{2} \gd_{\ga, \mu}
\end{equation}
with $\gd_{\ga, \mu}$ the Kronecker delta. So 
the contribution from this term is
\begin{align*}
T_2 &= -\frac{s}{2}
\sum_{\substack{
\ga \in\gD_{\mu+\ge}(\fg(1))\\
\gamma_t \in \gD_{\mu+\ge}(\fz(\fn))}}
(N_{-\mu,\gamma_t - \ge}N_{-\ge, \gamma_t})
(N_{\ga, -\gamma_t} N_{-\gt(\gamma_t), \gt(\ga)})
\gl_\fq([X_\mu, X_{-\ga}])R(X_{-\gt(\ga)})_e\\
&=-\frac{s||\ga_\fq||^2}{4}
\sum_{\gamma_t \in \gD_{\mu+\ge}(\fz(\fn))}
(N_{-\mu,\gamma_t - \ge}N_{-\ge, \gamma_t})
(N_{\mu, -\gamma_t} N_{-\gt(\gamma_t), \gt(\mu)})
R(X_{-\gt(\mu)})_e\\
&=-\frac{s||\ga_\fq||^2}{4}
\sum_{\gamma_t \in \gD_{\mu+\ge}(\fz(\fn))}
(N_{-\mu,\gamma_t - \ge}N_{-\ge, \gamma_t})
(N_{\mu, -\gamma_t} N_{\ge,\mu-\gamma_t})
R(X_{-\ge})_e.
\end{align*}
We showed in Lemma \ref{Lem2.6} that
$N_{\mu, -\gamma_t} N_{\ge,\mu-\gamma_t} 
=N_{\mu, \ge-\gamma_t}N_{\ge, -\gamma_t}$.
Hence,
\begin{align*}
T_2 &=
-\frac{s||\ga_\fq||^2}{4}
\sum_{\gamma_t \in \gD_{\mu+\ge}(\fz(\fn))}
(N_{-\mu,\gamma_t - \ge}N_{-\ge, \gamma_t})
(N_{\mu, -\gamma_t} N_{\ge,\mu-\gamma_t})
R(X_{-\ge})_e\\
&=-\frac{s||\ga_\fq||^2}{4}
\sum_{\gamma_t \in \gD_{\mu+\ge}(\fz(\fn))}
(N_{-\mu,\gamma_t - \ge}N_{-\ge, \gamma_t})
(N_{\mu, \ge-\gamma_t}N_{\ge, -\gamma_t})
R(X_{-\ge})_e\\
&=-\frac{s||\ga_\fq||^2}{4}
\sum_{\gamma_t \in \gD_{\mu+\ge}(\fz(\fn))}
(N_{\mu, \ge-\gamma_t} N_{-\mu, \gamma_t-\ge}
N_{\ge,-\gamma_t}N_{-\ge,\gamma_t})
R(X_{-\ge})_e\\
&=-\frac{s||\ga_\fq||^2}{4}C(\mu,\ge)R(X_{-\ge})_e.
\end{align*}
The same argument shows that the contribution from the third term is
\begin{align*}
T_3 &= -\frac{s}{2}
\sum_{\substack{
\ga \in\gD_{\mu+\ge}(\fg(1))\\
\gamma_t \in \gD_{\mu+\ge}(\fz(\fn))}}
(N_{-\mu,\gamma_t - \ge}N_{-\ge, \gamma_t})
(N_{\ga, -\gamma_t} N_{-\gt(\gamma_t), \gt(\ga)})
\gl_\fq([X_\mu, X_{-\gt(\ga)}])R(X_{-\ga})_e\\
&=-\frac{s||\ga_\fq||^2}{4}C(\mu,\ge)R(X_{-\ge})_e.
\end{align*}

Now we consider the contribution from the first term. It is
\begin{equation*}
T_1 =\frac{1}{2}
\sum_{\substack{\ga \in \gD_{\mu+\ge}(\fg(1))\\
\gamma_t \in \gD_{\mu+\ge}(\fz(\fn))}}
(N_{-\mu,\gamma_t - \ge}N_{-\ge, \gamma_t})
(N_{\ga, -\gamma_t} N_{-\gt(\gamma_t), \gt(\ga)})
R([[X_\mu, X_{-\ga}], X_{-\gt(\ga)}])_e.
\end{equation*}
We claim that if $\ga=\ge$ or $\mu$ then $[[X_\mu, X_{-\ga}], X_{-\gt(\ga)}] = 0$,
where $\gt(\ga)$ denotes $\gt(\ga) = (\mu+\ge) - \ga$.
If $\ga=\ge$ then, by Remark \ref{Rem5.4.6},
$[X_\mu, X_{-\ga}]=[X_\mu, X_{-\ge}]=0$.
If $\ga = \mu$ then
\begin{equation*}
[[X_\mu, X_{-\mu}],X_{-\gt(\mu)}]
=[[X_\mu, X_{-\mu}], X_{-\ge}]
=\ge(H_\mu)X_{-\ge}
=0.
\end{equation*}
Note that Remark \ref{Rem5.4.6} is applied to obtain 
$\ge(H_\mu) = \IP{\ge}{\mu} = 0$.
Moreover, by Remark \ref{Rem736}, we have
$\gD_{\mu+\ge}(\fg(1))\backslash \{\mu,\ge\} \neq \emptyset$.
The contribution from $T_1$ is
\begin{align*}
T_1 &=\frac{1}{2}
\sum_{\substack{\ga \in \gD_{\mu+\ge}(\fg(1))\\
\gamma_t \in \gD_{\mu+\ge}(\fz(\fn))}}
(N_{-\mu,\gamma_t - \ge}N_{-\ge, \gamma_t})
(N_{\ga, -\gamma_t} N_{-\gt(\gamma_t), \gt(\ga)})
R([[X_\mu, X_{-\ga}], X_{-\gt(\ga)}])_e\\
&=\frac{1}{2}
\sum_{\substack{\ga \in \gD_{\mu+\ge}(\fg(1))\backslash \{\mu,\ge\}\\
\gamma_t \in \gD_{\mu+\ge}(\fz(\fn))}}
(N_{-\mu,\gamma_t - \ge}N_{-\ge, \gamma_t})
(N_{\ga, -\gamma_t} N_{-\gt(\gamma_t), \gt(\ga)})
R([[X_\mu, X_{-\ga}], X_{-\gt(\ga)}])_e\\
&=\frac{1}{2}
\sum_{\substack{\ga \in \gD_{\mu+\ge}(\fg(1))\backslash \{\mu,\ge\}\\
\gamma_t \in \gD_{\mu+\ge}(\fz(\fn))}}
(N_{-\mu,\gamma_t - \ge}N_{-\ge, \gamma_t})
(N_{\ga, -\gamma_t} N_{-\gt(\gamma_t), \gt(\ga)})
(N_{\mu,-\ga}N_{\mu-\ga, -\gt(\ga)})
R(X_{-\ge})_e\\
&=\frac{1}{2}
\sum_{\substack{\ga \in \gD_{\mu+\ge}(\fg(1))\backslash \{\mu,\ge\}\\
\gamma_t \in \gD_{\mu+\ge}(\fz(\fn))}}
(N_{-\mu,\gamma_t - \ge}N_{-\ge, \gamma_t})
(N_{\ga, -\gamma_t} N_{-\gt(\gamma_t), \gt(\ga)})
(N_{\mu,-\ga}N_{-\gt(\ga),\gt(\mu)})
R(X_{-\ge})_e\\
&=\frac{1}{2}
\sum_{\substack{\ga \in \gD_{\mu+\ge}(\fg(1))\backslash \{\mu,\ge\}\\
\gamma_t \in \gD_{\mu+\ge}(\fz(\fn))}}
(N_{-\mu,\gamma_t - \ge}N_{-\ge, \gamma_t})
(N_{\ga, -\gamma_t} N_{\mu,-\ga}
N_{-\gt(\gamma_t), \gt(\ga)}N_{-\gt(\ga),\gt(\mu)})
R(X_{-\ge})_e.
\end{align*}
Note that, from line three to line four,
we use that $N_{\mu-\ga, -\gt(\ga)}=N_{-\gt(\ga),\gt(\mu)}$,
as $(\mu-\ga) + (-\gt(\ga)) + \gt(\mu) = 0$.
(See the normalization (H6) in Section \ref{chap:O1}.)
By Lemma \ref{Lem5.2.22}, we have
\begin{equation*}
N_{\ga, -\gamma_t}N_{\mu, -\ga}N_{-\gt(\gamma_t), \gt(\ga)}N_{-\gt(\ga), \gt(\mu)}
=N_{\mu, \ge-\gamma_t}N_{\ge, -\gamma_t} \frac{||\ga_\fq||^2}{2}.
\end{equation*}
Therefore,
\begin{align*}
T_1 &=\frac{1}{2}
\sum_{\substack{\ga \in \gD_{\mu+\ge}(\fg(1))\backslash \{\mu,\ge\}\\
\gamma_t \in \gD_{\mu+\ge}(\fz(\fn))}}
(N_{-\mu,\gamma_t - \ge}N_{-\ge, \gamma_t})
(N_{\ga, -\gamma_t} N_{\mu,-\ga}
N_{-\gt(\gamma_t), \gt(\ga)}N_{-\gt(\ga),\gt(\mu)})
R(X_{-\ge})_e\\
&=\frac{||\ga_\fq||^2}{4}
\sum_{\substack{\ga \in \gD_{\mu+\ge}(\fg(1))\backslash \{\mu,\ge\}\\
\gamma_t \in \gD_{\mu+\ge}(\fz(\fn))}}
(N_{-\mu,\gamma_t - \ge}N_{-\ge, \gamma_t})
(N_{\mu, \ge-\gamma_t}N_{\ge, -\gamma_t})
R(X_{-\ge})_e\\
&=\frac{||\ga_\fq||^2}{4}
\sum_{\substack{\ga \in \gD_{\mu+\ge}(\fg(1))\backslash \{\mu,\ge\}\\
\gamma_t \in \gD_{\mu+\ge}(\fz(\fn))}}
(N_{\mu, \ge-\gamma_t} N_{-\mu, \gamma_t-\ge}
N_{\ge,-\gamma_t}N_{-\ge,\gamma_t})
R(X_{-\ge})_e\\
&=\frac{||\ga_\fq||^2}{4}C(\mu,\ge) 
\sum_{\ga \in \gD_{\mu+\ge}(\fg(1)) \backslash \{\mu,\ge\}} R(X_{-\ge})_e\\
&=\frac{||\ga_\fq||^2}{4}C(\mu,\ge) 
(|\gD_{\mu+\ge}(\fg(1))| - 2) R(X_{-\ge})_e.
\end{align*}
Hence, we obtain 
\begin{align*}
[\pi_s(X_\mu), \Omega_2(Y^*_l)]_e
&=T_1 + T_2 + T_3\\
&=-\frac{||\ga_\fq||^2}{2}C(\mu,\ge)
\bigg(s - \bigg( \frac{|\gD_{\mu+\ge}(\fg(1))|}{2}-1 \bigg) \bigg)R(X_{-\ge})_e.
\end{align*}

Now we are going to prove the equation (\ref{Eqn924}).
If $V(\mu+\ge)$ is of type 2 then $\mu+\ge = 2\mu$; 
in particular, $\gt(\mu) = (2\mu)-\mu = \mu$.
By Lemma \ref{Lem5.2.141}, $\gD_{2\mu}(\fg(1)) = \{\mu\}$.
Thus, (\ref{Eqn82}) becomes
\begin{align} \label{Eqn:Type2}
\Omega_2(Y^*_l) 
&=\frac{1}{2}
\sum_{\substack{\ga \in \gD_{2\mu}(\fg(1))\\
\gamma_t \in \gD_{2\mu}(\fz(\fn))}}
(N_{-\mu,\gamma_t - \ge}N_{-\ge, \gamma_t})
(N_{\ga, -\gamma_t} N_{-\gt(\gamma_t), \gt(\ga)})
R(X_{-\ga})R(X_{-\gt(\ga)})\nonumber\\
&=\frac{1}{2}\sum_{\gamma_t \in \gD_{2\mu}(\fz(\fn))}
(N_{-\mu,\gamma_t - \mu}N_{-\mu, \gamma_t})
(N_{\mu, -\gamma_t} N_{-\gt(\gamma_t), \gt(\mu)})
R(X_{-\mu})R(X_{-\gt(\mu)})\nonumber\\
&=\frac{1}{2}\sum_{\gamma_t \in \gD_{2\mu}(\fz(\fn))}
(N_{-\mu,\gamma_t - \mu}N_{-\mu, \gamma_t})
(N_{\mu, -\gamma_t} N_{-\gt(\gamma_t), \mu})R(X_{-\mu})^2.
\end{align}

\noindent Since $(-\gt(\gamma_t)) + \mu + (\mu-\gamma_t) = 0$,
we have $N_{-\gt(\gamma_t), \mu} = N_{\mu, \mu-\gamma_t}$.
Thus,

\begin{align}\label{Eqn:Type21}
&\frac{1}{2}\sum_{\gamma_t \in \gD_{2\mu}(\fz(\fn))}
(N_{-\mu,\gamma_t - \mu}N_{-\mu, \gamma_t})
(N_{\mu, -\gamma_t} N_{-\gt(\gamma_t), \mu})R(X_{-\mu})^2 \nonumber\\
&=\frac{1}{2}\sum_{\gamma_t \in \gD_{2\mu}(\fz(\fn))}
(N_{-\mu,\gamma_t - \mu}N_{-\mu, \gamma_t})
(N_{\mu, -\gamma_t} N_{\mu, \mu-\gamma_t})R(X_{-\mu})^2 \nonumber\\
&=\frac{1}{2}\sum_{\gamma_t \in \gD_{2\mu}(\fz(\fn))}
(N_{\mu, \mu-\gamma_t} N_{-\mu, \gamma_t-\mu}
N_{\mu,-\gamma_t}N_{-\mu,\gamma_t})
R(X_{-\mu})^2\nonumber\\
&=\frac{1}{2}C(\mu,\mu)R(X_{-\mu})^2.
\end{align}
Therefore,
\begin{equation*}
[\pi_s(X_\mu), \Omega_2(Y^*_l)]_e
=\frac{1}{2}C(\mu,\mu)[\pi_s(X_\mu), R(X_{-\mu})^2]_e.
\end{equation*} 
It follows from (\ref{EqnLambda}) that
$\gl_\fq([X_\mu, X_{-\mu}]) = ||\ga_\fq||^2/2$.
Then, by Proposition \ref{Prop83} with $\ga = \gb = \mu$, we have
\begin{align*}
[\pi_s(X_\mu), R(X_{-\mu})^2]_e
&=R([[X_\mu, X_{-\mu}],X_{-\mu}])_e 
-2s\gl_\fq([X_\mu, X_{-\mu}])R(X_{-\mu})_e\\
&=-\mu(H_\mu)R(X_{-\mu})_e - 2s \cdot \frac{||\ga_\fq||^2}{2} R(X_{-\mu})_e\\
&=-(s||\ga_\fq||^2 + ||\mu||^2) R(X_{-\mu})_e.
\end{align*}
Observe that Table \ref{T55} in Subsection \ref{SS540} shows that a special constituent of type 2 occurs
only when $\fq$ is of type $B_n(n)$, type $C_n(i)$ or $F_4(4)$. 
Appendix \ref{chap:Data}
shows that when $\fq$ is of these types,
we have $||\mu||^2 = ||\ga_\fq||^2$. Therefore,
\begin{equation*}
[\pi_s(X_\mu), R(X_{-\mu})^2]_e
=-(s||\ga_\fq||^2 + ||\mu||^2) R(X_{-\mu})_e
=-||\ga_\fq||^2(s+1)R(X_{-\mu})_e.
\end{equation*}
Hence, we obtain
\begin{align*}
[\pi_s(X_\mu), \Omega_2(Y^*_l)]_e
&=\frac{1}{2}C(\mu,\mu)[\pi_s(X_\mu), R(X_{-\mu})^2]_e\\
&=-\frac{||\ga_\fq||^2}{2}C(\mu,\mu)(s+1)R(X_{-\mu})_e.
\end{align*}
\end{proof}

To emphasize the fundamental weight $\gl_\fq$,
we write $\Cal{L}(s\gl_\fq)$ for the line bundle $\Cal{L}_{s}$.
Now, by combining Proposition \ref{Prop7.1.2} and Theorem \ref{Thm8.3.1},
we conclude the following.

\begin{Cor}
Under the same hypotheses in Theorem \ref{Thm8.3.1}, we have:
\begin{enumerate}
\item[(1)] If $V(\mu+\ge)^*$ is of type 1a then 
the $\Omega_2|_{V(\mu+\ge)^*}$ system is conformally invariant
on the line bundle $\Cal{L}(s_2\gl_\fq)$,
where $s_2$ is the constant given in Theorem \ref{Thm8.3.1}.
\item[(2)] If $V(\mu+\ge)^*$ is of type 2 then 
the $\Omega_2|_{V(\mu+\ge)^*}$ system is conformally invariant
on the line bundle $\Cal{L}(-\gl_\fq)$.
\end{enumerate}
\end{Cor}

\begin{proof}
This corollary follows from Proposition \ref{Prop7.1.2} and 
Theorem \ref{Thm8.3.1}.
\end{proof}

As we defined in Definition \ref{Def5.2.7},
we denote by $V(\mu+\geg)$ the special constituent of $\fl \otimes \fz(\fn)$
so that $V(\mu+\geg) \subset \flg \otimes \fz(\fn)$,
and denote by $V(\mu+\geng)$
the special constituent so that $V(\mu+\geng) =\flng \otimes \fz(\fn)$.
See Table \ref{T55} in Subsection \ref{SS540} for the types of 
$V(\mu+\geg)$ and $V(\mu+\geng)$ for each case. 
Table \ref{T93} below summarizes
the line bundles $\Cal{L}(s_0\gl_\fq)$
on which the $\Omega_2$ systems are conformally invariant.
\begin{table}[h]
\caption{Line Bundles with Special Values}
\begin{center}
\begin{tabular}{c|c|c}
\hline
Parabolic subalgebra $\fq$
&$\Omega_2|_{V(\mu+\geg)^*}$ & $\Omega_2|_{V(\mu+\geng)^*}$\\
\hline
$B_n(i), 3\leq i \leq n-2$ & $\Cal{L}\big( (n- i - \frac{1}{2})\gl_i\big)$ & $\Cal{L}(\gl_i)$  \\
$B_n(n-1)$ & $\Cal{L}\big( \frac{1}{2}\gl_{n-1} \big)$ & $?$  \\
$B_n(n)$ & $\Cal{L}(-\gl_n)$ & $-$  \\
$C_n(i), 2 \leq i \leq n-1$  &$?$ & $\Cal{L}(-\gl_i)$ \\
$D_n(i), 3 \leq i \leq n-3$ & $\Cal{L}\big((n - i - 1)\gl_i\big)$& $\Cal{L}(\gl_i)$\\
$E_6(3)$ & $\Cal{L}(\gl_3)$ & $\Cal{L}(2\gl_3)$  \\
$E_6(5)$ & $\Cal{L}(\gl_5)$ & $\Cal{L}(2\gl_5)$  \\
$E_7(2)$ & $\Cal{L}(2\gl_2)$ & $-$ \\
$E_7(6)$ & $\Cal{L}(\gl_6)$ & $\Cal{L}(3\gl_6)$  \\
$E_8(1)$ & $\Cal{L}(3\gl_1)$ & $-$ \\
$F_4(4)$ & $\Cal{L}(-\gl_4)$ & $-$ \\
\hline
\end{tabular}\label{T93}
\end{center}
\end{table}
Here, a dash indicates that there does not exist the special constituent $V(\mu+\geng)$.
When $\fq$ is of type $B_n(n-1)$,
the special constituent $V(\mu+\geng)$ is of type 1b,
and when $\fq$ is of type $C_n(i)$,
the special constituent $V(\mu+\geg)$ is of type 3. 
Therefore, we put a question mark for these cases in the table.

%% file: SS_Apx.tex
\section{Miscellaneous Data}
\label{chap:Data}

This appendix summarizes the miscellenious data for 
the maximal parabolic subalgebras 
$\fq=\fl \oplus \fg(1) \oplus \fz(\fn)$ of quasi-Heisenberg type shown in
(\ref{Eqn4.0.1}) and (\ref{Eqn4.0.2}) in Section \ref{chap:TwoStep}.
For each case we give the deleted Dynkin diagram of $\fq$,
the subgraphs for $\fl_\gamma$ and $\fl_{n\gamma}$,
the simple root $\ga_\gamma$ that is not orthogonal to the highest root for $\fg$,
the highest weights for $\fg(1)$ and $\fz(\fn)$, and the highest roots for
$\fl_\gamma$ and $\fl_{n\gamma}$.
For the definition for the deleted Dynkin diagram
see Subsection \ref{SS23}.
Subsection \ref{SS41} describes about the subspaces $\fg(1)$ and $\fz(\fn)$.
The definitions for the simple ideals $\flg$ and $\flng$ of $\fl$
are given in Subsection \ref{SS42}.
For classical algebras
the sets of roots contributing to 
$\fg(1)$, $\fz(\fn)$, $\fl_{\gamma}$, and $\fl_{n\gamma}$ 
are given in the standard realization of the roots.
 


\vsp


\begin{center}
\S $\mathrm{B}_n(i)$, $3\leq i \leq n-2$ 
\end{center}

\begin{enumerate}

\item The deleted Dynkin diagram:
\begin{equation*}
\xymatrix{\belowwnode{\ga_1}\single[r]&
\belowwnode{\ga_2}\single[r]&\dots\single[r]
&\belowwnode{\ga_{i-1}}\single[r]&\belowcnode{\ga_i}\single[r]&
\belowwnode{\ga_{i+1}}\single[r]&\cdots \single[r]
&\belowwnode{\ga_{n-1}}\rdouble[r]&\belowwnode{\ga_n}}
\end{equation*}

\item The subgraph for $\flg$:
\begin{equation*}
\xymatrix{\belowwnode{\ga_1}\single[r]&
\belowwnode{\ga_2}\single[r]&
\belowwnode{\ga_3}\single[r]&\dots\single[r]
&\belowwnode{\ga_{i-1}}}
\end{equation*}

\item The subgraph for $\flng$:
\begin{equation*}
\xymatrix{
\belowwnode{\ga_{i+1}}\single[r]&\cdots \single[r]
&\belowwnode{\ga_{n-1}}\rdouble[r]&\belowwnode{\ga_n}}
\end{equation*}
\vspace{1pt}

\end{enumerate}

We have $\ga_\gamma = \ga_2$.
The highest weight $\mu$ and the set of roots $\gD(\fg(1))$ for 
$\fg(1)$ are $\mu = \vep_1+\vep_{i+1}$ and 
$\gD(\fg(1)) = \{\vep_j \pm \vep_k \; | \; 1 \leq j \leq i \text{ and } i+1 \leq k \leq n\}
 \cup \{\vep_j \; | \; 1\leq j \leq i \}$.
The highest weight $\gamma$ 
and the set of roots $\gD(\fz(\fn))$ for $\fz(\fn)$
are  $\gamma = \vep_1 + \vep_2$
and $\gD(\fz(\fn)) = \{ \vep_j + \vep_k \; | \; 1\leq j<k\leq i\}$.
The highest root $\xig$ and the set of positive roots 
$\gD^+(\flg)$ for $\flg$ are $\xig = \vep_1 - \vep_i$ and 
$\gD^+(\flg) = \{ \vep_j - \vep_k \; | \; 1\leq j < k\leq i\}$.
The highest root $\xing$ and the set of positive roots
$\gD^+(\flng)$ for $\flng$ are $\xing = \vep_{i+1} + \vep_{i+2}$
and $\gD^+(\flng) = \{ \vep_j \pm \vep_k \; | \; i+1 \leq j < k \leq n \}
\cup \{ \vep_j \; | \; i+1 \leq j \leq n\}$.

\vsp

\begin{center}
\S $\mathrm{B}_n(n-1)$
\end{center}

\begin{enumerate}

\item The deleted Dynkin diagram:
\begin{equation*}
\xymatrix{\belowwnode{\ga_1}\single[r]&
\belowwnode{\ga_2}\single[r]&\dots\single[r]
&\belowwnode{\ga_{n-2}}\single[r]
&\belowcnode{\ga_{n-1}}\rdouble[r]&\belowwnode{\ga_n}}
\end{equation*}

\item The subgraph for $\flg$:
\begin{equation*}
\xymatrix{\belowwnode{\ga_1}\single[r]&
\belowwnode{\ga_2}\single[r]&
\belowwnode{\ga_3}\single[r]&\dots\single[r]
&\belowwnode{\ga_{n-2}}}
\end{equation*}

\item The subgraph for $\flng$:
\begin{equation*}
\xymatrix{\belowwnode{\ga_n}}
\end{equation*}
\vspace{1pt}

\end{enumerate}

We have $\ga_\gamma = \ga_2$.
The highest weight $\mu$ and the set of weights $\gD(\fg(1))$ 
for $\fg(1)$ are $\mu = \vep_1+\vep_{n}$ and 
$\gD(\fg(1)) = \{\vep_j \pm \vep_n \; | \; 1 \leq j \leq n-1 \}
 \cup \{\vep_j \; | \; 1\leq j \leq n-1 \}$.
The highest weight $\gamma$ and the set of weights $\fg(\fz(\fn))$
for $\fz(\fn))$ are  
$\gamma = \vep_1 + \vep_2$ and 
$\gD(\fz(\fn)) = \{ \vep_j + \vep_k \; | \; 1\leq j<k\leq n-1\}$.
The highest root $\xig$ and the set of positive roots 
$\gD^+(\flg)$ for $\flg$ are 
$\xig = \vep_1 - \vep_{n-1}$ and 
$\gD^+(\flg) = \{ \vep_j - \vep_k \; | \; 1\leq j < k\leq n-1\}$.
The highest root $\xing$ and the set of positive roots 
$\gD^+(\flng)$ for $\flng$ are
$\xing = \vep_n$ and 
$\gD^+(\flng) = \{ \vep_n \}$.

\vsp


\begin{center}
\S $\mathrm{B}_n(n)$
\end{center}

\begin{enumerate}

\item The deleted Dynkin diagram:
\begin{equation*}
\xymatrix{\belowwnode{\ga_1}\single[r]&
\belowwnode{\ga_2}\single[r]&\dots\single[r]
&\belowwnode{\ga_{n-1}}\rdouble[r]&\belowcnode{\ga_n}}
\end{equation*}

\item The subgraph for $\flg$:
\begin{equation*}
\xymatrix{\belowwnode{\ga_1}\single[r]&
\belowwnode{\ga_2}\single[r]&
\belowwnode{\ga_3}\single[r]&\dots\single[r]
&\belowwnode{\ga_{n-1}}}
\end{equation*}

\item No subgraph for $\flng$ ($\flng = \{0 \}$)
\vspace{2pt}

\end{enumerate}

We have $\ga_\gamma = \ga_2$.
The highest weight $\mu$ and the set of weights
$\gD(\fg(1))$ are $\mu = \vep_1$ and 
$\gD(\fg(1)) =  \{\vep_j \; | \; 1\leq j \leq n \}$.
The highest weight $\gamma$ and the set of weights
$\gD(\fz(\fn))$ for $\fz(\fn)$ are 
$\gamma = \vep_1 + \vep_2$ and 
$\gD(\fz(\fn)) = \{ \vep_j + \vep_k \; | \; 1\leq j<k\leq n\}$.
The highest root $\xig$ and the set of positive roots 
for $\flg$ are 
$\xig = \vep_1 - \vep_n$ and 
$\gD^+(\flg) = \{ \vep_j - \vep_k \; | \; 1\leq j < k\leq n\}$.

\vsp


\begin{center}
\S $\mathrm{C}_n(i)$, $2\leq i \leq n-1$
\end{center}

\begin{enumerate}

\item The deleted Dynkin diagram:
\begin{equation*}
\xymatrix{\belowwnode{\ga_1}\single[r]&\dots\single[r]
&\belowwnode{\ga_{i-1}}\single[r]&\belowcnode{\ga_i}\single[r]&
\belowwnode{\ga_{i+1}}\single[r]&\cdots \single[r]
&\belowwnode{\ga_{n-1}}\ldouble[r]&\belowwnode{\ga_n}}
\end{equation*}

\item The subgraph for $\flg$:
\begin{equation*}
\xymatrix{\belowwnode{\ga_1}\single[r]&
\belowwnode{\ga_2}\single[r]&
\belowwnode{\ga_3}\single[r]&\dots\single[r]
&\belowwnode{\ga_{i-1}}}
\end{equation*}

\item The subgraph for $\flng$:
\begin{equation*}
\xymatrix{
\belowwnode{\ga_{i+1}}\single[r]&\cdots \single[r]
&\belowwnode{\ga_{n-1}}\ldouble[r]&\belowwnode{\ga_n}}
\end{equation*}
\vspace{1pt}

\end{enumerate}

We have $\ga_\gamma = \ga_1$.
The highest weight $\mu$ and the set of weights 
$\gD(\fg(1))$ for $\fg(1)$ are 
$\mu = \vep_1+\vep_{i+1}$ and 
$\gD(\fg(1)) = \{\vep_j \pm \vep_k \; | \; 1 \leq j \leq i \text{ and } i+1 \leq k \leq n\}$.
The highest weight $\gamma$ and the set of weights $\gD(\fz(\fn))$ 
for $\fz(\fn)$ are 
$\gamma = 2\vep_1$
$\gD(\fz(\fn)) = \{ \vep_j + \vep_k \; | \; 1\leq j<k\leq i\} 
\cup \{2\vep_j \; | \; 1 \leq j \leq i\}$.
The highest root $\xig$ and the set of positive roots 
$\gD^+(\flg)$ for $\flg$ are 
$\xig = \vep_1 - \vep_i$ and 
$\gD^+(\flg) = \{ \vep_j - \vep_k \; | \; 1\leq j < k\leq i\}$
The highest root $\xing$ and the set of positive roots $\gD(\flng)$
for $\flng$ are 
$\xing = 2\vep_{i+1}$ and 
$\gD^+(\flng) = \{ \vep_j \pm \vep_k \; | \; i+1 \leq j < k \leq n \}
\cup \{ 2\vep_j \; | \; i+1 \leq j \leq n\}$.

\vsp


\begin{center}
\S $\mathrm{D}_n(i)$, $3\leq i \leq n-3$ 
\end{center}

\begin{enumerate}

\item The deleted Dynkin diagram:
\begin{equation*}
\xymatrix{&&&&&&&&\abovewnode{\ga_{n-1}}\\
\belowwnode{\ga_1}\single[r]&
\belowwnode{\ga_2}\single[r]&\dots\single[r]&\belowwnode{\ga_{i-1}}\single[r]&
\belowcnode{\ga_i}\single[r]&\belowwnode{\ga_{i+1}}\single[r]&\cdots \single[r]&
\wnode\save []+<20pt,0pt>*\txt{$\ga_{n-2}$} \restore\single[ur]\single[dr]& \\
&&&&&&&&\belowwnode{\ga_n} }
\end{equation*}

\item The subgraph for $\flg$:
\begin{equation*}
\xymatrix{
\belowwnode{\ga_1}\single[r]&
\belowwnode{\ga_2}\single[r]&
\belowwnode{\ga_3}\single[r]&\dots\single[r]
&\belowwnode{\ga_{i-1}}}
\end{equation*}

\item The subgraph for $\flng$:
\begin{equation*}
\xymatrix{&&&\abovewnode{\ga_{n-1}}\\
\belowwnode{\ga_{i+1}}\single[r]&\cdots \single[r]&
\wnode\save []+<20pt,0pt>*\txt{$\ga_{n-2}$} \restore\single[ur]\single[dr]& \\
&&&\belowwnode{\ga_n} }
\end{equation*}
\vspace{1pt}

\end{enumerate}

We have $\ga_\gamma = \ga_2$.
The highest weight $\mu$ and the set of weights $\gD(\fg(1))$
for $\fg(1)$ are 
$\mu = \vep_1+\vep_{i+1}$ and 
$\gD(\fg(1)) = \{\vep_j \pm \vep_k \; | \; 1 \leq j \leq i \text{ and } i+1 \leq k \leq n\}$.
The highest weight $\gamma$ and the set of weights $\gD(\fz(\fn))$ 
for $\fz(\fn))$ are 
 $\gamma = \vep_1 + \vep_2$ and 
$\gD(\fz(\fn)) = \{ \vep_j + \vep_k \; | \; 1\leq j<k\leq i\}$.
The highest root $\xig$ and the set of positive roots $\gD^+(\flg)$
for $\flg$ are 
$\xig = \vep_1 - \vep_i$ and 
$\gD^+(\flg) = \{ \vep_j - \vep_k \; | \; 1\leq j < k\leq i\}$.
The highest root $\xing$ and the set of positive roots $\gD^+(\flng)$
for $\flng$ are 
$\xing = \vep_{i+1} + \vep_{i+2}$
$\gD^+(\flng) = \{ \vep_j \pm \vep_k \; | \; i+1 \leq j < k \leq n \}$.

\vsp


\begin{center}
\S $\mathrm{E}_6(3)$
\end{center}

\begin{enumerate}

\item The deleted Dynkin diagram:
\begin{equation*}
\xymatrix{
&&\abovewnode{\ga_2}\single[d]&&\\
\belowwnode{\ga_1}\single[r]&\belowcnode{\ga_3}\single[r]&\belowwnode{\ga_4}\single[r]
&\belowwnode{\ga_5}\single[r]&\belowwnode{\ga_6} }
\end{equation*}

\item The subgraph for $\flg$:
\begin{equation*}
\xymatrix{
\belowwnode{\ga_2}\single[r]&
\belowwnode{\ga_4}\single[r]&
\belowwnode{\ga_5}\single[r]&
\belowwnode{\ga_{6}}}
\end{equation*}

\item The subgraph for $\flng$:
\begin{equation*}
\xymatrix{
\belowwnode{\ga_{1}}}
\end{equation*}
\vspace{1pt}

\end{enumerate}

We have $\ga_\gamma = \ga_2$.
The highest weight $\mu$ for $\fg(1)$ is 
$\mu = \ga_1+ \ga_2 + \ga_3 + 2\ga_4 +2\ga_5 + \ga_6$.
The highest weight $\gamma$ for $\fz(\fn)$
is $\gamma = \ga_1 + 2\ga_2 + 2\ga_3 + 3\ga_4 + 2\ga_5 + \ga_6$.
The highest root $\xig$ for $\flg$ 
is $\xig = \ga_2 + \ga_4 + \ga_5 + \ga_6$.
The highest root $\xing$ for $\flng$
is $\xing = \ga_1$.

\vsp


\begin{center}
\S $\mathrm{E}_6(5)$
\end{center}

\begin{enumerate}

\item The deleted Dynkin diagram:
\begin{equation*}
\xymatrix{
&&\abovewnode{\ga_2}\single[d]&&\\
\belowwnode{\ga_1}\single[r]&\belowwnode{\ga_3}\single[r]&\belowwnode{\ga_4}\single[r]
&\belowcnode{\ga_5}\single[r]&\belowwnode{\ga_6} }
\end{equation*}

\item The subgraph for $\flg$:
\begin{equation*}
\xymatrix{
\belowwnode{\ga_1}\single[r]&
\belowwnode{\ga_3}\single[r]&
\belowwnode{\ga_4}\single[r]&
\belowwnode{\ga_2}}
\end{equation*}

\item The subgraph for $\flng$:
\begin{equation*}
\xymatrix{
\belowwnode{\ga_6}}
\end{equation*}
\vspace{1pt}

\end{enumerate}

We have $\ga_\gamma = \ga_2$.
The highest weight $\mu$ for $\fg(1)$ is 
$\mu = \ga_1+ \ga_2 + 2\ga_3 + 2\ga_4 +\ga_5 + \ga_6$.
The highest weight $\gamma$ for $\fz(\fn)$ is 
$\gamma = \ga_1 + 2\ga_2 + 2\ga_3 + 3\ga_4 + 2\ga_5 + \ga_6$.
The highest weight $\xig$ for $\flg$ is
$\xig = \ga_1 + \ga_2 + \ga_3 + \ga_4$.
The highest weight $\xing$ for $\flng$ is 
$\xing = \ga_6$.

\vsp


\begin{center}
\S $\mathrm{E}_7(2)$
\end{center}

\begin{enumerate}

\item The deleted Dynkin diagram:
\begin{equation*}
\xymatrix{
&&\abovecnode{\ga_2}\single[d]&&&\\
\belowwnode{\ga_1}\single[r]&\belowwnode{\ga_3}\single[r]&\belowwnode{\ga_4}\single[r]
&\belowwnode{\ga_5}\single[r]&\belowwnode{\ga_6}\single[r]&\belowwnode{\ga_7}}
\end{equation*}
\vspace{2pt}

\item The subgraph for $\flg$:
\begin{equation*}
\xymatrix{
\belowwnode{\ga_1}\single[r]&
\belowwnode{\ga_3}\single[r]&
\belowwnode{\ga_4}\single[r]&
\belowwnode{\ga_5}\single[r]&
\belowwnode{\ga_6}\single[r]&
\belowwnode{\ga_7}}
\end{equation*}

\item No subgraph for $\flng$ ($\flng = \{ 0 \}$)
\vspace{2pt}
\end{enumerate}

We have $\ga_\gamma = \ga_1$.
The highest weight $\mu$ for $\fg(1)$ is
$\mu = \ga_1+ \ga_2 + 2\ga_3 + 3\ga_4 +3\ga_5 + 2\ga_6 + \ga_7$.
The highest weight $\gamma$ for $\fz(\fn)$ is
$\gamma = 2\ga_1 + 2\ga_2 + 3\ga_3 + 4\ga_4 + 3\ga_5 + 2\ga_6 + \ga_7$.
The highest root $\xig$  for $\flg$ is 
$\xig = \ga_1 + \ga_3 + \ga_4 + \ga_5+ \ga_6 + \ga_7$.

\vsp


\begin{center}
\S $\mathrm{E}_7(6)$
\end{center}

\begin{enumerate}

\item The deleted Dynkin diagram:
\begin{equation*}
\xymatrix{
&&\abovewnode{\ga_2}\single[d]&&&\\
\belowwnode{\ga_1}\single[r]&\belowwnode{\ga_3}\single[r]&\belowwnode{\ga_4}\single[r]
&\belowwnode{\ga_5}\single[r]&\belowcnode{\ga_6}\single[r]&\belowwnode{\ga_7}}
\end{equation*}

\item The subgraph for $\flg$:
\begin{equation*}
\xymatrix{&&&\abovewnode{\ga_{2}}\\
\belowwnode{\ga_1}\single[r]&\belowwnode{\ga_{3}}\single[r]&
\wnode\save []+<20pt,0pt>*\txt{$\ga_{4}$} \restore\single[ur]\single[dr]& \\
&&&\belowwnode{\ga_5} }
\end{equation*}

\item The subgraph for $\flng$:
\begin{equation*}
\xymatrix{\belowwnode{\ga_7}}
\end{equation*}
\vspace{1pt}

\end{enumerate}

We have $\ga_\gamma = \ga_1$.
The highest weight $\mu$ for $\fg(1)$ is
$\mu = \ga_1+ 2\ga_2 + 2\ga_3 + 3\ga_4 +2\ga_5 + \ga_6 + \ga_7$.
The highest weight $\gamma$ for $\fz(\fn)$ is 
$\gamma = 2\ga_1 + 2\ga_2 + 3\ga_3 + 4\ga_4 + 3\ga_5 + 2\ga_6 + \ga_7$.
The highest root $\xig$ for $\flg$ is 
$\xig = \ga_1 + \ga_2 + 2\ga_3 + 2\ga_4 + \ga_5$.
The highest root $\xing$ for $\flng$ is 
$\xing = \ga_7$.

\vsp


\begin{center}
\S $\mathrm{E}_8(1)$
\end{center}

\begin{enumerate}

\item The deleted Dynkin diagram:
\begin{equation*}
\xymatrix{
&&\abovewnode{\ga_2}\single[d]&&&&\\
\belowcnode{\ga_1}\single[r]&\belowwnode{\ga_3}\single[r]&\belowwnode{\ga_4}\single[r]
&\belowwnode{\ga_5}\single[r]&\belowwnode{\ga_6}\single[r]&\belowwnode{\ga_7}\single[r]&
\belowwnode{\ga_8} \\}
\end{equation*}

\item The subgraph for $\flg$:
\begin{equation*}
\xymatrix{&&&&&\abovewnode{\ga_{2}}\\
\belowwnode{\ga_8}\single[r]&\belowwnode{\ga_{7}}\single[r]&
\belowwnode{\ga_{6}}\single[r]&\belowwnode{\ga_{5}}\single[r]&
\wnode\save []+<20pt,0pt>*\txt{$\ga_{4}$} \restore\single[ur]\single[dr]& \\
&&&&&\belowwnode{\ga_3} }
\end{equation*}

\item No subgraph for $\flng$ ($\flng = \{0\}$)
\vspace{2pt}

\end{enumerate}

We have $\ga_\gamma = \ga_8$.
The highest weight $\mu$ for $\fg(1)$ is 
$\mu = \ga_1+ 3\ga_2 + 3\ga_3 + 5\ga_4 +4\ga_5 + 3\ga_6 + 2\ga_7+\ga_8$.
The highest weight $\gamma$ for $\fz(\fn)$ is 
$\gamma = 2\ga_1 + 3\ga_2 + 4\ga_3 + 6\ga_4 + 5\ga_5 + 4\ga_6 + 3\ga_7+2\ga_8$.
The highest root $\xig$ for $\flg$ is 
$\xig = \ga_2+\ga_3+2\ga_4+2\ga_5+2\ga_6+2\ga_7+\ga_8$.

\vsp


\begin{center}
\S $\mathrm{F}_4(4)$
\end{center}

\begin{enumerate}

\item The deleted Dynkin diagram:
\begin{equation*}
\xymatrix{
\belowwnode{\ga_1}\single[r]&\belowwnode{\ga_2}\rdouble[r]&
\belowwnode{\ga_3}\single[r]&\belowcnode{\ga_4}}
\end{equation*}

\item The subgraph for $\flg$:
\begin{equation*}
\xymatrix{
\belowwnode{\ga_1}\single[r]&\belowwnode{\ga_2}\rdouble[r]&
\belowwnode{\ga_3}}
\end{equation*}

\item No subgraph for $\flng$ ($\flng = \{0\}$)
\vspace{2pt}

\end{enumerate}

We have $\ga_\gamma = \ga_1$.
The highest weight $\mu$ for $\fg(1)$ is 
$\mu = \ga_1+ 2\ga_2 + 3\ga_3 + \ga_4$.
The highest weight $\gamma$ for $\fz(\fn)$ is 
$\gamma = 2\ga_1 + 3\ga_2 + 4\ga_3 + 2\ga_4$.
The highest root for $\xig$ for $\flg$ is
$\xig = \ga_1 + 2\ga_2 + 2\ga_3$.

\vskip 0.2in
\noindent \textbf{Acknowledgments.}
This work is part of author's Ph.D. thesis at Oklahoma State University.
The author would like to thank his advisor, Leticia Barchini, 
for her generous guidance. 
He would also like to thank Anthony Kable and  
Roger Zierau for their valuable comments on this work.